\documentclass[11pt]{amsart}

\usepackage[utf8]{inputenc}
\usepackage[T1]{fontenc}
\usepackage{amsmath, amssymb, amsthm}
\usepackage{bm, mathtools}
\usepackage{newtxtext, newtxmath}
\usepackage[normalem]{ulem}
\usepackage[babel=true]{csquotes}
\usepackage[english]{babel}
\usepackage[margin=3.9cm]{geometry}

\usepackage{graphicx}
\usepackage{xcolor}
\usepackage{tikz, tkz-euclide}
\usepackage{pgfplots}
\usepackage{graphpap}
\usepackage[outline]{contour}
\usepackage[mode=buildnew]{standalone}
\usepackage{blindtext}

\usepackage{enumitem}
\usepackage{array,multirow}

\usepackage{upref}
\usepackage{hyperref}
\usepackage{amsrefs}

\usepgfplotslibrary{fillbetween}

\hypersetup{colorlinks, breaklinks, linkcolor=black,citecolor=black,
filecolor=black,urlcolor=black}

\mathtoolsset{showonlyrefs,showmanualtags}

\newcommand{\define}[1]{\emph{#1}}

\definecolor{mygray}{RGB}{205,205,205}

\theoremstyle{plain}
\newtheorem{theorem}{Theorem}[section]
\newtheorem{corollary}[theorem]{Corollary}
\newtheorem{lemma}[theorem]{Lemma}
\newtheorem{proposition}[theorem]{Proposition}

\theoremstyle{definition}
\newtheorem{definition}[theorem]{Definition}
\newtheorem{example}[theorem]{Example}

\theoremstyle{remark}
\newtheorem{remark}{Remark}[section]

\newcommand{\R}{\mathbb{R}}
\newcommand{\C}{\mathbb{C}}
\newcommand{\K}{\mathbb{K}}

\newcommand{\HH}{\mathbb{H}}

\renewcommand{\SS}{\mathbb{S}}

\newcommand{\CP}{\C\mathrm{P}}
\newcommand{\chat}{\hat{\C}}

\DeclareMathOperator{\Sym}{Sym}
\DeclareMathOperator{\Herm}{Herm}

\DeclareMathOperator{\SO}{SO}

\DeclareMathOperator{\SL}{SL}
\DeclareMathOperator{\PSL}{PSL}

\DeclareMathOperator*{\argmin}{argmin}
\DeclareMathOperator{\arcosh}{arcosh}

\DeclareMathOperator{\interior}{int}

\DeclareMathOperator{\tr}{tr}
\DeclareMathOperator{\dist}{dist}
\DeclareMathOperator{\tandist}{td}
\DeclareMathOperator{\conv}{conv}

\DeclareMathOperator{\poly}{poly}

\DeclareMathOperator{\lin}{span}

\DeclareMathOperator{\trunc}{trunc}

\newcommand{\upT}{\mathsf{T}}

\newcommand{\ip}[2]{\langle #1, #2 \rangle}

\newcommand{\config}{\mathfrak{D}}

\newcommand{\eq}{\,=\,}
\newcommand{\plus}{\,+\,}
\newcommand{\minus}{\,-\,}
\newcommand{\ee}{\mathrm{e}}
\newcommand{\ii}{\bm{i}}

\newcommand{\surf}{\Sigma}
\newcommand{\csurf}{\bar{\surf}}
\newcommand{\Fgroup}{\Gamma}
\newcommand{\tri}{T}
\newcommand{\verts}{V}

\newcommand{\len}{\ell}

\begin{document}

\title[Canonical tessellations of decorated hyperbolic surfaces]{
   Canonical tessellations \\ of decorated hyperbolic surfaces
}

\author{Carl O.\ R.\ Lutz}
\address{Technische Universit\"at Berlin, Institut f\"ur Mathematik,
   Str.\ des 17.\ Juni 136, 10623 Berlin, Germany}
\email{clutz@math.tu-berlin.de}

\date{\today}



\begin{abstract}
   A decoration of a hyperbolic surface of finite type is a choice of
   circle, horocycle or hypercycle about each cone-point, cusp or
   flare of the surface, respectively. In this article we show that
   a decoration induces a unique canonical tessellation and dual decomposition of
   the underlying surface. They are analogues of the weighted Delaunay
   tessellation and Voronoi decomposition in the Euclidean plane.
   We develop a characterisation in terms of the hyperbolic geometric equivalents
   of \textsc{Delaunay}'s empty-discs and \textsc{Laguerre}'s tangent-distance,
   also known as power-distance. Furthermore, the relation between the
   tessellations and convex hulls in Minkowski space is presented, generalising
   the Epstein-Penner convex hull construction. This relation allows us to extend
   \textsc{Weeks}' flip algorithm to the case of decorated finite type hyperbolic
   surfaces. Finally, we give a simple description of the configuration space of
   decorations and show that any fixed hyperbolic surface only admits a finite
   number of combinatorially different canonical tessellations.
\end{abstract}

\maketitle

\section{Introduction}
It is commonly known that one can associate to a finite set of points $\verts$ in
the Euclidean plane two dual combinatorial structures: the Delaunay tessellation
and the Voronoi decomposition. The former is a tessellation of $\conv(\verts)$ with
vertex set $\verts$ such that faces are given as the convex hulls of vertices on the
boundaries of \emph{empty discs}, i.e., discs which contain no point of $\verts$
in their interiors. The latter is a decomposition of the Euclidean plane into regions,
each consisting of all points closest to one of the points in $\verts$, respectively.

The study of Voronoi decompositions has a long history. It dates back at least to
\textsc{L.~Dirichlet's} analysis of fundamental domains of $2$- and $3$-dimensional
Euclidean lattices in 1850 \cite{Dirichlet50}. The analogous considerations of
\textsc{H.~Poincar{\'e}} for Fuchsian groups acting on the hyperbolic plane are
almost as old \cite{Poincare82}. The first general results in $N$-dimensions
are due to \textsc{G.~Vorono{\"{\i}}} \cite{Voronoi08}. His student
\textsc{B.~Delaunay} introduced the dual approach via \emph{empty spheres}
\cite{Delaunay34}. The classical motivations for studying these tessellations and there
generalisations, i.e., \define{weighed Delaunay tessellations} and
\define{weighted Voronoi decompositions}, are diverse. They range form number-theoretic
considerations in the reduction theory of quadratic forms
\citelist{\cite{Dirichlet50}\cite{Voronoi08}\cite{Delaunay34}}, over questions in
surface theory \cite{Poincare82} to statics and kinematics of frameworks and
the analysis of combinatorics of convex polyhedra
\citelist{\cite{Izmestiev19}\cite{EricksonLin21}}.

Since the 1980s canonical tessellations, i.e., weighted Delaunay tessellations, of
hyperbolic cusp surfaces play an important role in the analysis of Teichm\"uller and
moduli spaces \citelist{\cite{Harer88}\cite{Penner12}}. One approach to these
tessellations considers level sets of intrinsic
\emph{\enquote{polar coordinates}} which can be
introduced about each cusp of the surface. It is due to ideas of \textsc{W.~Thurston}
and was
worked out by \textsc{B.~Bowditch}, \textsc{D.~Epstein} and \textsc{L.~Mosher}
\citelist{\cite{Harer88}*{Chapter 2, \S3}\cite{BowditchEpstein88}}. Another
approach introduced by \textsc{D.~Epstein} and \textsc{R.~Penner} uses affine lifts
to convex hulls in Minkowski-space \cite{EpsteinPenner88}, the
\emph{Epstein-Penner convex hull construction}. Subsequently, the latter approach
was successfully applied to closed hyperbolic surfaces with a finite
number of distinguished points \cite{NaatanenPenner91}, compact hyperbolic surfaces
with boundaries \cite{Ushijima99} and projective manifolds with radial ends
\cite{CooperLong13}. The convex hull construction gives rise to an explicit method to
compute the weighted Delaunay tessellations, the \emph{flip algorithm}. It was
originally explored by \textsc{J.~Weeks} for decorated hyperbolic cusp surfaces
\cite{Weeks93} and recently extended to projective surfaces by \textsc{S.~Tillmann}
and \textsc{S.~Wong} \cite{TillmannWong16}.

Weighted Delaunay tessellations of closed hyperbolic surfaces and surfaces with
singular Euclidean structure (\emph{PL-surfaces}) are closely related to $3$-dimensional
hyperbolic polyhedra. They can be characterised as critical points of certain
\emph{\enquote{energy-functionals}} related to the volumes of these polyhedra
\citelist{\cite{Leibon02}\cite{Rivin94}\cite{Springborn08}}. Conversely, weighted
Delaunay tessellations have proven to be a valuable tool for finding variational methods
for the polyhedral realisation of hyperbolic cusp surfaces
\citelist{\cite{Springborn20}\cite{Prosanov20}}
(see \cite{Fillastre08} for an overview of the general problem).

Finally, weighted Delaunay tessellations and Voronoi decompositions are interesting
for applications (see
\citelist{\cite{BogdanovEtAl14}\cite{Edelsbrunner00}\cite{ImaiEtAl85}} and references
therein). In particular, there is a flip algorithm to compute Delaunay triangulations
of PL-surfaces \cite{IndermitteEtAl01}. These can be used to define a discrete notion
of intrinsic Laplace-Beltrami operator \cite{BobenkoSpringborn07}. Furthermore,
there are surprising connections between discrete conformal equivalence and decorated
hyperbolic cusp surfaces \cite{BobenkoEtAl15}. Their weighted Delaunay tessellations
play an important role in the theoretical proof of the discrete uniformization
theorem \citelist{\cite{GuEtAl18a}\cite{GuEtAl18b}} (see also \cite{Springborn20}).
Additionally, they are necessary to ensure optimal results when computing discrete
conformal maps \cite{GillespieEtAl21}.

\subsection{Statements of the article}
In this article we are going to define and analyse weighted Delaunay tessellations
and Voronoi decompositions on decorated hyperbolic surfaces of finite type. Our
constructions contain the results about tessellations obtained in
\citelist{\cite{BowditchEpstein88}\cite{EpsteinPenner88}\cite{NaatanenPenner91}
\cite{Ushijima99}\cite{Weeks93}} as special cases. Another important class of examples
are hyperbolic surfaces corresponding to the quotients of the hyperbolic plane by
a finitely generated, non-elementary Fuchsian groups (see Figure
\ref{fig:surface_from_group} and Example \ref{example:fuchsian_group}).

\begin{figure}[t]
   \includegraphics[width=0.4\textwidth]{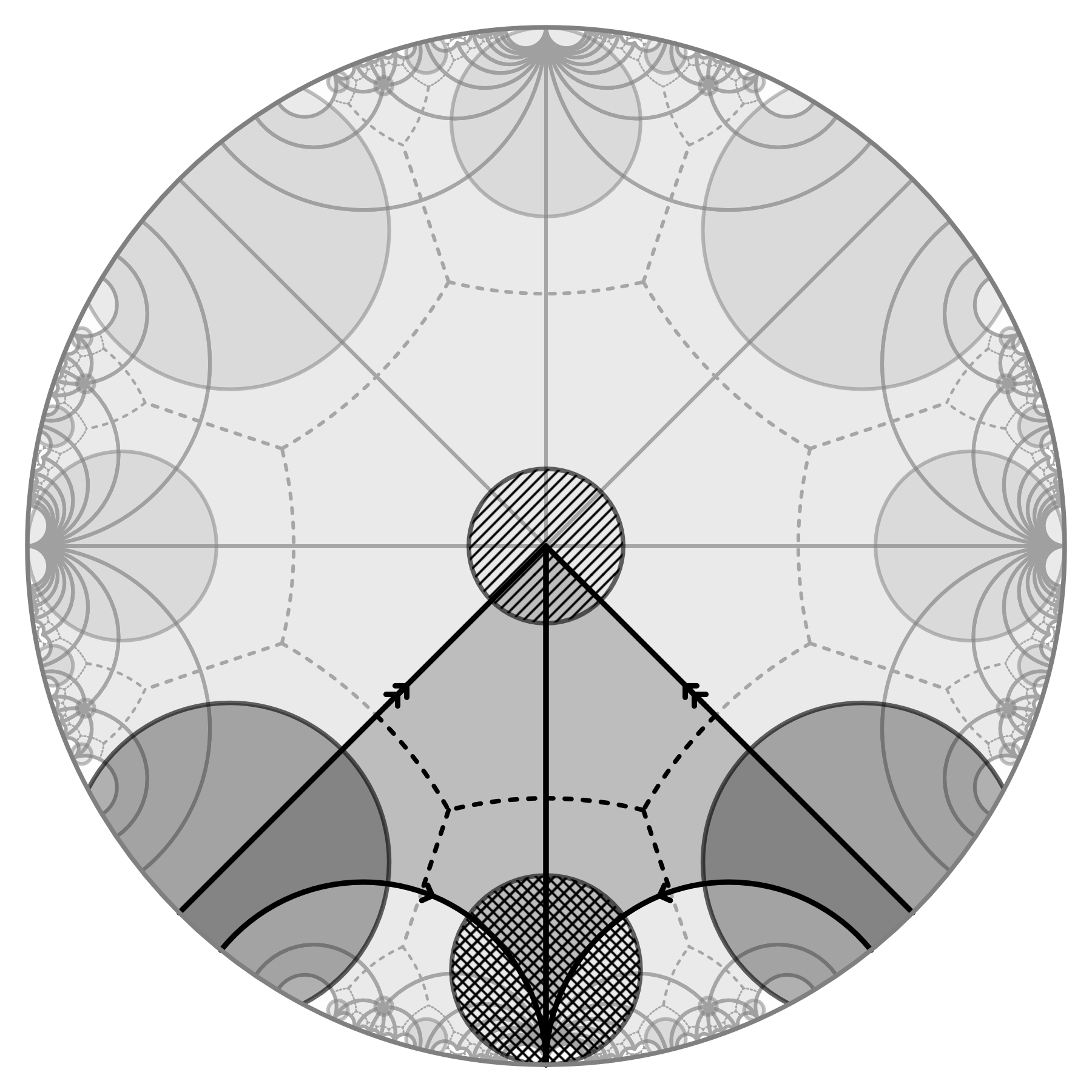}
   \caption{A tessellation of the hyperbolic plane corresponding to a
      weighted Delaunay triangulation (solid lines) and its dual weighted Voronoi
      decomposition (dashed lines) of a decorated hyperbolic surface. The surface can
      be obtained by identifying the boundary edges of a fundamental domain
      (darker shaded) as indicated by arrows. It is homeomorphic to a twice punctured
      sphere and has a cone-point (striped), cusp (chequered) and flare (solid). The
      identifications correspond to the action of a Fuchsian group.
   }
   \label{fig:surface_from_group}
\end{figure}

Informally speaking, a hyperbolic surface of finite type consists of a surface
$\surf$ which is homeomorphic to a closed orientable surface $\csurf$ minus a finite
set of points $\verts_0\cup\verts_1$ endowed with a complete hyperbolic path-metric
$\dist_{\surf}$ which possesses a finite number of cone-points $\verts_{-1}\subset\surf$,
ends of finite area $\verts_{0}$ (\define{cusps}) and infinite area ends
$\verts_{1}$ (\define{flares}). Each \emph{\enquote{point}} in
$\verts\coloneq\verts_{-1}\cup\verts_{0}\cup\verts_{1}$ is decorated with a hyperbolic
cycle of the respective type, i.e., a circle, horocycle or hypercycle.

In Theorem \ref{theorem:weighted_Delaunay_tessellations} we prove the existence and
uniqueness of weighted Delaunay tessellations. They are defined using
\emph{properly immersed discs}, the analogue of \textsc{B.\ Delaunay's} empty discs
for decorated hyperbolic surfaces. The corresponding results about weighted Voronoi
decompositions are contained in Theorem \ref{theorem:properties_laguerre_voronoi}.
Their $2$-cells consist of all points of $\surf$ closest to one of the decorated
vertices in $\verts$ measured in the \emph{modified tangent distance}, respectively.
The distance is an analogue of \textsc{E.\ Laguerre's} \emph{tangent-distance}, also
known as \emph{\enquote{power distance}} in the Euclidean plane (see \cite{Blaschke29}).
Our construction generalises the approach of \textsc{L.\ Mosher}, \textsc{B.\ Bowditch}
and \textsc{D.\ Epstein}. Theorem
\ref{theorem:weighted_Delaunay_tessellations_revisited} reveals the
connections between weighted Delaunay tessellations and Voronoi decompositions.
The flip algorithm is discussed in Theorem \ref{theorem:flip_algorithm}.
We prove that for all proper decorations (see section
\ref{sec:laguerre_voronoi_decompositions}) a weighted Delaunay triangulation
of $\surf$ can be computed from an arbitrary geodesic triangulation in finite time.
All steps to actually implement the algorithm are discussed. For the analysis of the
flip algorithm we introduce \emph{support functions}, i.e., the local
\emph{\enquote{scaling offsets}} from the one-sheeted hyperboloid, on the surface
(compare to \cite{Fillastre13}). If the hyperbolic surface corresponds to a Fuchsian
group, the support function associated to a weighted Delaunay triangulation
induces a convex hull in Minkowski-space (Corollary \ref{cor:convex_hull_construction}).
This is a direct generalisation of the
Epstein-Penner convex hull construction to all finitely generated, non-elementary
Fuchsian groups. Finally, Theorem \ref{theorem:configuration_space} identifies
the configuration space of proper decorations as a convex, connected subset of
$\R_{>0}^{\verts}$ and discusses the dependence of the combinatorics of weighted
Delaunay tessellations on the decoration. In particular, we prove a generalisation of
\emph{\enquote{Akiyoshi's compactification}}
\citelist{\cite{Akiyoshi00}\cite{GuEtAl18a}*{Appendix}}, that is, we prove that any
fixed hyperbolic surface of finite type only admits a finite number of combinatorially
different weighted Delaunay tessellations. Moreover, we show that
weighted Delaunay tessellations induce a decomposition of the configuration space
into convex polyhedral cones. This is an analogue of the classical secondary fan
associated to a finite number of points in the Euclidean plane
\citelist{\cite{GelfandEtAl94}*{Chapter 7}\cite{DeLoeraEtAl10}*{Chapter 5}}.

We highlight that the main methods of this article, namely, properly immersed discs,
tangent-distances and support functions, are intrinsic in nature. That is, they only
depend on the metric of the surface and the given decoration. In contrast most other
approaches like the classical Epstein-Penner convex hull construction or the
\emph{\enquote{empty discs}} utilised in \cite{DespreEtAl20} rely on the existence
of (metric) covers of the surface by the hyperbolic plane. Notable exceptions are the
approach by \textsc{B.~Bowditch}, \textsc{D.~Epstein} and \textsc{L.~Mosher} for
hyperbolic cusp surfaces and the \emph{\enquote{empty immersed discs}}
\textsc{A.~Bobenko} and \textsc{B.~Springborn} considered for PL-surfaces. It is
important to notice that for the objects of interest of this article, i.e., canonical
tessellations of finite type hyperbolic surfaces, (metric) covers by the
hyperbolic plane do in general not exist. Thus a classical Epstein-Penner convex
hull construction is not feasible.

\subsection{Outline of the article}
We begin our expositions with an introduction to the local geometry of hyperbolic
cycles and their associated polygons in section \ref{sec:local}. The main aim is
to derive relations between hyperbolic cycles, hyperbolic polygons and the
hyperbolic analogue of \textsc{E.\ Laguerre's} tangent-distance. This will lead us
to a generalisation of \textsc{J.\ Weeks}' tilt formula \cite{Weeks93,SakumaWeeks95}.

In the next section \ref{sec:global} we are turning our attention to hyperbolic
surfaces of finite type. After collecting some properties of these surfaces we
introduce and analyse weighted Delaunay tessellations and Voronoi decompositions.
We close this section with an analysis of the flip algorithm and a generalisation of
the Epstein-Penner convex hull construction to decorated hyperbolic surfaces.

The last section \ref{sec:config_space} is about characterising the configuration
space of decorations of a fixed hyperbolic surface of finite type. Furthermore, we
consider some explicit examples.

\subsection{Open questions}
Using the convex hull construction, \textsc{R.~Penner} introduced a mapping class
group invariant cell decomposition of the decorated Teichm\"uller space of hyperbolic
cusp surfaces \cites{Penner87, Penner12}. \textsc{A.~Ushijima} presented
a similar construction for Teichm\"uller spaces of compact surfaces with boundary
\cite{Ushijima99}. But his constructions do not cover decorations of these
surfaces.  Actually, in light of this article, we see that \textsc{A.~Ushijima}
implicitly prescribes a constant radius decoration for all surfaces. It remains
the question whether his decompositions extend to decorated Teich\"uller spaces
exhibiting equal properties to the case of hyperbolic cusp surfaces.

Independently of these questions the structure of the configuration space of
decorations for a fixed surface remains interesting on its own.
\textsc{M.~Joswig}, \textsc{R.~L\"owe} and \textsc{B.~Springborn} showed
that the notions of secondary fan and polyhedron can be defined for
decorated hyperbolic cusp surfaces \cite{JoswigEtAl20}. Our Theorem
\ref{theorem:configuration_space} provides the existence of secondary fans for
the more general class of finite type hyperbolic surfaces. Their secondary polyhedra
still remain to be investigated.

The algorithmic aspects of finding weighted Delaunay tessellations on hyperbolic
surfaces, or PL-surfaces, are still little explored. To date,
\textsc{J.~Weeks'} flip algorithm and its generalisations, presents the only general
means to compute such tessellations known to the author. Except for correctness and
termination in the case of surfaces there is not much known about the flip algorithm.
Recently, \textsc{V.~Despr\'e}, \textsc{J.-M.~Schlenker} and \textsc{M.~Teillaud}
found upper bounds for the run-time in the case of undecorated compact hyperbolic
surfaces with a finite number of distinguished points \cite{DespreEtAl20}. For
dimensions $\geq3$ even an algorithm which is guaranteed to terminate with a correct
tessellation is an open question.

Another question is characterising all decorations of a fixed hyperbolic surface
whose weighted Delaunay tessellation can be computed via the flip algorithm.
Our Theorem \ref{theorem:flip_algorithm} guarantees that this is possible for all
decorations of a hyperbolic surfaces without cone points, i.e.,
$\verts_{-1}=\emptyset$. Should cone points exist we only consider proper decorations
(see section \ref{sec:laguerre_voronoi_decompositions}). Experiments for a finite set
of points on a compact hyperbolic surface indicate that the flip-algorithm is still
valid for (some) non-proper decorations. Indeed we conjecture that our configuration
space of proper decorations is optimal iff all cone-angles at vertices in
$\verts_{-1}$ are $\leq\pi$.

\subsection{Acknowledgements}
This work was supported by DFG via SFB-TRR 109:
\enquote{Discretization in Geometry and Dynamics}. The author wishes to thank his
doctoral advisor \textsc{Alexander Bobenko} for his encouragement and support,
\textsc{Boris Springborn} for always having an open door and \textsc{Fabian Bittl}
for interesting discussions.

\section{The local geometry}\label{sec:local}
In this section we consider the geometry of hyperbolic cycles and their
associated hyperbolic polygons in the hyperbolic plane. We approach this topic
from a M\"obius geometric point of view. Apart from some elementary facts
about hyperbolic geometry our expositions are self contained.

The interested reader can find a classical account of M\"obius geometry in
\cite{Blaschke29}. In-depth discussions of its relations to complex numbers
and matrix-groups are given in \citelist{\cite{Yaglom68}\cite{Benz73}}. A
modern introduction to M\"obius geometry and its connections to hyperbolic
geometry is given in \cite{BobenkoEtAl21}. More informations about the
differential aspects of M\"obius geometry can be found in \cite{Cecil92}.

For comprehensive overviews of hyperbolic geometry and its different models we refer
the reader to \citelist{\cite{CannonEtAl97}\cite{Ratcliffe94}}. If the reader wishes
to get a better intuition of hyperbolic geometry we recommend \cite{Thurston97}.

\subsection{M\"obius circles and hyperbolic cycles}
The complex plane $\C$ extended by a single point $\infty$ is called
the \define{M\"obius plane} $\chat$. Its automorphisms are given
by \define{(orientation preserving) M\"{o}bius transformations}, i.e., complex
linear fractional transformations
\begin{equation}
   z \mapsto \frac{az + b}{cz + d},
\end{equation}
where $ad - bc \neq 0$. They form a group isomorphic to $\PSL(2;\C)$ as $\chat$
is equivalent to the complex projective line $\CP^1$. M\"{o}bius transformations
act bijectively on the set of quadratic equations of the form
\begin{equation}
   \label{eq:conformal_circles}
   az\bar{z} - \bar{b}z - b\bar{z} + c \eq 0,
\end{equation}
with non-simultaneously vanishing $a,c\in\R$, $b\in\C$. The solution set of
such a quadratic equation, if non-empty, is a
point, line or circle in $\C$. It is easy to see that for points
$b\bar{b} - ac = 0$ while $b\bar{b} - ac > 0$ for lines and circles.
Furthermore, M\"{o}bius transformations preserve these relations.
Therefore, M\"obius transformations act bijectively on the set of lines and
circles of the complex plane, the \define{M\"obius circles} of $\chat$.

The left hand side of the quadratic equation \eqref{eq:conformal_circles} can be
uniquely identified with an Hermitian matrix, namely
\[
   \begin{pmatrix} a & b\\ \bar{b} & c \end{pmatrix} \in \Herm(2).
\]
Endowed with the bilinear form
\begin{equation}
   \label{eq:bilinear_form}
   \ip{X}{Y}
   \,\coloneq\,
   -\frac{1}{2}
   \tr\left(
   X\left(\begin{smallmatrix}0&-\ii\\ \ii&0\end{smallmatrix}\right)
   Y^{\upT}\left(\begin{smallmatrix}0&- \ii\\ \ii&0\end{smallmatrix}\right)\right)
\end{equation}
the Hermitian matrices constitute an inner product space of signature $(3, 1)$.
More precisely, parametrising $X\in\Herm(2)$ by
\[
   X
   = \begin{pmatrix}
      x_0 + x_3 & x_1 + \ii x_2\\
      x_1 - \ii x_2 & x_0 - x_3
   \end{pmatrix}
\]
we see that $\ip{X}{Y} = -x_0y_0 + x_1y_1 + x_2y_2 + x_3y_3$. The
identity component $\SO^{+}(3,1)$ of its isometry group is isomorphic to
$\SL(2;\C)$, the isomorphism $\varphi\colon\SO^{+}(3,1)\to\SL(2;\C)$ being
defined by
\begin{equation}
   f(X) \eq \bar{\varphi}_f^{\upT}\,X\,\varphi_f.
\end{equation}
Utilising the bilinear form \eqref{eq:bilinear_form} the collection of M\"obius
circles and points in $\chat$ can be identified up to scaling with the elements of
$\{|X|^2 > 0\}$ and $\{|X|^2=0\}$, respectively. Here $|X|^2\coloneq\ip{X}{X}$.
We will not further distinguish between elements of $\Herm(2)$ and their
M\"obius-geometric counterparts if the scaling ambiguity poses no problem for
the presented constructions.

\begin{lemma}
   \label{lemma:conformal_circles_orthogonality}
   Two M\"obius circles $C_1$ and $C_2$ intersect orthogonally iff
   $\ip{C_1}{C_2} \eq 0$.
\end{lemma}
\begin{proof}
   Using a M\"{o}bius transformation we can assume that the first M\"obius
   circle is the $x$-axis and the second one intersects it in $\pm1$. Thus,
   they intersect orthogonally iff the second M\"obius circle is the
   unit circle centred at the origin. The claim follows by direct computation.
\end{proof}

It is clear that any two M\"obius circles $C_1$ and $C_2$ span a $2$-dimensional
subspace of $\Herm(2)$ and thus induce a $1$-parameter family of M\"obius circles.
It is call the \define{pencil (of circles)} spanned by $C_1$ and $C_2$.
The non-degeneracy of $\ip{\cdot}{\cdot}$ grants that there is a unique
complementary subspace in $\Herm(2)$ such that the two $1$-parameter families
of M\"obius circles are mutually orthogonal. They are said to be \define{dual}
to each other (see Figure \ref{fig:pencils}).

\begin{figure}[t]
   \begin{minipage}{0.48\textwidth}
      \begin{tikzpicture}
         \clip (-2.2,-1.9) rectangle (4.2,1.9);
         \tkzDefPoint(0,0){A};
         \tkzDefPoint(2,0){B};

         \foreach \K in {1.2, 1.5, 2, 2.5, 3.4}{
            \tkzDefCircle[apollonius,K=\K](A,B);
            \tkzGetPoint{C}; \tkzGetLength{r};
            \tkzDrawCircle[R](C,\r);

            \tkzDefCircle[apollonius,K=\K](B,A);
            \tkzGetPoint{C}; \tkzGetLength{r};
            \tkzDrawCircle[R](C,\r);
         }

         \foreach \y in {0.5, 1.5, 5}{
            \tkzDefPoint(1,\y){C};
            \tkzDrawCircle[style=dotted](C,A);
            \tkzDefPoint(1,-\y){C};
            \tkzDrawCircle[style=dotted](C,A);
         }

         \tkzDrawPoints(A,B);
         \tkzDefLine[mediator](A,B);
         \tkzGetPoints{C}{D};
         \tkzDrawLine(C,D);

         \tkzDrawLine[add=1.1 and 1.1,style=dotted](A,B);
         \tkzDefPoint(1,0){C}; \tkzDrawCircle[style=dotted](C,A);
      \end{tikzpicture}
   \end{minipage}
   \quad
   \begin{minipage}{0.48\textwidth}
      \begin{tikzpicture}
         \clip (-3.2,-1.9) rectangle (3.2,1.9);
         \tkzDefPoint(-1,0){A}; \tkzDefPoint(0,0){M}; \tkzDefPoint(1,0){B};

         \foreach \t in {0.5, 1, 1.5, 5}{
            \tkzDefPoint(0,\t){C};
            \tkzDrawCircle[style=dotted](C,M);
            \tkzDefPoint(0,-\t){C};
            \tkzDrawCircle[style=dotted](C,M);

            \tkzDefPoint(\t, 0){C};
            \tkzDrawCircle(C,M);
            \tkzDefPoint(-\t, 0){C};
            \tkzDrawCircle(C,M);
         }
         \tkzDrawLine[add=1.1 and 1.1,style=dotted](A,B);

         \tkzDrawPoint(M);
         \tkzDefLine[mediator](A,B);
         \tkzGetPoints{C}{D};
         \tkzDrawLine(C,D);
      \end{tikzpicture}
   \end{minipage}
   \caption{Pencils of circles (solid) in $\chat$ with their respective
      dual pencils (dotted). They are also known as Apollonian circles.
      There are three types of pencils.
   }
   \label{fig:pencils}
\end{figure}
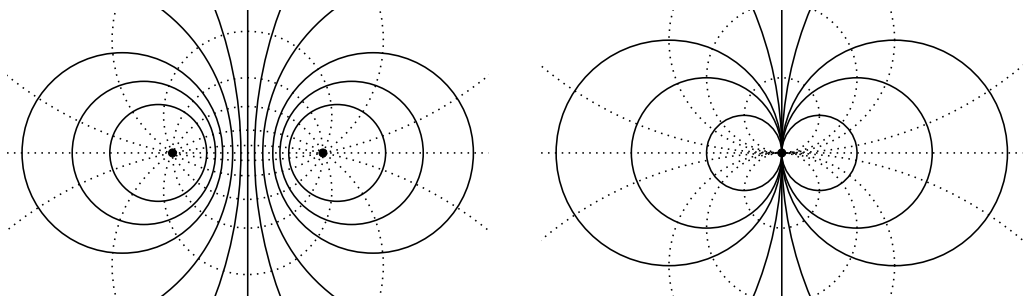

\begin{lemma}
   \label{lemma:gramian_determinant}
   Two M\"obius circles $C_1$ and $C_2$ intersect, touch or are disjoint iff
   the expression
   \begin{equation}
      \label{eq:gramian_determinant}
      |C_1|^2|C_2|^2 - \ip{C_1}{C_2}^2
   \end{equation}
   is positive, zero or negative, respectively.
\end{lemma}
\begin{proof}
   If there are any common points of $C_1$ and $C_2$ they are contained in
   their dual pencil. A pencil of circles contains two, one or zero points
   depending on whether its signature is $+-$, $+0$ or $++$, respectively.
   Thus, the question of common points can be decided by looking at the
   sign of the Gramian determinant of the subspace spanned by $C_1$ and $C_2$,
   that is expression \eqref{eq:gramian_determinant}.
\end{proof}

Prescribing a M\"obius circle, say the $x$-axis, divides the M\"obius plane
$\chat$ into two components. One of them, say the upper half plane, can be
identified with the hyperbolic plane. We denote this component
by $\HH$ and its bounding M\"obius circle by $\partial\HH$. Using the mentioned
normalisation, the subgroup of M\"{o}bius transformations leaving $\partial\HH$
invariant is given by $\PSL(2;\R)$, the group of \define{hyperbolic motions}.
Clearly they preserve M\"obius circles intersecting $\partial\HH$ orthogonally.
These M\"obius circles, or rather their intersection with $\HH$, are the hyperbolic
lines of the hyperbolic plane $\HH$.

\begin{definition}[hyperbolic cycle]
   A M\"obius circle, or more precisely its intersections with $\HH$,
   which is neither $\partial\HH$ nor a hyperbolic line is called a
   \define{hyperbolic cycle} (see Figure \ref{fig:cycles}). The type of a hyperbolic
   cycle is given by the number of intersection points with $\partial\HH$:
   \begin{center}
      \begin{tabular}{ccc}
         no.\ points  && type\\
         \hline\hline
         $0$ && (hyperbolic) circle\\
         $1$ && horocycle\\
         $2$ && hypercycle
      \end{tabular}
   \end{center}
\end{definition}

\begin{figure}[h]
   \begin{minipage}{0.55\textwidth}
      \includegraphics[width=\textwidth]{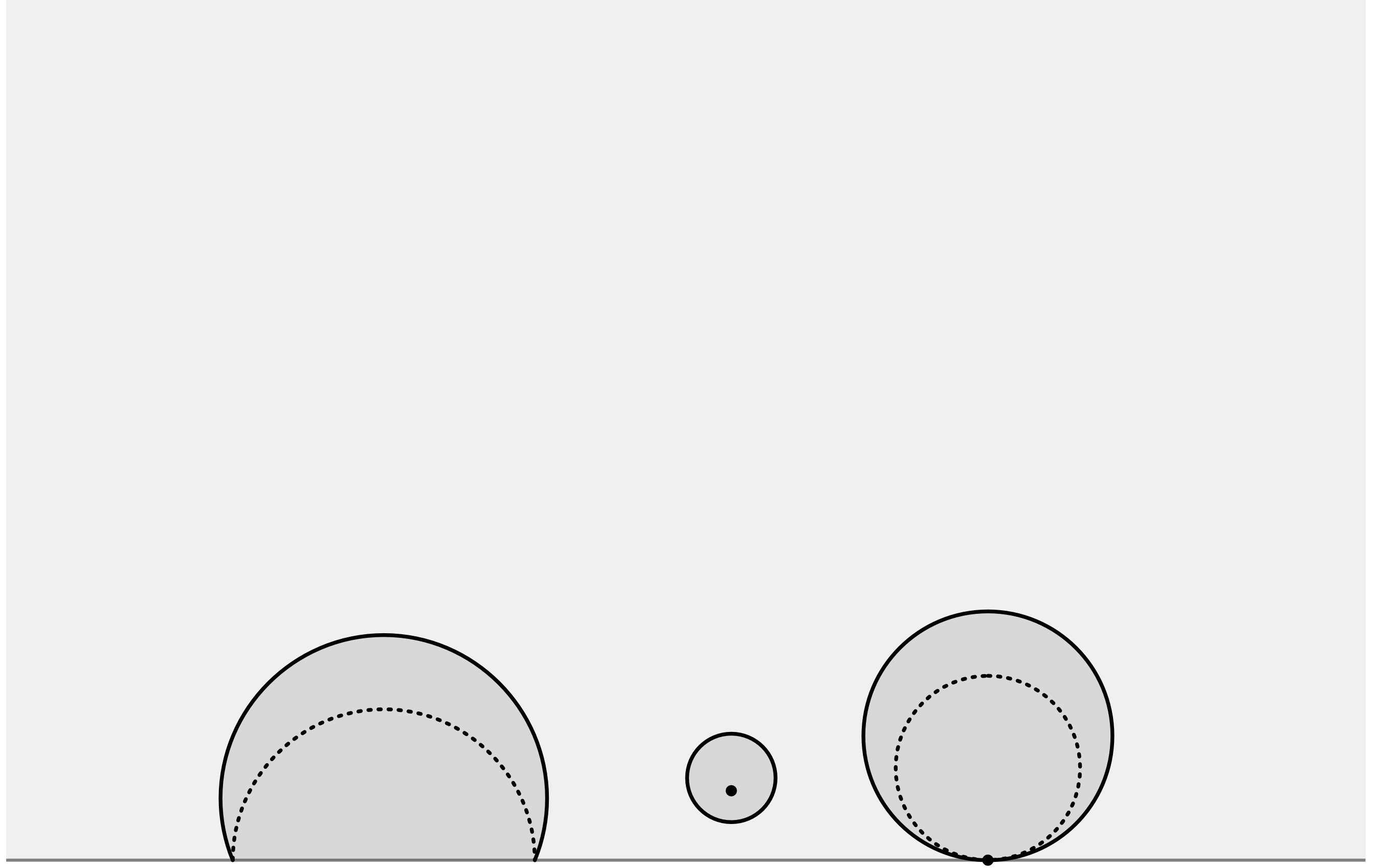}
   \end{minipage}
   \qquad
   \begin{minipage}{0.35\textwidth}
      \includegraphics[width=\textwidth]{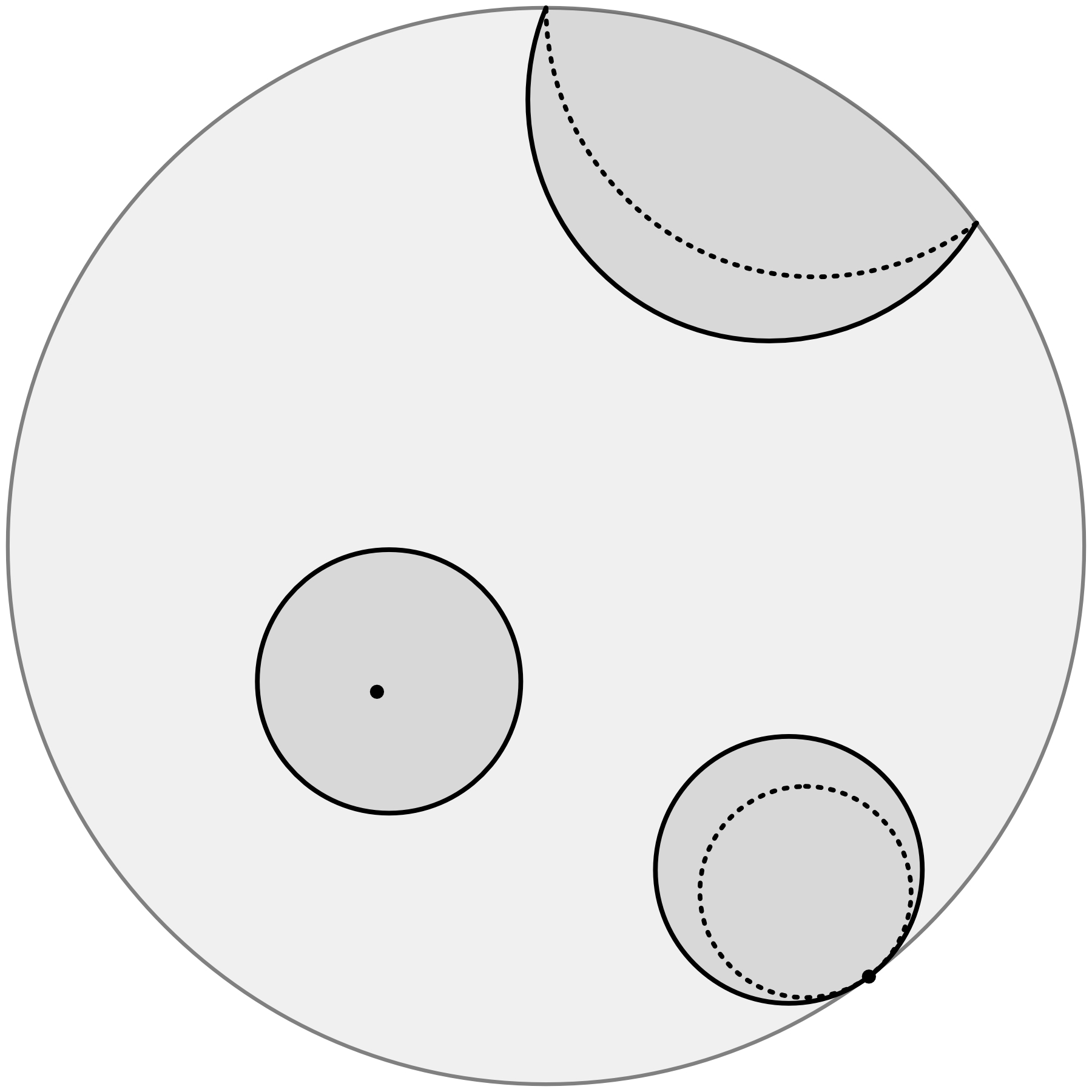}
   \end{minipage}
   \caption{\textsc{Left:} hyperbolic cycles (solid) with their centres
      (dotted) and associated discs (shaded). From left to right: hypercycle, circle,
      horocycle. \textsc{Right:} using the Cayley transform $z\mapsto (z-\ii)/(z+\ii)$
      we can switch to the Poincar\'e disc model of the hyperbolic plane.
   }
   \label{fig:cycles}
\end{figure}

Each hyperbolic cycle spans a pencil together with $\partial\HH$. This pencil
either contains a point in $\HH\cup\partial\HH$ or a hyperbolic line. These
members are called the \define{centres} of the corresponding hyperbolic cycles.
Furthermore, a hyperbolic cycle divides $\HH$ into two components. For a circle
or a hypercycle one of these components contains its centre and in the case of a
horocycle there is a component such that the intersection of its closure in
$\chat$ with $\partial\HH$ is its centre. These components are called
\define{(open) circular discs}, \define{(open) horodiscs} or
\define{(open) hyperdiscs}, respectively. The closure of these discs will
always be considered relative to $\HH$, i.e., it is given by the union of the
disc with its bounding cycle.

\begin{lemma}
   \label{lemma:uniqueness_radical_line}
   Any pencil spanned by two hyperbolic cycles, say $C_1$ and
   $C_2$, contains at most one hyperbolic line.
\end{lemma}
\begin{proof}
   From dimension considerations it follows that the intersection of
   $\{\ip{X}{\partial\HH}=0\}$ and $\lin\{C_1, C_2\}$ is non-empty. Furthermore,
   $\lin\{C_1, C_2\}\nsubset\{\ip{X}{\partial\HH}=0\}$ since neither $C_1$ nor
   $C_2$ is a hyperbolic line or a point.
\end{proof}

\begin{definition}[radical line]\label{def:radical_line}
   Given two hyperbolic cycles. The unique hyperbolic line, if
   existent, in their pencil is called their \emph{(hyperbolic) radical line}.
\end{definition}

Since we normalised $\partial\HH$ to be the $x$-axis its complement
is given by $\Sym(2)\subset\Herm(2)$. Hence, the space of hyperbolic cycles
is given by $\{|X|^2 > 0\}\setminus\Sym(2)$ up to scaling. This can be
simplified by considering an affine space parallel to $\Sym(2)$.

\begin{proposition}[space of hyperbolic cycles]
   \label{prop:hyperbolic_cycles}
   The hyperbolic cycles and points of $\HH$ can be identified with elements
   of $\Sym(2)$. In particular, the type of a cycle represented by $C\in\Sym(2)$
   can be determined using $\ip{\cdot}{\cdot}$ (see Figure
   \ref{fig:space_of_cycles}):
   \begin{center}
      \begin{tabular}{ccc}
         type && norm\\
         \hline\hline
         hypercycle && $|C|^2 > 0$\\
         horocycle && $|C|^2 = 0$, $x_0>0$\\
         circle && $0 > |C|^2 > -1$, $x_0>0$\\
         point && $|C|^2 = -1$, $x_0>0$.
      \end{tabular}
   \end{center}
   Furthermore, two hyperbolic cycles are orthogonal iff their representatives
   $C_1,C_2\in\Sym(2)$ satisfy $\ip{C_1}{C_2} = -1$.
\end{proposition}
\begin{proof}
   As described, we can identify the hyperbolic cycles with part of an
   affine space parallel to $\Sym(2)$, say $\{\ip{X}{\partial\HH} = 1\}$. By
   definition, the type of a cycle $C$ is determined by the signature of
   $\lin\{\partial\HH, C\}$. Our choice of affine space and Lemma
   \ref{lemma:gramian_determinant} lead to the table above. Similarly,
   the characterisation of orthogonality follows from Lemma
   \ref{lemma:conformal_circles_orthogonality}.
\end{proof}

\begin{figure}[h]
   \begin{tikzpicture}
      \tikzset{
         hatch distance/.store in=\hatchdistance,
         hatch distance=10pt,
         hatch thickness/.store in=\hatchthickness,
         hatch thickness=2pt
      }

      \makeatletter
      \pgfdeclarepatternformonly[\hatchdistance,\hatchthickness]{flexible hatch}
      {\pgfqpoint{0pt}{0pt}}
      {\pgfqpoint{\hatchdistance}{\hatchdistance}}
      {\pgfpoint{\hatchdistance-1pt}{\hatchdistance-1pt}}%
      {
         \pgfsetcolor{\tikz@pattern@color}
         \pgfsetlinewidth{\hatchthickness}
         \pgfpathmoveto{\pgfqpoint{0pt}{0pt}}
         \pgfpathlineto{\pgfqpoint{\hatchdistance}{\hatchdistance}}
         \pgfusepath{stroke}
      }
      \makeatother

      \begin{axis}[width=0.5\textwidth, axis equal,
         xmin=-2, xmax=2, ymin=-1, ymax=2, hide axis]
         \addplot[name path=A,
            domain=-2:2, samples=1000, line width=1.2pt]{sqrt(x^2+0.5)};
         \addplot[name path=B, domain=0:2, samples=2, dashed, line width=1.2pt]{x};
         \addplot[name path=C, domain=-2:0, samples=2, dashed, line width=1.2pt]{-x};
         \addplot[name path=D, domain=0:2, samples=2, opacity=0]{-x};
         \addplot[name path=E, domain=-2:0, samples=2, opacity=0]{x};

         \addplot [mygray] fill between [of=A and B, soft clip={domain=0:2}];
         \addplot [mygray] fill between [of=A and C, soft clip={domain=-2:0}];
         \addplot [pattern color=mygray,
            pattern=flexible hatch,
            hatch distance=6pt,
            hatch thickness=1pt]
            fill between [of=B and D, soft clip={domain=0:2}];
         \addplot [pattern color=mygray,
            pattern=flexible hatch,
            hatch distance=6pt,
            hatch thickness=1pt]
            fill between [of=C and E, soft clip={domain=-2:0}];
      \end{axis}
   \end{tikzpicture}
   \caption{2D-sketch of the domains representing hyperbolic cycles:
      points (solid line); circles (shaded region); horocycles (dotted line);
      hypercycles (striped region).}
   \label{fig:space_of_cycles}
\end{figure}
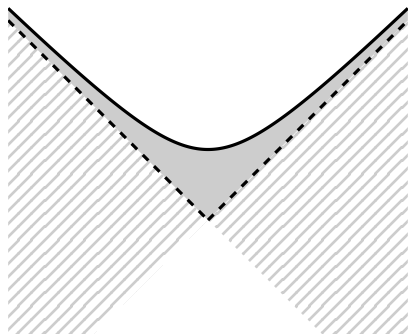

\subsection{Hyperbolic polygons and decorations}
\label{sec:hyperbolic_polygons_decorations}
To a finite collection of hyperbolic circles we can naturally associate a hyperbolic
polygon by considering the convex hull of their centres. We are now going to
investigate how this construction can be extended to more general collections
of hyperbolic cycles.

\begin{definition}[hyperbolic polygons]
   Consider a finite collection $\{C_n\}_{n=1}^{N}$, $N\geq3$, of hyperbolic cycles.
   Suppose that their associated discs are pairwise disjoint. Their
   \define{associated hyperbolic polygon} is
   \[
      \poly\big(\{C_n\}_{n=1}^N\big)
      \;\coloneq\;
      \left\{p\in\HH \,:\,
         p = \sum_{n=1}^N\alpha_nC_n,\;\,
         \alpha_1,\dotsc,\alpha_N\geq0
      \right\}.
   \]
   A \define{(convex) hyperbolic polygon}  (see Figure \ref{fig:hyperbolic_polygons},
   left) is a subset $P\subset\HH$ such that there is some sequence
   $\{C_n\}_{n=1}^N$ of hyperbolic cycles with $P=\poly\big(\{C_n\}_{n=1}^N\big)$.
\end{definition}

We call a collection of hyperbolic cycles \define{minimal} if there is no $n$ such that
the centre of $C_n$ is contained in $\poly\big(\{C_n\}_{n=1}^N\big)\eqcolon P$. In this
case we also call $P$ a \define{hyperbolic $N$-gon} and the centres of the $C_n$ the
\define{vertices} of $P$. In particular, $P$ is a \define{hyperbolic triangle} or
\define{quadrilateral} if $N=3$ or $N=4$, respectively. By our assumption about the
associated discs, we can reorder a minimal sequence of cycles defining $P$ such that
there are $L_n\in\Sym(2)$ with $|L_n|^2=1$,
\[
   \poly\big(\{C_n\}_{n=1}^N\big)
   \;=\;
   \HH \,\cap\, \bigcap_{n=1}^N\{\ip{X}{L_n}\leq0\}
\]
and $\ip{C_n}{L_n}=0=\ip{C_{n+1}}{L_n}$, where $C_{N+1}=C_1$. The intersection
$P\cap\{\ip{X}{L_n}=0\}$ is a hyperbolic line segment and we call it an \define{edge}
of $P$. Suppose that there are
$0<M\leq N$ vertices of $P$ which are hyperbolic lines. For each such vertex there is
$c_m\in\Sym(2)$ representing the centre of $C_{n_m}$ with $|c_m|^2 = 1$ and
$\ip{c_m}{C_n} < 0$ for all $n\neq n_m$. The \define{truncation} of $P$ is defined as
\[
   \trunc(P)
   \;\coloneq\;
   P \,\cap\, \bigcap_{m=1}^M \{\ip{X}{c_m}\leq0\}.
\]

\begin{definition}[decorated hyperbolic polygon]
   \label{def:decorated_polygon}
   Let $P$ be a hyperbolic $N$-gon and denote by $v_1, \dotsc, v_N$
   its vertices. A \define{decoration} of $P$ is a choice of hyperbolic cycles
   $C_{v_1}, \dotsc, C_{v_N}$ such that $C_{v_n}$ is centred at $v_n$ and all cycles
   intersect the interior of the truncation of $P$. The polygon $P$ together
   with the cycles $C_{v_n}$ is called a \define{decorated hyperbolic polygon}
   (see Figure \ref{fig:hyperbolic_polygons}, right).
\end{definition}

\begin{figure}[h]
   \begin{minipage}{0.45\textwidth}
      \centering
      \begin{tikzpicture}[scale=2]
         \tkzDefPoint(0,0){O}; \tkzDefPoint(1,0){Z};
         \tkzDefPoint(1,1){Y}; \tkzInterLC(O,Y)(O,Z); \tkzGetPoints{X}{A};
         \tkzDefPoint(-1,-0.5){Y}; \tkzInterLC(O,Y)(O,Z); \tkzGetPoints{X}{D}
         \tkzDefPoint(0.85,-0.1){B}; \tkzDefPoint(0.1,-1.2){C};
         \tkzDefPoint(-0.93,0.65){E};

         \tkzDrawCircle[gray,fill=mygray!60,line width=1pt](O,Z);
         \begin{scope}
            \tkzClipCircle(O,Z);
            \tkzDrawPolygon[fill=black!20](A,B,C,D,E);
         \end{scope}

         \tkzDefTangent[from=C](O,Z); \tkzGetPoints{c1}{c2};
         \tkzDrawSegment[line width=1pt, dotted](C,c1);
         \tkzDrawSegment[line width=1pt, dotted](C,c2);
         \tkzDrawSegment[line width=1pt, dashed](c1,c2);

         \tkzDefTangent[from=E](O,Z); \tkzGetPoints{c1}{c2};
         \tkzDrawSegment[line width=1pt, dotted](E,c1);
         \tkzDrawSegment[line width=1pt, dotted](E,c2);
         \tkzDrawSegment[line width=1pt, dashed](c1,c2);

         \tkzDrawPolygon[line width=1pt](A,B,C,D,E);
         \tkzDrawPoints(A,B,C,D,E);
      \end{tikzpicture}
   \end{minipage}
   \quad
   \begin{minipage}{0.51\textwidth}
      \includegraphics[width=\textwidth]{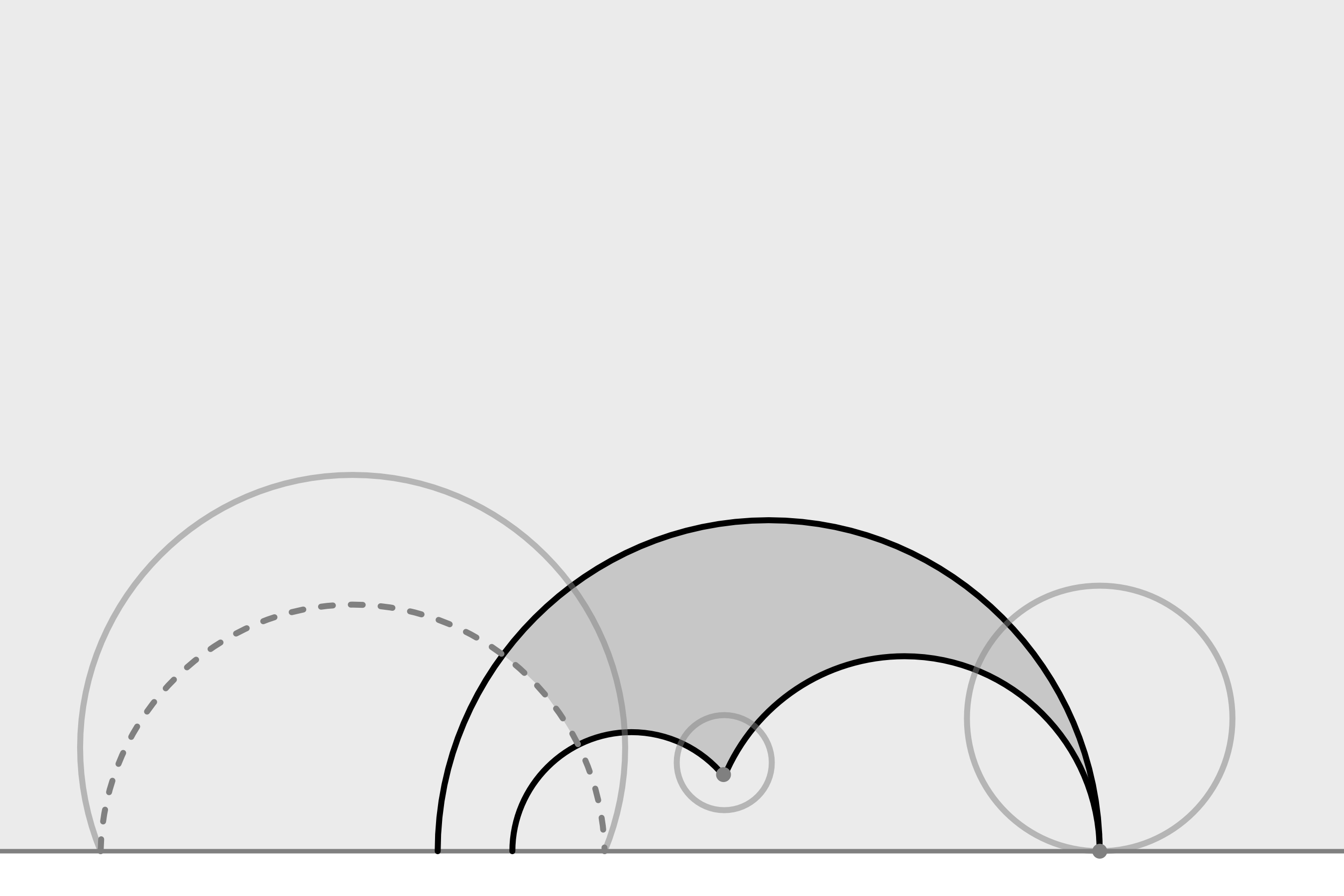}
   \end{minipage}
   \caption{\textsc{Left:} in the Beltrami-Klein model of the hyperbolic plane a
      hyperbolic polygon is given by the intersection of the Euclidean unit disc with
      a Euclidean polygon. Vertices outside the disc correspond to hyperbolic lines
      by polarity (dashed lines). \textsc{Right:} a decorated hyperbolic triangle in
      the half-plane model. The truncation of the triangle is shaded.
   }
   \label{fig:hyperbolic_polygons}
\end{figure}

Consider a vertex $v$ of a decorated hyperbolic polygon $P$ incident
to the hyperbolic lines $L_n$ and $L_m$ with decorating cycle $C_v$.
The \define{(generalised) angle} $\theta_v$ at $v$ in $P$ is defined as
follows: if $v\in\HH$ then $\theta_v$ is the interior angle
between $L_n$ and $L_m$ in $P$. For $v\in\partial\HH$ the angle is the hyperbolic
length of the horocyclic arc $C_v\cap P$. Finally, if $v$ is a hyperbolic
line we define $\theta_v$ to be the hyperbolic distance between $L_n$ and $L_m$.

For $v\notin\partial\HH$ we define the \define{radius} $r_v$ of $C_v$
to be its distance to its centre $v$. If $v\in\partial\HH$ we choose some
horocycle $H_v$ centred at $v$. We call it an \define{auxiliary centre} of
$C_v$ and the oriented hyperbolic distance $r_v$ between $H_v$ and $C_v$
the \define{(auxiliary) radius} of $C_v$. The orientation is chosen such that $r_v$
is negative if $C_v$ is contained in the horodisc bounded by $H_v$. Whenever it is
clear from the context that we are talking about $H_v$ and not $v$ we might
call $H_v$ a centre, too. Furthermore, let $e$ be an edge of $P$ contained in the
line $L$ with adjacent vertices $u$ and $v$. Its \define{(generalised) edge-length}
$\len_e$ is the oriented distance between the (auxiliary) centres at $u$ and $v$.
Clearly, the notions of auxiliary radius and edge-length depend on the choice of
auxiliary centres. But we will see in the following (Lemma
\ref{lemma:distance_from_product}) that different choice, say $H_v$ and
$\tilde{H}_v$, only result in a constant offset, i.e., the oriented distance
between $H_v$ and $\tilde{H}_v$ (see Figure \ref{fig:horocycle_distances}).

\begin{figure}[h]
   \begin{picture}(400, 150)
      \put(0, 0){
         \includegraphics[width=0.38\textwidth]{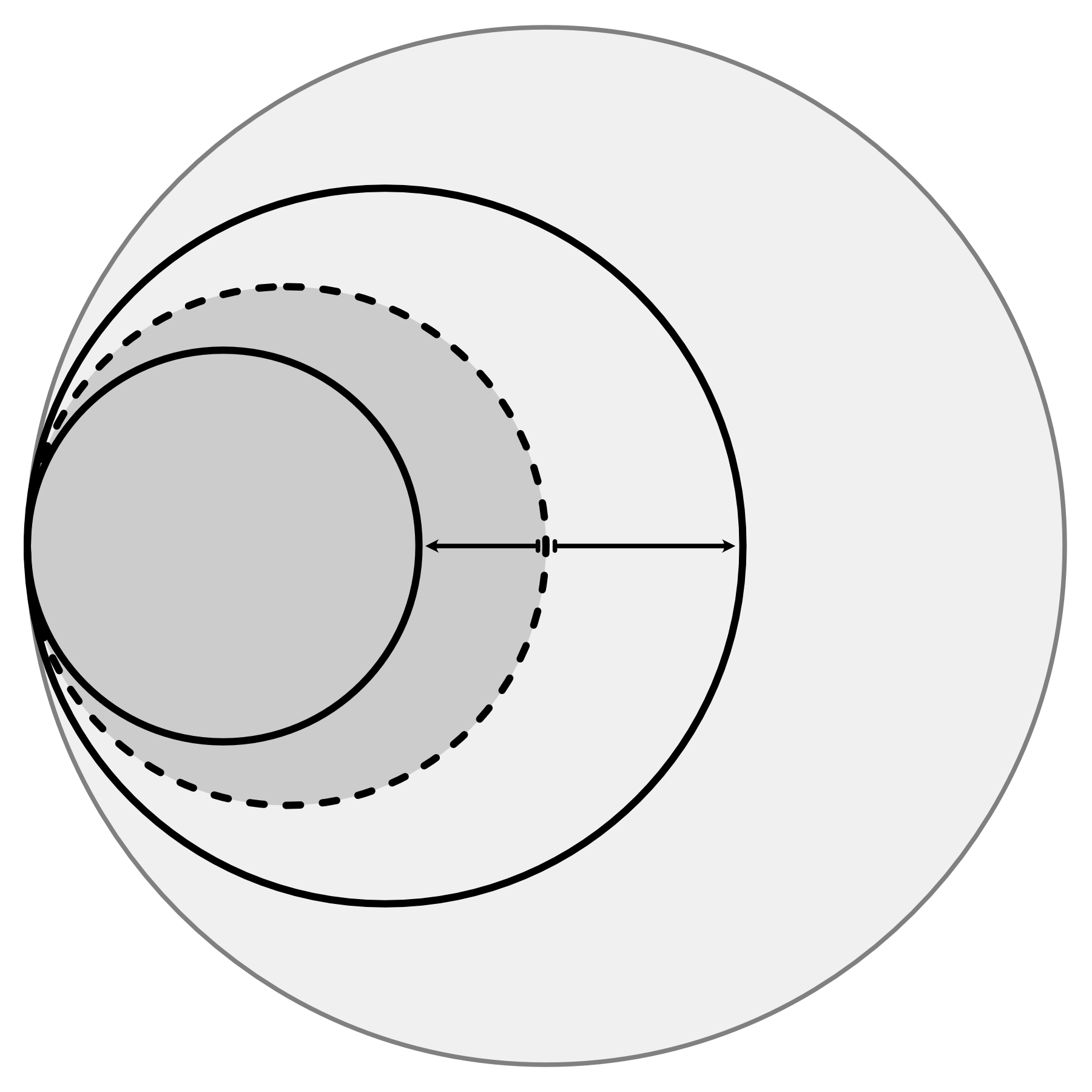}
      }
      \put(170, 0){
         \includegraphics[width=0.55\textwidth]{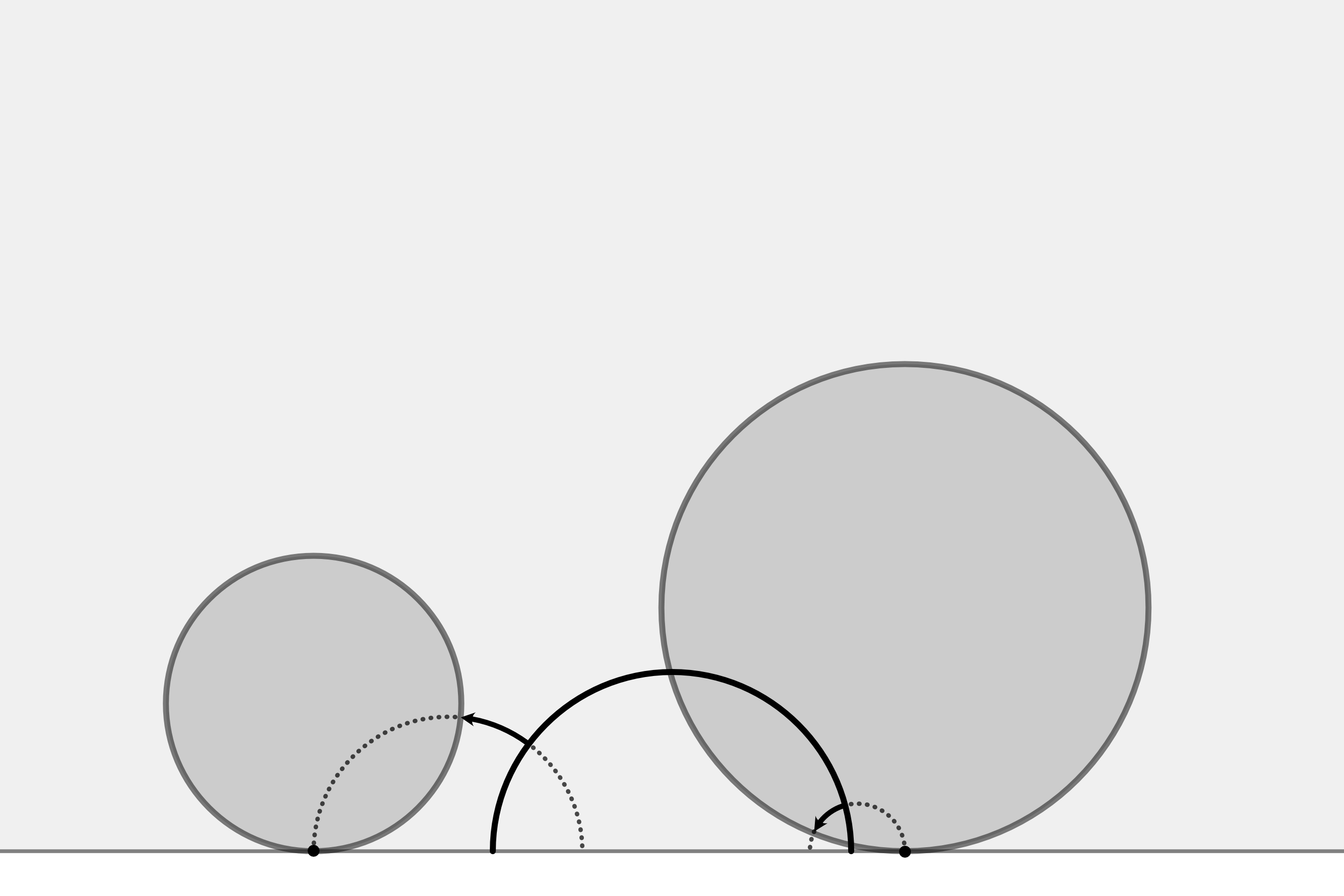}
      }
      \put(59, 107){\footnotesize$H_v$}
      \put(61, 76){\footnotesize$<0$}
      \put(84, 76){\footnotesize$>0$}
      \put(247, 30){\footnotesize$>0$}
      \put(293, 15){\footnotesize$<0$}
   \end{picture}
   \caption{Definition of radius and edge-length for horocycles. \textsc{Left:}
      Concentric horocycles (solid) with auxiliary centre $H_v$ (dashed). The disc
      belonging to $H_v$ is shaded. \textsc{Right:} Edge-length between horocycles
      and a hyperbolic line.
   }
   \label{fig:horocycle_distances}
\end{figure}
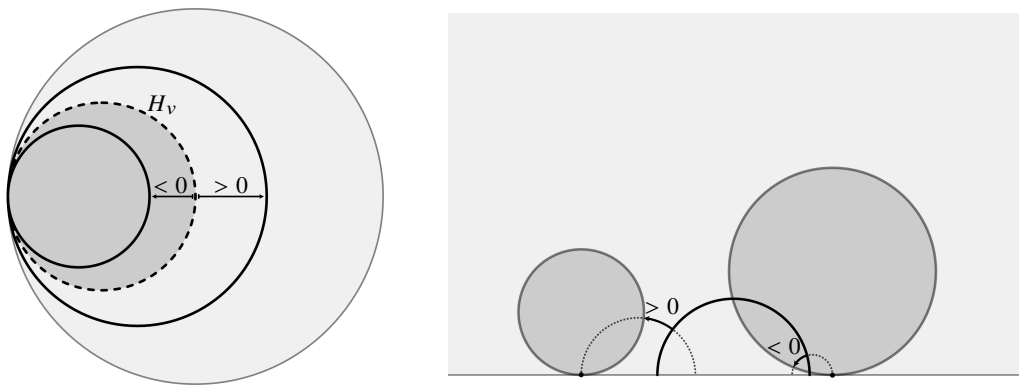

We aim to relate the metric properties of decorated triangles to the
representation of their cycles in $\Sym(2)$. Therefore, we need to introduce some
extra notation. The \define{type} $\epsilon_v$ of a vertex $v$ is $-1$, $0$
or $+1$ depending on whether $v\in\HH$, $v\in\partial\HH$ or is a hyperbolic line.
Furthermore, we define the \define{angle-modifiers} $\rho_{\epsilon}\colon\R\to\R$
by
\[
   \rho_{-1}(\theta) \coloneq \sin(\theta), \qquad
   \rho_{0}(\theta) \coloneq \theta, \qquad
   \rho_{1}(\theta) \coloneq \sinh(\theta)
\]
and the \define{length-modifiers} $\tau_{\epsilon}\colon\R\to\R$ are given by
\begin{equation}
   \label{eq:length_modifiers}
   \tau_{-1}(\len) \coloneq \cosh(\len), \qquad
   \tau_{0}(\len) \coloneq \frac{1}{2}\ee^{\len}, \qquad
   \tau_{1}(\len) \coloneq \sinh(\len).
\end{equation}

\begin{lemma}
   \label{lemma:radius_from_norm}
   Consider a hyperbolic cycle $C$ with centre of type $\epsilon=\pm1$ and
   radius $r$. Then its representative in $C\in\Sym(2)$ satisfies
   \begin{equation}
      |C|^2 \eq \frac{\epsilon}{\tau_{\epsilon}^2(r)}.
   \end{equation}
\end{lemma}
\begin{proof}
   Using a M\"{o}bius transformation we can assume that the centre
   of the cycle is $\ii$ or intersects the $y$-axis orthogonally in $\ii$,
   respectively. The hyperbolic distance in the Poincar\'e metric for two points
   $p\ii, q\ii\in\HH$ on the $y$-axis takes the form
   \begin{equation}
      \dist_{\HH}(p\ii, q\ii) = |\ln(p) - \ln(q)|.
   \end{equation}
   Hence, it follows that the cycle can be represented in $\Herm(2)$ by
   \begin{equation}
      \label{eq:cycle_herm_matrix_1}
      \begin{pmatrix}
         1 & \ii\tau_{\epsilon}(r)\\
         -\ii\tau_{\epsilon}(r) & -\epsilon
      \end{pmatrix}.
   \end{equation}
   The assertion follows by direct computation.
\end{proof}

\begin{lemma}
   \label{lemma:distance_from_product}
   Given a decorated hyperbolic polygon. Denote by $\len_{uv}$ the length of the edge
   between two adjacent vertices $u$ and $v$. Then the product of the cycles
   $C_u, C_v\in\Sym(2)$ at these vertices is
   \begin{equation}
      -\ip{C_u}{C_v}
      \eq
      \frac{\tau_{\epsilon_u\epsilon_v}'(\len_{uv})}
         {\tau_{\epsilon_u}(r_v)\tau_{\epsilon_v}(r_v)}.
   \end{equation}
\end{lemma}
\begin{proof}
   We begin by normalising the first cycle as in the previous Lemma
   \ref{lemma:radius_from_norm}. Note that equation \eqref{eq:cycle_herm_matrix_1}
   for the representative in $\Herm(2)$ remains valid for the horocycle passing
   through $0$ and $\ii$ with auxiliary radius $r_u=0$. The second cycle is then
   given by
   \begin{equation}
      \label{eq:cycle_herm_matrix_2}
      \begin{cases}
         \begin{pmatrix}
            0 & \ii\\
            -\ii & -2\ee^{\len_{uv}}
         \end{pmatrix}
         &\text{if $\epsilon_v=0$ or}\\[0.5cm]
         \begin{pmatrix}
            1 & -\ii\tau_{\epsilon_v}(r_v)\ee^{\len_{uv}}\\
            \ii\tau_{\epsilon_v}(r_v)\ee^{\len_{uv}} & -\epsilon_v\ee^{2\len_{uv}}
         \end{pmatrix}
         &\text{if $\epsilon_v\neq0$.}
      \end{cases}
   \end{equation}
   Again, the assertion follows by direct computation.
\end{proof}

\begin{lemma}[hyperbolic cosine laws]
   \label{lemma:cos_sin_laws}
   Consider a decorated hyperbolic triangle with vertices $u$, $v$ and $w$.
   Denote by $\len_{uv}$, $\len_{vw}$ and $\len_{wu}$ the edge-lengths and
   suppose that $\epsilon_v=-1$. Then the angle $\theta_v$ is related to the
   edge-lengths by
   \begin{equation}
      \cos(\theta_v)
      \;=\;
      \frac{-\tau_{\epsilon_w\epsilon_u}'(\len_{wu})
         \plus \tau_{\epsilon_u}(\len_{uv})\tau_{\epsilon_w}(\len_{vw})}
         {\tau_{\epsilon_u}'(\len_{uv})\tau_{\epsilon_w}'(\len_{vw})}.
   \end{equation}
\end{lemma}
\begin{proof}
   These relations follow either by direct computation for the different cases
   \cite{Ratcliffe94}*{\S3.5} or using a combined approach by analysing bases in
   $\Sym(2)$ \cite{Thurston97}*{Section 2.4}.
\end{proof}

\begin{definition}[modified tangent distance]
   \label{def:modified_tangent_distance}
   Let $C$ be hyperbolic cycle of type $\epsilon$ with (auxiliary) centre $c$
   and radius $r$. The \emph{modified tangent distance} between $C$ and a
   point $x\in\HH$ is
   \[
      \tandist_x(c,r)\,\coloneq\,\tandist_x(C)\,\coloneq\,
      \frac{\tau_{\epsilon}(\dist_{\HH}(c,x))}{\tau_{\epsilon}(r)}.
   \]
   Here, we orient $\dist_{\HH}$ such that $\dist_{\HH}(c,x)>0$
   iff $\ip{C}{x}<(\epsilon - 1)/\tau_{\epsilon}(r)$ (compare to Lemma
   \ref{lemma:distance_from_product}). Note that this condition is satisfied for
   all points $x\neq c$ if $\epsilon=-1$.
\end{definition}

\begin{lemma}
   \label{lemma:formula_radical_line}
   Let $C_1$ and $C_2$ be hyperbolic cycles whose associated discs do not contain
   each other, respectively. Then the radical line of $C_1$ and $C_2$ exists and
   is given by $\{x\in\HH : \tandist_x(C_1)=\tandist_x(C_2)\}$.
\end{lemma}
\begin{proof}
   Suppose a point $x\in\HH$ is not contained in the disc associated to $C_n$,
   $n=1,2$. By Lemma \ref{lemma:cos_sin_laws}, $\tandist_x(C_n) = \cosh(\delta)$
   where $\delta$ is the hyperbolic length of the hyperbolic segments which starts
   in $x$ and ends at a tangent point with $C_n$ (see Figure
   \ref{fig:tangent_distance}). Thus, if there is a point $x$ such that
   \begin{equation}
      \label{eq:equal_tangent_distance}
      \tandist_x(C_1)\eq\tandist_x(C_2)
   \end{equation}
   then $x$ is the centre of a hyperbolic circle $C$ which is orthogonal to
   both $C_1$ and $C_2$.

   \begin{figure}[t]
      \begin{picture}(385, 160)
         \put(115, 0){
            \includegraphics[width=0.4\textwidth]{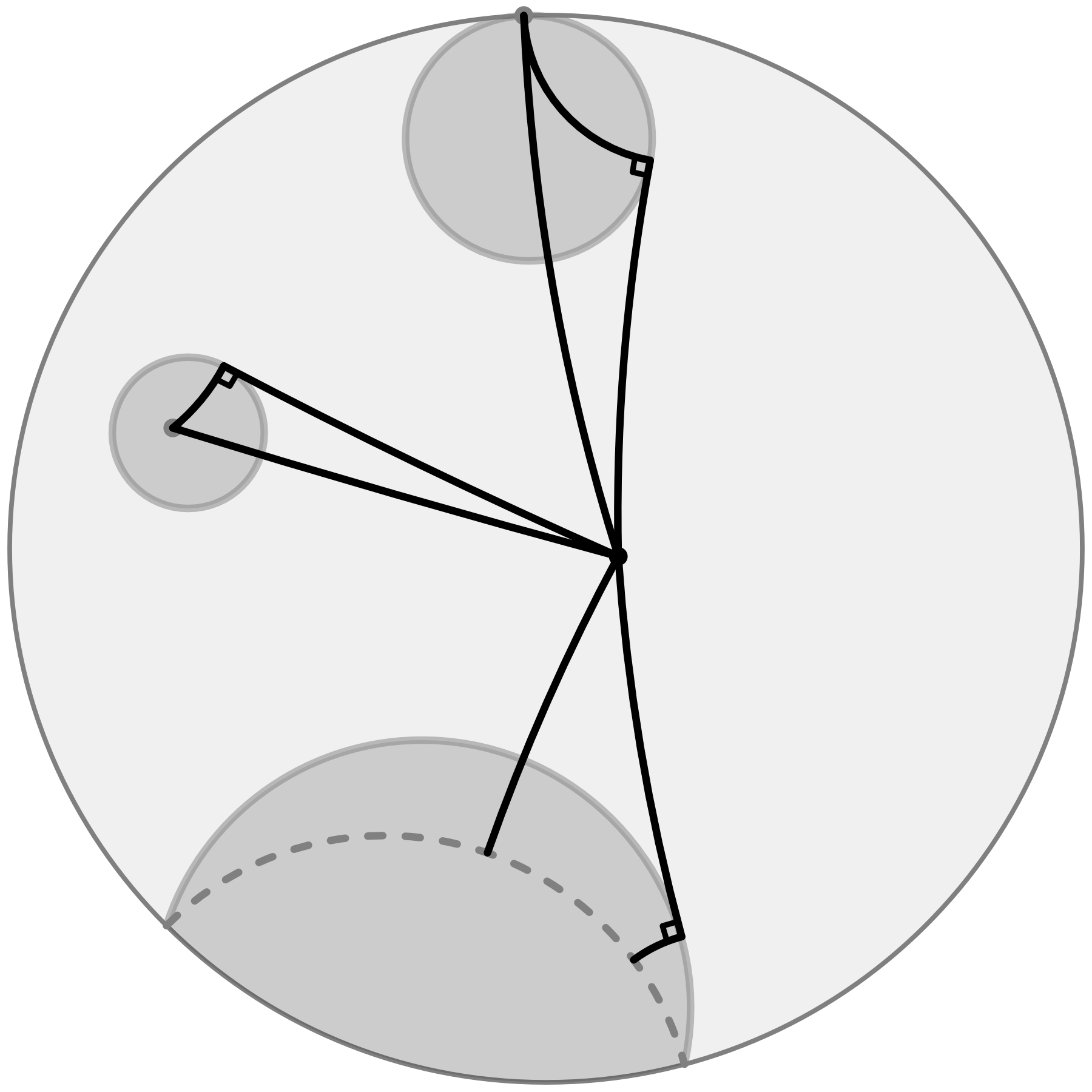}
         }
         \put(210, 75){\small $x$}

         \put(200, 141){\small $r_1$}
         \put(209, 105){\small $\delta_1$}

         \put(137, 100){\small $r_2$}
         \put(171, 95){\small $\delta_2$}

         \put(211, 14){\small $r_3$}
         \put(212, 45){\small $\delta_3$}
      \end{picture}
      \caption{The modified tangent distance between a cycle $C$ and a point $x$
         outside its associated disc is the hyperbolic cosine of the length of
         the tangential segment to $C$ starting at $x$. It can be computed
         using right-angled hyperbolic triangles.
      }
      \label{fig:tangent_distance}
   \end{figure}
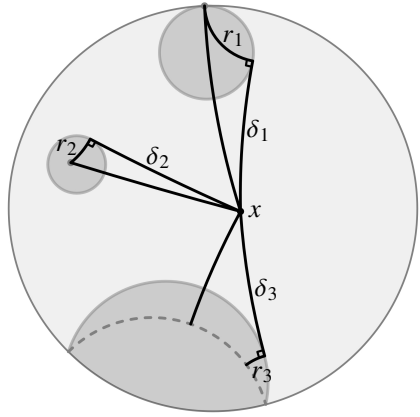

   Now, consider the function $f\colon x\mapsto \tandist_x(C_1) - \tandist_x(C_2)$.
   It is continuous. Furthermore, we find $x_{+},x_{-}\in\HH$ with $f(x_{+})>0$ and
   $f(x_{-})<0$ because of our assumption about the associated discs. It follows,
   considering two non-intersecting continuous path starting at $x_{+}$ and
   ending at $x_{-}$, that there are $p,\tilde{p}\in\HH$ which satisfy equation
   \eqref{eq:equal_tangent_distance}. We already observed that they are the
   centres of two hyperbolic circles which are orthogonal to both $C_1$ and
   $C_2$. Hence, they span the dual pencil of $C_1$ and $C_2$. We conclude that
   the hyperbolic line connecting $p$ and $\tilde{p}$ is in the pencil spanned by
   $C_1$ and $C_2$, that is, the unique radical line of $C_1$ and $C_2$ (Lemma
   \ref{lemma:uniqueness_radical_line}).
\end{proof}

\begin{corollary}
   \label{cor:radical_line_intersection}
   Let $C_1$ be a hyperbolic cycle with (auxiliary) centre $c_1$ and $C_2$ a circle
   with centre $c_2$. In addition to the assumptions of Lemma
   \ref{lemma:formula_radical_line} we require that $\dist_{\HH}(c_1, c_2) > 0$.
   The orientation is chosen as in Definition \ref{def:modified_tangent_distance}.
   Then the radical line of $C_1$ and $C_2$ intersects the hyperbolic ray starting at
   $c_2$ extending towards the centre of $C_1$
   iff $\tau_{\epsilon_1}(r_1)/\cosh(r_2)<\tau_{\epsilon_1}(\dist_{\HH}(c_1, c_2))$.
\end{corollary}

Suppose $C_1, C_2, C_3\in\Sym(2)$ are the vertex cycles of a decorated hyperbolic
triangle. They form a basis of $\Sym(2)$. Hence, they determine a unique affine
plane in $\Sym(2)$. There is a unique $F\in\Sym(2)$ such that this plane is given
by $\{\ip{X}{F} = -1\}$. We call $F$ the \define{face-vector} of the triangle.

\begin{lemma}
   \label{lemma:intersection_radical_lines}
   The three radical lines defined by a decorated hyperbolic triangle
   whose vertex discs do not contain each other, respectively, either
   intersect in a common point in $\HH\cup\partial\HH$ or are
   all orthogonal to a common hyperbolic line.
\end{lemma}
\begin{proof}
   Let $C_1$, $C_2$ and $C_3$ denote the vertex cycles and $F$ their face-vector.
   From Lemma \ref{lemma:distance_from_product} and
   Lemma \ref{lemma:formula_radical_line} we deduce that the radical line of
   the vertex cycles $C_m$ and $C_n$, $m\neq n$, is given by
   $\{\ip{X}{C_n-C_m}=0\}$. Since $\ip{F}{C_n} = -1$, $n=1,2,3$, we see that
   \[
      \bigcap_{n\neq m}\{\ip{X}{C_n-C_m} = 0\} \eq \lin\{F\}.\qedhere
   \]
\end{proof}

By definition the vertices of an associated hyperbolic polygon are centres of the
defining cycles. But in general not all centres need to be vertices, too. We are now
going to find some sufficient conditions in terms of the \define{intersection angles}
the cycles. It is understood to be the interior intersection angle of their associated
discs.

\begin{lemma}
   \label{lemma:intersection_angle}
   Consider two vertex cycles, say $C_1$ and $C_2$. Their associated discs either
   do not intersect or intersect at most with an angle of $\pi/2$ if
   $\ip{C_1}{C_2} \leq -1$.
\end{lemma}
\begin{proof}
   This assertion follows from combining Lemma \ref{lemma:cos_sin_laws}
   and Lemma \ref{lemma:distance_from_product}.
\end{proof}

\begin{lemma}
   \label{lemma:concave_hull}
   Consider a decorated hyperbolic triangle with vertex cycles given by
   $C_1,C_2,C_3\in\Sym(2)$. Let $F$ be their face-vector. Suppose that
   $X=\alpha_1 C_1 + \alpha_2 C_2 + \alpha_3 C_3$ is a hyperbolic
   cycle with $\alpha_1,\alpha_2,\alpha_3 \geq0$ and $\ip{F}{X} \leq -1$. Then
   there is an $n$ such that $\ip{C_n}{X} > -1$.
\end{lemma}
\begin{proof}
   From $\ip{F}{C_i} = -1$, we deduce that
   \[
      -1 \,\geq\, \ip{F}{X} \eq -(\alpha_1+\alpha_2+\alpha_3).
   \]
   By assumption $X$ is a hyperbolic cycle. Proposition \ref{prop:hyperbolic_cycles}
   leads to
   \begin{equation}
      |X|^2
      \eq \alpha_1\ip{C_1}{X} \plus \alpha_2\ip{C_2}{X} \plus\alpha_3\ip{C_3}{X}
      \,>\, -1.
   \end{equation}
   Combining these two inequalities yields the result.
\end{proof}

\begin{proposition}
   \label{prop:associated_polygon}
   Let $C_1, \dotsc, C_N$ be hyperbolic cycles with pairwise non-intersecting
   discs. Suppose that there is $F\in\Sym(2)$ such that $\ip{F}{C_n}=-1$
   for all $n=1,\dotsc, N$. Then the centre of each $C_n$ is a vertex of the
   associated hyperbolic polygon $\poly(C_1,\dotsc,C_N)$.
\end{proposition}
\begin{proof}
   Suppose otherwise. Then there are four cycles, w.l.o.g., $C_1$, $C_2$,
   $C_3$ and $C_4$, such that $C_4= \alpha_1 C_1 + \alpha_2 C_2 + \alpha_3 C_3$
   with $\alpha_1,\alpha_2,\alpha_3\geq0$. By Lemma \ref{lemma:concave_hull} there
   is an $n\in\{1,2,3\}$ such that $\ip{C_4}{C_n} > -1$. Hence, Lemma
   \ref{lemma:intersection_angle} implies that the discs of $C_4$ and $C_n$
   intersect. This contradicts our assumption.
\end{proof}

\subsection{The local Delaunay condition}
A \define{triangulated decorated hyperbolic quadrilateral} is the union of two
decorated hyperbolic triangles with disjoint interiors which share an edge and the
corresponding vertex cycles. For the rest of the section we refer to them simply as
decorated quadrilaterals. The two triangles give a \define{triangulation} of the
quadrilateral and their common edge is called \define{diagonal}. Combinatorially
there are two triangulation for each quadrilateral. This change of combinatorics is
called an \define{edge-flip}. The diagonal of a decorated quadrilateral is called
\define{flippable} if its edge-flip can be geometrically realised in $\HH$. It is
immediate that a diagonal is flippable iff the decorated quadrilateral is strictly
convex. Note that a decorated quadrilateral in the sense of
this section need not be convex. Hence it is not necessarily a decorated hyperbolic
$4$-gon as defined in Definition \ref{def:decorated_polygon}. Still, a decorated
quadrilateral is always strictly convex if it has no vertices contained in $\HH$.

\begin{definition}[local Delaunay condition]
   \label{def:local_delaunay_condition}
   Consider a decorated quadrilateral with vertex cycles
   $C_1, C_2, C_3, C_4\in\Sym(2)$ such that $C_2$ and $C_3$ belong to the
   diagonal. Denote by $F_{123}, F_{234}\in\Sym(2)$ the face-vectors.
   We say that the diagonal satisfies the \define{local Delaunay condition},
   or is \define{local Delaunay}, iff
   \begin{equation}
      \label{eq:local_delaunay_vectors}
      \ip{C_1}{F_{234}} \,\leq\, -1
      \qquad\text{and}\qquad
      \ip{C_4}{F_{123}} \,\leq\, -1.
   \end{equation}
\end{definition}

\begin{remark}
   \label{remark:geometric_interpretation_face_vector}
   Suppose that $|F_{123}|^2 > -1$. Then it represents a hyperbolic cycle
   which is orthogonal to $C_1$, $C_2$ and $C_3$. The proof of Lemma
   \ref{lemma:intersection_radical_lines} shows that the centre of $F_{123}$
   is the \enquote{intersection point} of the radical lines of the triangle
   corresponding to $F_{123}$. In addition, Lemma \ref{lemma:intersection_angle}
   shows that the local Delaunay is equivalent to $F_{123}$ intersecting
   $C_4$ at most orthogonally.
\end{remark}

In the following we are going to derive a way to determine the local Delaunay
condition just by intrinsic properties of the decorated quadrilateral. To this
end, again denote by $C_1$, $C_2$ and $C_3$ the vertex cycles of a decorated
hyperbolic triangle. For any permutation $(m,n,k)$ of $\{1,2,3\}$ the subspace
spanned by $C_m$ and $C_n$ corresponds to an edge of the triangle. Therefore,
there is a $L_{k}\in\Sym(2)$ with $|L_{k}|^2 = 1$ such that this subspace is given
by the complement of $L_{k}$ in $\Sym(2)$. The $L_{k}$ can be chosen in such a way
that $\ip{L_{k}}{C_k} < 0$. Denote by $r_n$ the radius and by $\epsilon_n$ the type
of the cycle represented by $C_n$. Furthermore, let $\theta_n$ be the
interior angle at the vertex $n$ and $d_n$ be the (oriented) distance
between the (auxiliary) centre of $C_n$ and the line $L_n$ (see
Figure \ref{fig:notation_triangle}). Note that the $d_n$ can be computed from the
edge-lengths.

\begin{figure}[h]
   \begin{picture}(385, 160)
      \put(115, 0){
         \includegraphics[width=0.4\textwidth]{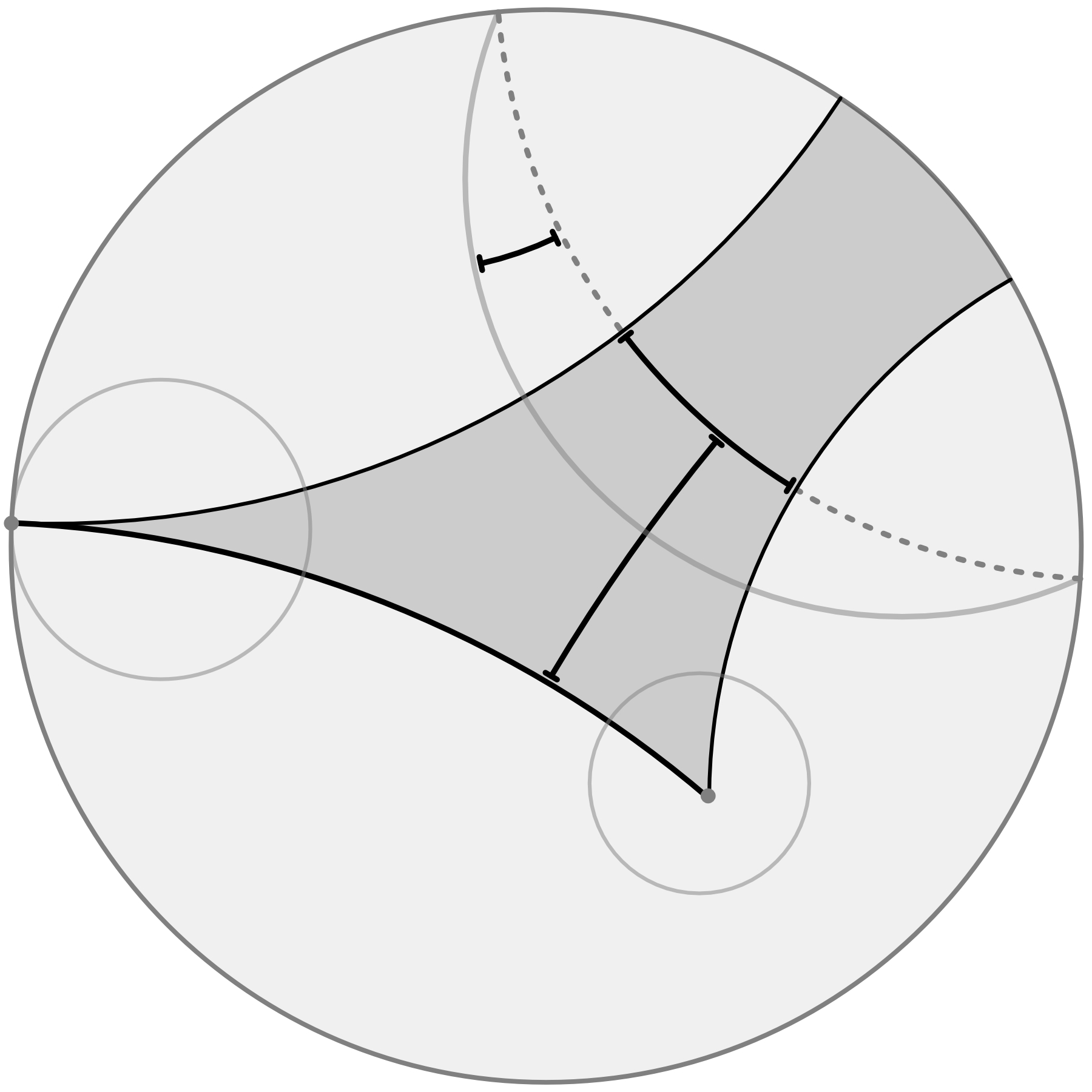}
      }
      \put(181,56){\small $L_n$}
      \put(196,78){\small $d_n$}
      \put(186,124){\small $r_n$}
      \put(218,99){\small $\theta_n$}
   \end{picture}
   \caption{Notation in a decorated hyperbolic triangle.}
   \label{fig:notation_triangle}
\end{figure}
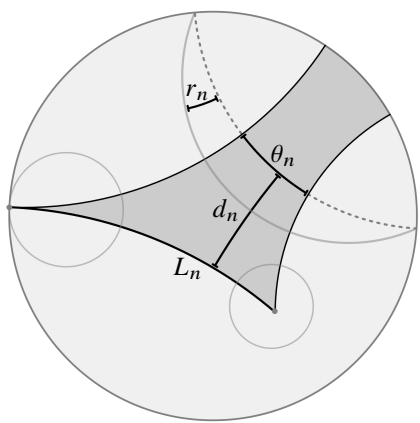

\begin{lemma}
   \label{lemma:line_basis}
   The $L_{k}$ form a basis of $\Sym(2)$. Their dual basis is
   $-\tau_{\epsilon_k}(r_k)/\tau_{\epsilon_k}'(d_k)\,C_k$.
\end{lemma}
\begin{proof}
   By construction $\ip{L_k}{C_m} = 0 = \ip{L_k}{C_m}$. Hence,
   up to a scalar, $C_k$ is the dual vector of $L_k$. The scale factor
   follows from Lemma \ref{lemma:distance_from_product}.
\end{proof}

\begin{lemma}
   \label{lemma:angles_by_product}
   Let $L_1,L_2\in\Herm(2)$ be representatives of two hyperbolic lines.
   Suppose that $|L_n|^2=1$. Denote by $\theta$ the generalised angle between
   $L_1$ and $L_2$ and by $\epsilon$ the type of there common vertex. Then
   \begin{equation}
      |\rho_{\epsilon}'(\theta)| \eq |\ip{L_1}{L_2}|.
   \end{equation}
\end{lemma}
\begin{proof}
   We can normalise the first line to be the $y$-axis. Then its representative
   is given by $L_1=\left(\begin{smallmatrix}0&1\\1&0\end{smallmatrix}\right)$.
   The representative of the second line can now be obtained by applying
   the appropriate hyperbolic motion to $L_1$, i.e.,
   \begin{equation}
      \label{eq:standard_transformations}
      \begin{pmatrix}
         \sin(\theta) & 1-\cos(\theta)\\
         \cos(\theta)-1 & \sin(\theta)
      \end{pmatrix},
      \qquad
      \begin{pmatrix}
         1 & 0\\
         \theta & 1
      \end{pmatrix}
      \qquad\text{or}\qquad
      \begin{pmatrix}
         \sinh(\theta) & 1-\cosh(\theta)\\
         1-\cosh(\theta) & \sinh(\theta)
      \end{pmatrix}
   \end{equation}
   depending on whether the common vertex of $L_1$ and $L_2$ has type $-1$, $0$
   or $1$, respectively. The rest of the proof follows by direct computation.
\end{proof}

\begin{lemma}
   \label{lemma:tilds}
   Consider a decorated hyperbolic triangle with face-vector $F$. Define
   $t_n \coloneq \ip{F}{L_n}$. The $t_n$ can be computed by
   \begin{equation}
      \begin{pmatrix}
         t_1\\ t_2\\ t_3
      \end{pmatrix}
      \eq
      \begin{pmatrix}
         1 & -\rho_{\epsilon_3}'(\theta_3) & -\rho_{\epsilon_2}'(\theta_2)\\
         -\rho_{\epsilon_3}'(\theta_3) & 1 & -\rho_{\epsilon_1}'(\theta_1)\\
         -\rho_{\epsilon_2}'(\theta_2) & -\rho_{\epsilon_1}'(\theta_1) & 1
      \end{pmatrix}
      \begin{pmatrix}
         \tau_{\epsilon_1}(r_1)/\tau_{\epsilon_1}'(d_1)\\
         \tau_{\epsilon_2}(r_2)/\tau_{\epsilon_2}'(d_2)\\
         \tau_{\epsilon_3}(r_3)/\tau_{\epsilon_3}'(d_3)
      \end{pmatrix}.
   \end{equation}
\end{lemma}
\begin{proof}
   Using the defining property of the face-vector and Lemma \ref{lemma:line_basis}
   we see
   \begin{equation}
      \label{eq:face_vector_by_edges}
      F \eq \sum_{n=1}^3
         \ip{-\frac{\tau_{\epsilon_n}(r_n)}{\tau_{\epsilon_n}'(d_n)}C_n}{F}\,L_n
      \eq\sum_{n=1}^3\frac{\tau_{\epsilon_n}(r_n)}{\tau_{\epsilon_n}'(d_n)}\,L_n.
   \end{equation}
   The result follows from Lemma \ref{lemma:angles_by_product}. Note that
   the sign of the generalised angles follows from our choice of $L_n$.
\end{proof}

\begin{definition}[tilts of a decorated hyperbolic triangle]
   \label{def:tilds}
   The $t_n$ in the Lemma \ref{lemma:tilds} above is called the
   \define{tilt} of the decorated hyperbolic triangle along the edge $L_n$.
\end{definition}

\begin{remark}
   \label{remark:geometric_interpretation_tild}
   The tilts have a special geometric meaning if $|F|^2 > -1$.
   In this case $F$ represents the unique hyperbolic cycle orthogonal to all
   vertex-cycles of the decorated triangle. Thus, by Lemma
   \ref{lemma:distance_from_product}, the tilt $t_n$ is given by the (oriented)
   distance of the centre of $F$ to $L_n$ divided by the radius of $F$. By Remark
   \ref{remark:geometric_interpretation_face_vector}, this distance is
   measured along the radical line belonging to the edge $L_n$.
\end{remark}

\begin{proposition}[local Delaunay condition]
   \label{prop:local_delaunay_condition}
   Given a decorated quadrilateral. Let $t_{\mathrm{left}}$ and
   $t_{\mathrm{right}}$ be the tilts along its diagonal relative to the two
   triangles constituting the quadrilateral. Then the interior edge is locally
   Delaunay iff its tilts satisfy
   \begin{equation}
      t_{\mathrm{left}} \plus t_{\mathrm{right}} \,\leq\, 0.
   \end{equation}
\end{proposition}
\begin{proof}
   This proof follows the ideas in \cite{Ushijima02}. Let $C_1,\dotsc,C_4\in\Sym(2)$
   denote the vertex cycles of the decorated
   quadrilateral such that $C_1$, $C_2$ and $C_3$ belong to the left and
   $C_2$, $C_3$ and $C_4$ to the right triangle. Since the subspace spanned
   by $C_2$ and $C_3$ corresponds to the common edge of the triangles there
   is an $L\in\Sym(2)$ with $|L|^2 = 1$ and $\ip{L}{C_2} = 0 = \ip{L}{C_3}$.
   We can normalise $L$ such that $\ip{L}{C_1} < 0$ and $\ip{L}{C_4} > 0$.
   It follows that $\{L, C_2, C_3\}$ is a basis of $\Sym(2)$. Using this
   basis we can represent the remaining vertex cycles as linear combinations:
   \begin{equation}
      C_1 \eq \alpha L \plus \alpha_2 C_2 \plus \alpha_3 C_3
      \qquad\text{and}\qquad
      C_4 \eq \tilde{\alpha} L\plus \tilde{\alpha}_2 C_2\plus \tilde{\alpha}_3 C_3.
   \end{equation}
   Note that $\tilde{\alpha}>0$ by our choice of $L$. Furthermore, we get the
   representation
   \begin{equation}
      F_{\mathrm{left}} \eq \beta L \plus \beta_2 C_2 \plus \beta_3 C_3
      \qquad\text{and}\qquad
      F_{\mathrm{right}}
         \eq \tilde{\beta} L \plus \tilde{\beta}_2 C_2 \plus \tilde{\beta}_3 C_3
   \end{equation}
   for the face-vectors $F_{\mathrm{left}}$ and $F_{\mathrm{right}}$ of
   the left and right triangle, respectively. By the defining property of the
   face-vectors we see that
   \(
      -1
      = \ip{C_i}{F_{\mathrm{right}}}
      = \tilde{\beta_2}\ip{C_i}{C_2} + \tilde{\beta_3}\ip{C_i}{C_3}
   \)
   for $i=2,3$. A similar equation holds for $\ip{C_i}{F_{\mathrm{left}}}$.
   Hence, we compute
   \begin{align}
      -1
      &\eq \ip{C_4}{F_{\mathrm{right}}}\\
      &\eq \tilde{\alpha}\tilde{\beta}
         \plus \tilde{\alpha}_2(\tilde{\beta}_2\ip{C_2}{C_2}
            + \tilde{\beta}_3\ip{C_2}{C_3})
         \plus \tilde{\alpha}_3(\tilde{\beta}_2\ip{C_3}{C_2}
            + \tilde{\beta}_3\ip{C_3}{C_3})\\
      &\eq \tilde{\alpha}\tilde{\beta} \minus\tilde{\alpha}_2 \minus\tilde{\alpha}_3.
   \end{align}
   Finally, we see that the local Delaunay condition
   \eqref{eq:local_delaunay_vectors} is given by
   \begin{equation}
      -1
      \,\geq\, \ip{C_4}{F_{\mathrm{left}}}
      \eq \tilde{\alpha}\beta \minus\tilde{\alpha}_2 \minus\tilde{\alpha}_3
      \eq \tilde{\alpha}(\beta-\tilde{\beta}) \minus1.
   \end{equation}
   Observing that $\beta=t_{\mathrm{left}}$ and $\tilde{\beta}=-t_{\mathrm{right}}$
   yields the claim.
\end{proof}

\begin{example}
   \label{ex:isosceles_triangle}
   Consider a decorated quadrilateral by gluing two copies of an
   isosceles decorated triangle along one of their legs, respectively.
   Note that by construction the vertex cycles at the base are of the
   same type and radius. Suppose that the radii satisfy the condition given in
   Corollary \ref{cor:radical_line_intersection}. By symmetry, the radical line
   of the cycles at the base and the altitude erected over the base of the isosceles
   triangle coincide. Hence, the tilt with respect to one leg, and thus both, is
   negative (see Remark \ref{remark:geometric_interpretation_tild} and Figure
   \ref{fig:isosceles_triangle}, left). It follows that the diagonal satisfies the
   local Delaunay condition.
\end{example}

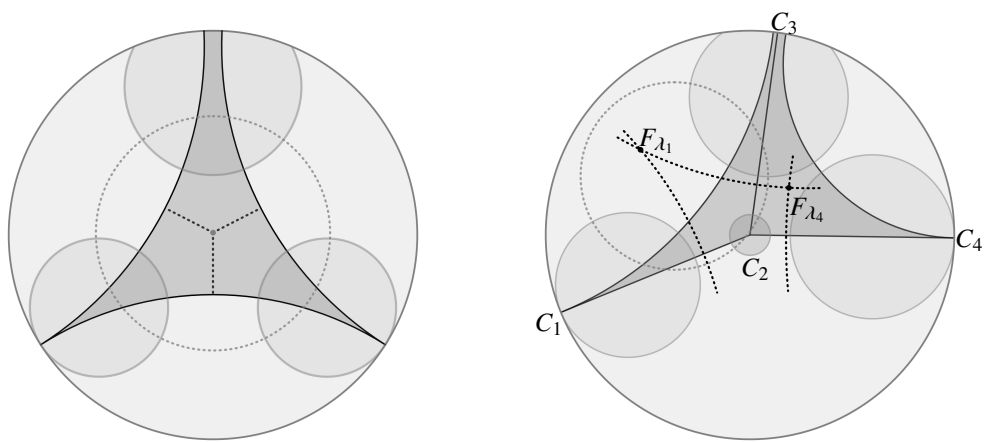
\begin{figure}[h]
   \begin{picture}(385, 160)
      \put(7, 0){
         \includegraphics[width=0.4\textwidth]{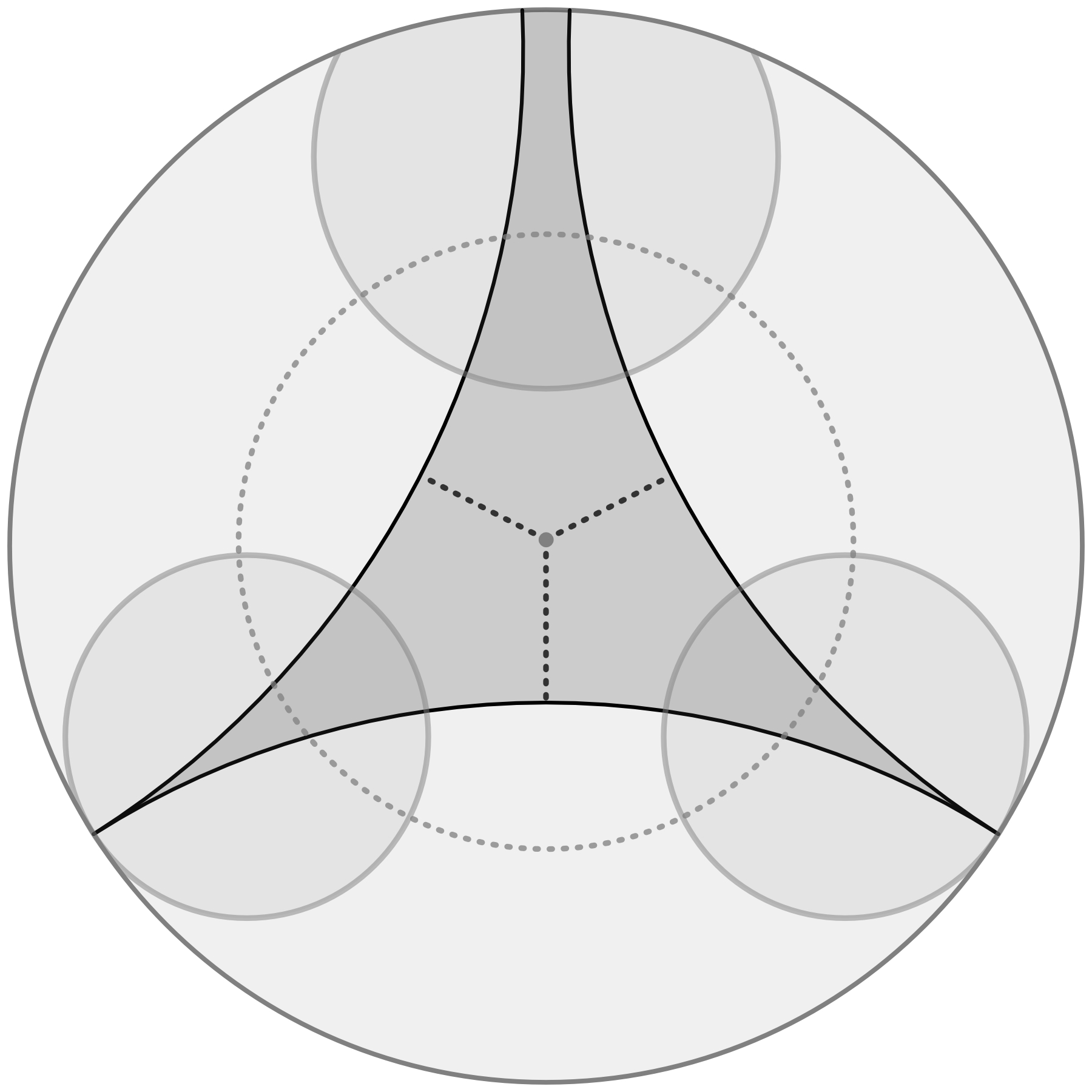}
      }
      \put(210, 0){
         \includegraphics[width=0.4\textwidth]{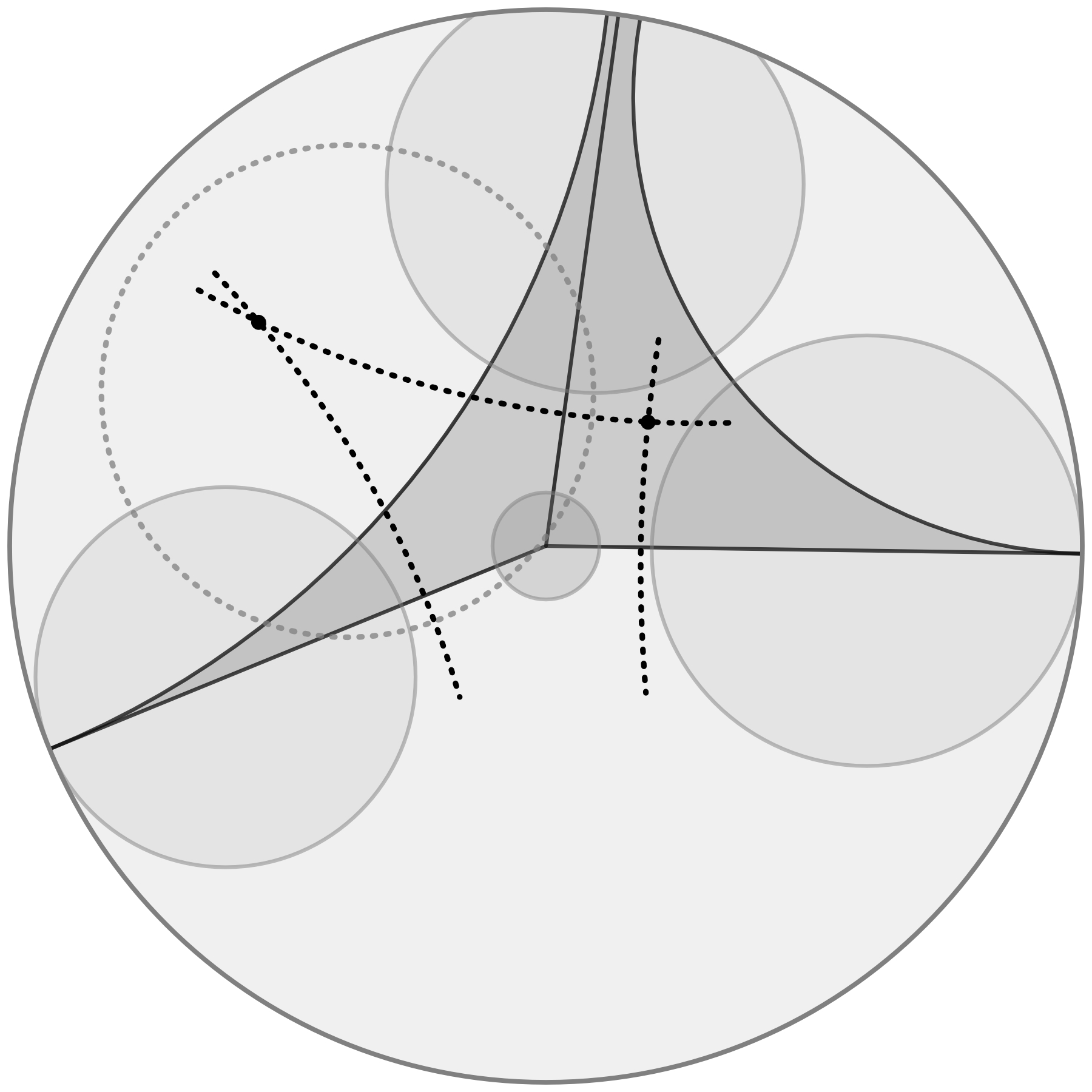}
      }
      \put(210, 42){\small $C_1$}
      \put(288, 63){\small $C_2$}
      \put(299, 156){\small $C_3$}
      \put(369, 74){\small $C_4$}
      \put(249, 112){\small $F_{\lambda_1}$}
      \put(306, 86){\small $F_{\lambda_4}$}
   \end{picture}
   \caption{\textsc{Left:} an isosceles hyperbolic triangle. The centre of the
      hyperbolic cycle which is orthogonal to all vertex cycles lies always to the
      left of the legs of the triangle. \textsc{Right:} a concave decorated
      quadrilateral with radical lines (dashed). Corollary
      \ref{cor:radical_line_intersection} provides conditions for the radical lines to
      intersect the interior of the edges. This guarantees that the diagonal
      satisfies the local Delaunay condition.
   }
   \label{fig:isosceles_triangle}
\end{figure}

\begin{proposition}
   \label{prop:flippable}
   Suppose a quadrilateral is decorated such that the radii satisfy the condition
   given in Corollary \ref{cor:radical_line_intersection}. Then it can always be
   triangulated such that its diagonal satisfies the local Delaunay condition.
   Equivalently, if the diagonal is not local Delaunay then it is flippable and the
   new diagonal is local Delaunay.
\end{proposition}
\begin{proof}
   The vertex cycles give rise to an affine tetrahedron in $\Sym(2)$.
   A triangulation of the decorated quadrilateral whose diagonal is local
   Delaunay corresponds to the lower convex hull of this tetrahedron.
   If the decorated quadrilateral is strictly convex then the lower and upper
   convex hulls project to the two possible triangulations of the quadrilateral.
   Hence one of them is locally Delaunay.

   Should the quadrilateral be not strictly convex then it possesses only one
   (geometric) triangulation. Let $C_1,\dotsc,C_4\in\Sym(2)$ denote the vertex cycles
   of the decorated quadrilateral as indicated in Figure
   \ref{fig:isosceles_triangle}, right. The dual pencil to the pencil spanned by
   $C_2$ and $C_3$ is given by
   \(
      F_{\lambda} \coloneq F_0 \,+\, \lambda\, L_{23},
   \)
   where $L_{23}$ is the line supporting the diagonal such that $\ip{L_{23}}{C_1}<0$
   and $F_0\in\Sym(2)$ satisfies $\ip{C_2}{F_0}=\ip{C_3}{F_0}=\ip{L_{23}}{F_0}=0$.
   There are $\lambda_n$, $n\in\{1,4\}$, such that $F_{\lambda_n}$ is the face-vector
   of the left or right triangle, respectively. Equivalently,
   $\ip{C_n}{F_{\lambda_n}}=-1$. It follows that $\ip{C_1}{F_{\lambda}}<-1$ for all
   $\lambda>\lambda_1$. Corollary \ref{cor:radical_line_intersection} grants that the
   radical lines intersect their corresponding edges. Furthermore, the diagonal of
   the triangulation is incident to a vertex of type $-1$ which has an interior
   angle sum $\geq\pi$. Thus, using Lemma \ref{cor:radical_line_intersection},
   we deduce $\lambda_2>\lambda_1$.
\end{proof}

\section{Decorated surfaces and their tessellations}
\label{sec:global}
\subsection{Decorated hyperbolic surfaces of finite type}
Let $\csurf$ be a closed orientable surface, that is, a closed orientable
$2$-manifold, and $\verts\subset\csurf$ a finite set of points partitioned into
$\verts_{-1}\cup\verts_{0}\cup\verts_{1}=\verts$. This partition determines a
\define{type} $\epsilon_v\in\{-1, 0, 1\}$ for each point $v\in\verts$. Note that
it is allowed for some $V_{\epsilon}$ to be empty. A complete path metric
$\dist_{\surf}$ on $\surf\coloneq\csurf\setminus(\verts_{0}\cup\verts_{1})$ is
\define{hyperbolic} if there is a cell-complex homeomorphic to $\csurf$ with
$0$-cells given by $V$ such that each open $2$-cell endowed with the restriction of
$\dist_{\surf}$ is isometric to a hyperbolic polygon whose vertices have the same
type as the corresponding points in $V$. For more information on path metrics we
refer the reader to \cites{BridsonHaefliger99,CooperEtAl00}. We call $\surf$
together with $\dist_{\surf}$ a \define{hyperbolic surface of finite type}, or short
\define{hyperbolic surface}. The restriction $\trunc(\surf)\subseteq\surf$ such
that each restricted $2$-cell is isometric to the corresponding truncated
hyperbolic polygon is the \define{truncation of $\surf$}.

The $1$-cells of the cell-complex above straighten to geodesics in $\surf$.
Therefore, we call it a \define{geodesic tessellation} of $\surf$. In general
there are infinitely many geodesic tessellations for a given hyperbolic surface.
If each $2$-cell is isometric to a hyperbolic triangle we call the
tessellation a \define{triangulation}. We also refer to the $0$-cells as
\define{vertices}, the $1$-cells as \define{edges} and the $2$-cells as
\define{faces}. Finally, a \define{decoration} of a hyperbolic surface is a
choice of decoration for each face such that it is consistent along the common
edges of each pair of neighbouring faces (more details are given in section
\ref{sec:config_space}). Note that a decoration is independent of the tessellation
since it can be completely described by the path metric $\dist_{\surf}$.

\begin{example}
   \label{example:fuchsian_group}
   Let $\Fgroup < \PSL(2;\R)$ be a finitely generated non-elementary
   Fuchsian group, i.e., it has a finite-sided fundamental domain (see
   \cite{Beardon83}*{\S 10.1}). The quotient $\surf\coloneq\HH/\Fgroup$ is a
   hyperbolic surface of finite type  (see Figure \ref{fig:surface_from_group}).
   A triangulation of $\surf$ can be obtained
   by triangulating the fundamental domain with hyperbolic triangles. Using the
   Beltrami-Klein model this is just triangulating a finite-sided convex
   polygon in the Euclidean sense. Identifying the sides of this Euclidean polygon
   according to the action of $\Fgroup$ we obtain a closed surface homeomorphic to
   $\csurf$. If we decorate these triangles consistently with
   the action of $\Fgroup$, i.e., identified vertices get cycles with the
   same radius, we obtain a decoration of $\surf$.
\end{example}

A decoration introduces about each vertex $v\in\verts$ of a hyperbolic surface
$\surf$ a closed curve: the \define{vertex cycle} $C_v$. These are special constant
curvature curves in $\surf$. A decorated hyperbolic surface is thus a pair
$(\surf, \{C_v\}_{v\in\verts})$. Furthermore, the notions introduced in the local
setting (section \ref{sec:hyperbolic_polygons_decorations}), e.g., centres, vertex
discs, radius, edge-length, etc., carry over to decorated hyperbolic surfaces. In
particular, we denote the disc associated with $C_v$ by $D_v$ and its radius by $r_v$.
The \define{weight-vector} $\omega\coloneq(\tau_{\epsilon_v}(r_v))_{v\in\verts}$
determines the decoration. We also write $\surf^{\omega}$ for the hyperbolic surface
$\surf$ decorated with $\omega$. Note that the centres of hypercycles, i.e.,
$v\in\verts_1$, are simple closed geodesics. In the following, if not mentioned
otherwise, we assume the auxiliary centre about a vertex $v\in\verts_0$ to be chosen
such that it does not intersect any other vertex cycle $C_{\tilde{v}}$. Slightly
abusing notation, we write $\dist_{\surf}(x, v)$ for the distance between a point
$x\in\surf$ and the (auxiliary) centre of the vertex cycle $C_v$. Similarly,
$\dist_{\surf}(v, \tilde{v})$ will be the smallest non-zero distance between
the centres of the vertex cycles $C_v$ and $C_{\tilde{v}}$.

\begin{lemma}
   \label{lemma:bounded_geodesics}
   Let $\surf^{\omega}$ be a decorated hyperbolic surface with non-intersecting
   vertex discs. For each pair of vertex cycles, say $C_{v_0}$ and $C_{v_1}$,
   and $L>0$ there is only a finite number of geodesic arcs which are
   orthogonal to both cycles and their length is $\leq L$.
\end{lemma}
\begin{proof}
   Denote by $\mathcal{A}$ the set of all constant speed parametrised arcs
   $\gamma\colon[0,L]\to\surf$ with $\gamma(0)\in C_{v_0}$ and $\gamma(1)\in C_{v_1}$.
   We endow $\mathcal{A}$ with the topology of uniform convergence induced by
   the path metric $\dist_{\surf}$. Furthermore, let $\mathcal{A}_L\subset\mathcal{A}$
   be the subset of all geodesic arcs orthogonal to $C_{v_0}$ and $C_{v_1}$ with
   length $\leq L$.

   The family $\mathcal{A}_L$ is equicontinuous and has bounded diameter.
   Indeed, $1$ is a uniform Lipschitz constant for $\mathcal{A}_L$. To see
   the boundedness of the diameter restrict $\surf$ as follows: for each vertex
   $v\in\verts_0\setminus\{v_0,v_1\}$ choose a horocycle $C_v$ with
   $\dist_{\surf}(C_v,C_{v_0}\cup C_{v_1}) > L$. Then denote
   by $\surf'\subseteq\surf$ the surface obtained by removing the horodiscs
   associated to the $C_v$ from the truncation $\trunc(\surf)$. The surface
   $\surf'$ is compact and contains the support of all arcs contained in
   $\mathcal{A}_L$. Thus the diameter of $\mathcal{A}_L$ is at most the
   diameter of $\surf'$.

   Using the Arzel\`a-Ascoli theorem we conclude that $\mathcal{A}_L$ is compact
   in $\mathcal{A}$. Finally, we observe that the elements of $\mathcal{A}_L$
   are isolated with respect to the uniform topology. Equivalently, because
   they are locally length minimising, each such geodesic arc possess a tubular
   neighbourhood which can not completely contain another element form
   $\mathcal{A}_L$. Hence, $\mathcal{A}_L$ has to be a finite set.
\end{proof}

\subsection{Weighted Delaunay tessellations}
\label{sec:weighted_delaunay_tesselations}
Assume for the rest of this section \ref{sec:weighted_delaunay_tesselations}
that $\surf^{\omega}$ is decorated with non-intersecting vertex discs. Define
$\surf' \coloneq \surf\setminus\bigcup_{v\in\verts}D_v$. By
assumption, $\surf'$ is a compact connected surface with $|V|$ boundary components.
A \define{properly immersed (circular) disc} $(\varphi, D)$ is a continuous map
$\varphi\colon\bar{D}\to\surf$, where $D\subset\HH$ is a circular
disc and $\bar{D}$ its closure, such that $\varphi\vert_{D}$ is an isometric
immersion, i.e., each point in $D$ possesses a neighbourhood which is mapped
isometrically, and the circle $\varphi(\partial D)$ intersects no $C_v$ more then
orthogonally. As in the local setting (section
\ref{sec:hyperbolic_polygons_decorations}) the intersection angle is understood to
be the interior intersection angle of the associated discs.

Let $N\geq2$ be a positive integer. Suppose there is a properly immersed disc
$(\varphi, D)$ such that $\varphi^{-1}(\bigcup_{v\in\verts'}\bar{D}_v)$ consists of
exactly $N$ connected components, where $v\in\verts'\subseteq\verts$ iff $C_v$
intersects $\varphi(\partial D)$ orthogonally. To each of these connected
components corresponds a hyperbolic cycle in $\HH$ because $\varphi$ is isometric.
We refer to them as the vertex cycles pulled back by $\varphi$. Denote them by
$C_1,\dotsc,C_N$. Then we call $\surf'\cap\varphi(\poly(C_1,\dotsc,C_N)\cap D)$
a \define{truncated $N$-vertex cell} and $\varphi\vert_{\poly(C_1,\dotsc,C_N)\cap D}$
is its \define{attachment map}. Note that $\poly(C_1,\dots,C_N)$ is well defined
by Proposition \ref{prop:associated_polygon} and our assumption about the decoration.

\begin{definition}[weighted Delaunay tessellation, non-intersecting vertex discs]
   \label{def:weighted_Delaunay_tessellation}
   Let $\surf^{\omega}$ be a decorated hyperbolic surface with non-intersecting
   vertex discs. Geodesically extending the truncated vertex cells into the
   vertex discs defines a geodesic tessellation of $\surf$. It is
   called the \define{weighted Delaunay tessellation of $\surf^{\omega}$}. We refer to
   the extended truncated $2$-vertex cells as \define{Delaunay $1$-cells} and
   to the extended truncated $N$-vertex cells with $N\geq3$ as
   \define{Delaunay $2$-cells}.
\end{definition}

\begin{remark}
   The assumption about non-intersecting vertex discs is important to ensure the
   existence of properly immersed discs. For a surface without cone-points, i.e.,
   $\verts_{-1}=\emptyset$, this poses no real loss of generality. In these cases,
   we can consider the rescaled weights $s\omega$, where $s$ is a small positive
   scalar. If $s$ is small enough, it follows that $s\omega$ induces non-intersecting
   vertex discs and Corollary \ref{lemma:scalability} will guaranty that $\omega$
   and $s\omega$ induce the same weighted Delaunay tessellation. We note that this
   observation was already utilised in the classical Epstein-Penner convex hull
   construction.
\end{remark}

\begin{theorem}[well-definedness of weighted Delaunay tessellations]
   \label{theorem:weighted_Delaunay_tessellations}
   The weighted Delaunay tessellation defined in Definition
   \ref{def:weighted_Delaunay_tessellation} is a geodesic tessellation of $\surf$.
\end{theorem}
\begin{proof}
   First, we observe that the truncated $2$-vertex cells are geodesic segments
   since they are isometric images of hyperbolic segments in $\HH$.
   Each such truncated $2$-vertex cell intersects its two vertex cycles orthogonally.
   Hence, we can geodesically extend the segments into the corresponding vertex
   discs. Lemma \ref{lemma:delaunay_1_cells} shows that truncated $2$-vertex
   cells do not intersect, so lifted to $\csurf$ they form an embedded
   $1$-dimensional cell-complex with vertex set $\verts$.

   Now, Lemma \ref{lemma:delaunay_n_cells} grants that the interiors of truncated
   $N$-vertex cells, $N\geq3$, are homeomorphic to open discs. Their boundary is
   mapped into the union of the vertex cycles with the truncated $2$-vertex cells.
   Furthermore, the truncated $2$-vertex cells do not intersect the interiors of
   the truncated $N$-vertex cells. Thus, the truncated $N$-vertex cells can be
   extended into the vertex discs along with the truncated $2$-vertex cells. All
   ideas needed to prove these assertions are presented in the rest of this section.
   We omit further details.

   Finally, by Lemma \ref{lemma:delaunay_cover}, we see that every point of
   $\csurf$ is either contained in $\verts$ or in the geodesic extension of
   a truncated vertex cell.
\end{proof}

\begin{lemma}
   \label{lemma:delaunay_n_cells}
   The interiors of truncated $N$-vertex cells with $N\geq3$ are
   homeomorphic to open discs.
\end{lemma}
\begin{proof}
   Let $(\varphi, D)$ be the properly immersed disc used to define the
   truncated $N$-vertex cell and $C_1,\dotsc,C_N$ the pulled back vertex cycles.
   Define $P\coloneq\poly(C_1,\dotsc,C_N)$. We show that
   $\varphi\vert_{\interior(P)\cap D}$ is injective, thus a homeomorphism
   onto its image.

   To this end suppose there is $q\in\interior P$ and $\tilde{q}\in D$ with
   $q\neq\tilde{q}$ and $\varphi(q)=\varphi(\tilde{q})$. Since $\varphi\vert_D$
   is isometric, there is a neighbourhood $U\subset D$ of $q$ and a hyperbolic
   motion $M$ such that $M(q) = \tilde{q}$ and $M(U)\subset D$.
   Define $\tilde{D}\coloneq M(D)$ and $\tilde{\varphi}\coloneq\varphi\circ M^{-1}$.
   Then $(\tilde{\varphi}, \tilde{D})$ is a properly immersed disc and
   $\varphi\vert_{D\cap\tilde{D}}=\tilde{\varphi}\vert_{D\cap\tilde{D}}$
   (see Figure \ref{fig:delaunay_2_cells}).

   Now, let $\tilde{C}_n\coloneq M(C_n)$, $n=1,\dotsc,N$. Clearly,
   $\poly(\tilde{C}_1,\dotsc,\tilde{C}_N) =  M(P) \eqcolon \tilde{P}$.
   Since $\partial D$ and $\partial\tilde{D}$ are mirror symmetric
   about the unique hyperbolic line through their two intersection points, there is
   a representative $L\in\Sym(2)$ of this line such that $\ip{L}{\partial D} < 0$ and
   $\ip{L}{\partial\tilde{D}}>0$. The $\tilde{C}_n$ intersect $\partial D$ less then
   orthogonally, as $(\varphi, D)$ is proper, whilst they intersect
   $\partial\tilde{D}$ orthogonally, by construction. It follows that
   $\ip{L}{\tilde{C}_n}>0$. Similarly
   we see that $\ip{L}{C_n}<0$. Hence $P\cap\tilde{P} = \emptyset$. The assertion
   follows from observing that $\tilde{q}\in\interior\tilde{P}$.
\end{proof}

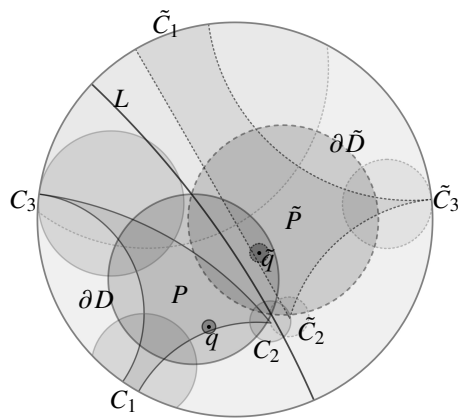
\begin{figure}[t]
   \begin{picture}(385, 160)
      \put(115, 0){
         \includegraphics[width=0.4\textwidth]{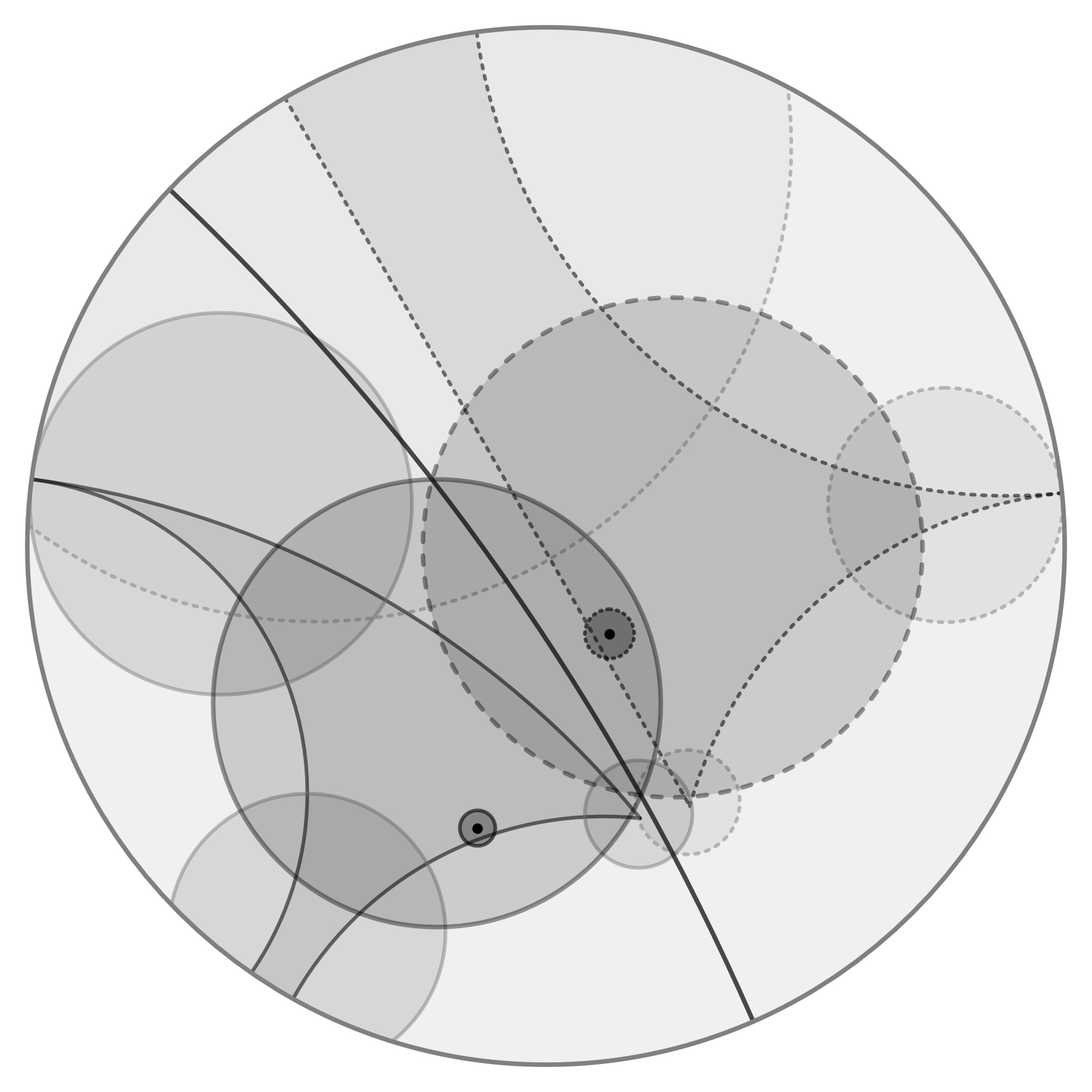}
      }
      \put(185, 31){\small $q$}
      \put(149, 7){\small $C_1$}
      \put(203, 26){\small $C_2$}
      \put(111, 83){\small $C_3$}
      \put(137, 45){\small $\partial D$}
      \put(172, 47){\small $P$}

      \put(206, 61){\small $\tilde{q}$}
      \put(165, 149){\small $\tilde{C}_1$}
      \put(220, 33){\small $\tilde{C}_2$}
      \put(271, 83){\small $\tilde{C}_3$}
      \put(232, 104){\small $\partial\tilde{D}$}
      \put(215, 75){\small $\tilde{P}$}

      \put(150, 121){\small $L$}
   \end{picture}
   \caption{A visualisation of the construction described in Lemma
      \ref{lemma:delaunay_n_cells}. Observe that the transformed cycles $\tilde{C}_n$
      intersect the circle $\partial\tilde{D}$ (dashed) orthogonally whilst
      intersecting $\partial D$ (solid) less the orthogonally, if at all.
   }
   \label{fig:delaunay_2_cells}
\end{figure}

\begin{lemma}
   \label{lemma:delaunay_1_cells}
   Two distinct truncated $2$-vertex cells do not cross each other
   or themselves.
\end{lemma}
\begin{proof}
   Let $e$ and $\tilde{e}$ be two distinct truncated $2$-vertex cells with
   attachment maps $\phi$ and $\tilde{\phi}$, respectively. Furthermore, let
   $(\varphi,D)$ and $(\tilde{\varphi}, \tilde{D})$ be the properly immersed
   discs used to define them. Towards a contradiction suppose that there are
   $q\in\phi^{-1}(e)$ and $\tilde{q}\in\tilde{\phi}^{-1}(\tilde{e})$ such that
   $\phi(q) = \tilde{\phi}(\tilde{q})$.
   Using the same argument as in Lemma \ref{lemma:delaunay_n_cells},
   we can assume that $q=\tilde{q}$ and
   $\varphi\vert_{D\cap\tilde{D}}=\tilde{\varphi}\vert_{D\cap\tilde{D}}$.
   On one hand, the vertex cycles pulled back by $\varphi$ and $\tilde{\varphi}$
   intersect $\partial\tilde{D}$ or $\partial D$ less then orthogonally, respectively.
   Hence $D\neq\tilde{D}$. On the other hand, the pulled back vertex cycles define a
   decorated hyperbolic quadrilateral. Its diagonals are given by $\phi^{-1}(e)$
   and $\tilde{\phi}^{-1}(\tilde{e})$. Since $\partial D$ intersects the vertex
   cycles pulled back by $\tilde{\varphi}$ less then orthogonal, the diagonal
   $\phi^{-1}(e)$ is local Delaunay. The same argument applies to
   $\tilde{\phi}^{-1}(\tilde{e})$, implying $D=\tilde{D}$.
\end{proof}

\begin{lemma}
   \label{lemma:delaunay_cover}
   The surface $\surf'$ is covered by the truncated cells.
\end{lemma}
\begin{proof}
   Consider a properly immersed disc $(\varphi, D)$. Let $(c, r)\in\surf'\times\R_{>0}$
   such that $\varphi^{-1}(c)$ is the centre and $r$ the radius of $D$. Then
   $(c, r)$ determines $(\varphi, D)$ up to hyperbolic motions. Utilising Lemma
   \ref{lemma:cos_sin_laws}, we see that the closure of the configuration space of
   properly immersed discs, up to hyperbolic motions, is given by
   \[
      \bar{\mathcal{D}}
      \coloneq
      \Big\{(c, r)\in\surf'\times\R_{\geq0}\,:\,
         \cosh(r) \,\leq\,
         \min_{v\in\verts}\tandist_c(C_v)
      \Big\}.
   \]
   Here the modified tangent distance $\tandist$ on $\surf$ is defined by
   replacing $\dist_{\HH}$ by $\dist_{\surf}$ in Definition
   \ref{def:modified_tangent_distance}. The configuration space $\bar{\mathcal{D}}$
   is a compact $3$-dimensional manifold with boundary. If $(\varphi, D)$ is a
   properly immersed disc corresponding to $(c,r)\in\bar{\mathcal{D}}$ then
   $\tandist_x(c,r)$ is greater, equal or smaller then $1$
   iff $x\in\surf'\setminus\varphi(\bar{D})$, $x\in\varphi(\partial D)$ or
   $x\in\varphi(D)$, respectively. Moreover, for each fixed $x\in\surf'$ the tangent
   distance $\tandist_x$ is a continuous function over $\bar{\mathcal{D}}$. Hence,
   as $\bar{\mathcal{D}}$ is compact, $\tandist_x$ attains a minimum at some
   $(c_{\mathrm{min}}, r_{\mathrm{min}})\in\bar{\mathcal{D}}$. Clearly some
   properly immersed disc contains $x$. Thus $r_{\mathrm{min}} > 0$,
   implying that there is a properly immersed disc
   $(\varphi_{\mathrm{min}}, D_{\mathrm{min}})$ with centre
   $c_{\mathrm{min}}$ and radius $r_{\mathrm{min}}$ such that
   $x\in\varphi_{\mathrm{min}}(D_{\mathrm{min}})$. Let
   $C_1,\dotsc,C_N$ be the vertex cycles pulled back by $\varphi_{\mathrm{min}}$.
   Considering the degrees of freedom of a properly
   immersed disc it follows that $N\geq3$.

   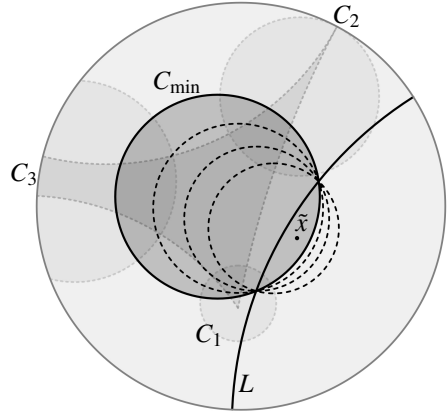
\begin{figure}[t]
      \begin{picture}(385, 155)
         \put(115, 0){
            \includegraphics[width=0.4\textwidth]{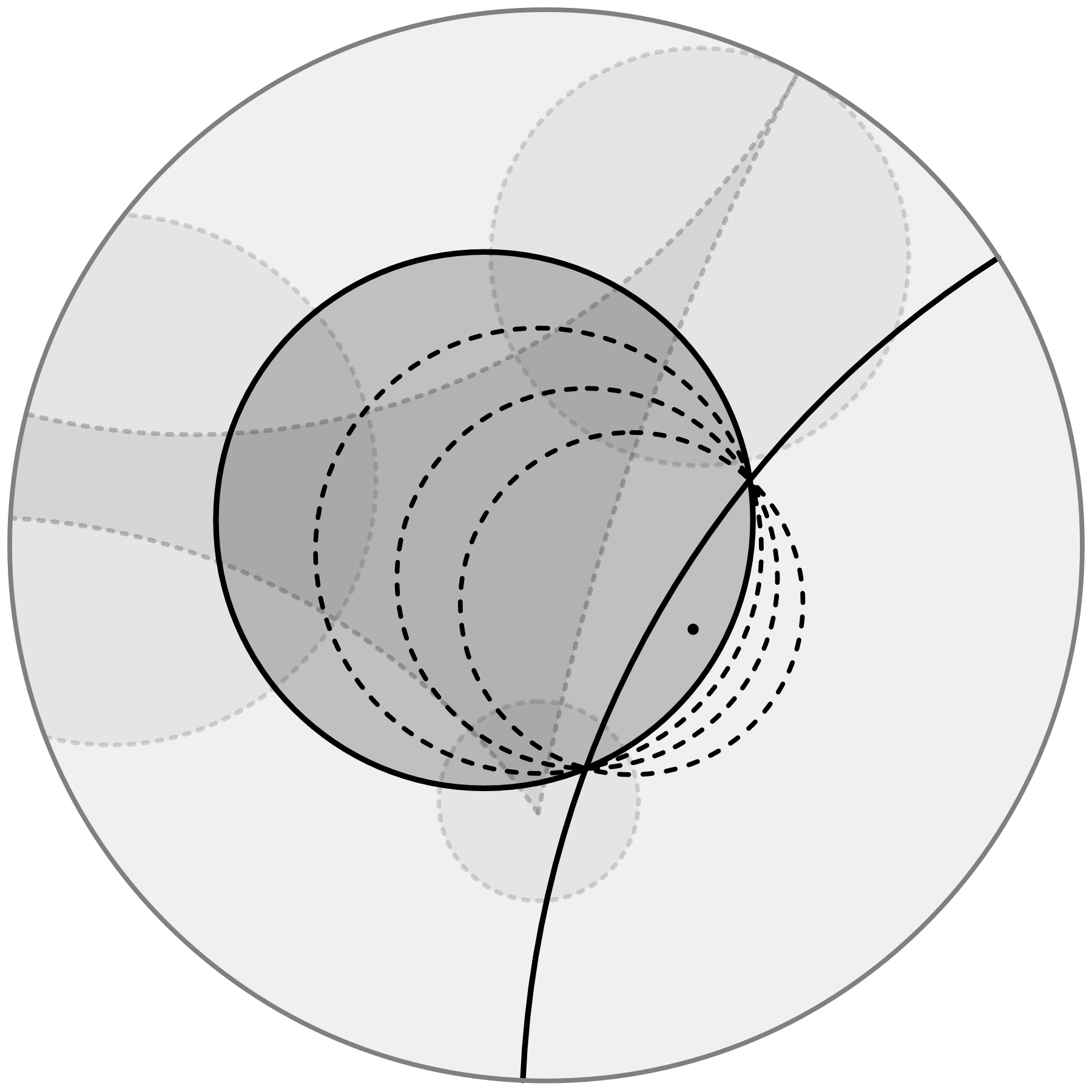}
         }

         \put(179, 27){\small $C_1$}
         \put(231, 148){\small $C_2$}
         \put(109, 88){\small $C_3$}

         \put(163, 123){\small $C_{\mathrm{min}}$}
         \put(217, 69){\small $\tilde{x}$}
         \put(195, 8){\small $L$}
      \end{picture}
      \caption{Sketch of the construction used in Lemma \ref{lemma:delaunay_cover}:
         small perturbations of $C_{\mathrm{min}}$ (solid black) in the pencil of
         cycles spanned by $C_{\mathrm{min}}$ and $L$ are still hyperbolic circles
         (dashed black). Since $C_{\mathrm{min}}$ and $\tilde{x}$ lie on
         opposite sides of $L$, some of these perturbed hyperbolic circles have
         a smaller modified tangent distance to $\tilde{x}$ then $C_{\mathrm{min}}$
         whilst still being at most orthogonal to all $C_n$.
      }
      \label{fig:delaunay_cover}
   \end{figure}

   We prove by contradiction that
   $x\in\varphi_{\mathrm{min}}(\poly(C_1,\dotsc,C_N)\cap D_{\mathrm{min}})$: denote by
   $\tilde{c}_{\mathrm{min}}\coloneq\varphi_{\mathrm{min}}^{-1}(c_{\mathrm{min}})\in\HH$
   the centre of $D_{\mathrm{min}}$. There is a $\tilde{x}\in\varphi^{-1}(x)$ such that
   $\dist_{\HH}(\tilde{c}_{\mathrm{min}},\tilde{x})=\dist_{\surf}(c_{\mathrm{min}},x)$.
   Suppose $\tilde{x}\notin\poly(C_1,\dotsc,C_N)$. Then we find a hyperbolic line
   separating $\tilde{x}$ from $\poly(C_1,\dotsc,C_N)$. Choose a representative
   $L\in\Sym(2)$ of this line with $\ip{L}{\tilde{x}} > 0$ and $\ip{L}{C_n} < 0$,
   $n=1,\dotsc,N$. The circle $C_{\mathrm{min}}\coloneq \partial D_{\mathrm{min}}$
   and line $L$ span a pencil of hyperbolic cycles given by
   \[
      C^{\lambda}
      \;\coloneq\; C_{\mathrm{min}} \plus \lambda L.
   \]
   Note that $C_{\mathrm{min}} = C^0$. By continuity, $C^{\lambda}$
   represents a hyperbolic circle for small $0\leq \lambda$
   (see Figure \ref{fig:delaunay_cover}). Furthermore, we have
   $\ip{C^{\lambda}}{C_n} < -1$, $n=1,\dotsc,N$, and
   $\ip{C^{\lambda}}{\tilde{x}} > \ip{C_{\mathrm{min}}}{\tilde{x}}$. It follows
   that if $\lambda$ is small enough there is a properly immersed disc
   $(\varphi^{\lambda}, D^{\lambda})$ with $C^{\lambda} = \partial D_{\lambda}$.
   Denote by $c^{\lambda}\in\surf'$ its centre and by $r^{\lambda}>0$ its radius.
   Using Lemma \ref{lemma:distance_from_product}, we observe that
   $\ip{C^{\lambda}}{\tilde{x}} = -\tandist_x(c^{\lambda}, r^{\lambda})$. But
   this leads to
   \[
      \tandist_x(c^{\lambda}, r^{\lambda})
      \,<\, \tandist_{\tilde{x}}(C_{\mathrm{min}})
      = \tandist_x(c_{\mathrm{min}}, r_{\mathrm{min}})
   \]
   contradicting the assumption that $(c_{\mathrm{min}}, r_{\mathrm{min}})$
   is a minimum point of $\tandist_x$.
\end{proof}

\subsection{Weighted Voronoi decompositions}
\label{sec:laguerre_voronoi_decompositions}
A decoration of a hyperbolic surface $\surf$ is called \define{proper} if for all
$(v,\tilde{v})\in\verts\times\verts_{-1}$ holds
\(
   \tau_{\epsilon_v}(r_v)/\cosh(r_{\tilde{v}})
   <\tau_{\epsilon_v}(\dist_{\surf}(v,\tilde{v})).
\)
For a point $x\in\surf$ of a properly decorated hyperbolic surface $\surf$ and a
vertex cycle $C_v$ there might be multiple geodesic arcs realising the modified
tangent distance $\tandist_x(C_v)$. By $\mathfrak{m}_x(v, r_v)$ we will denote the
number of such arcs. Note that always $\mathfrak{m}_x(v, r_v)\geq1$. We define
\[
   \mathfrak{M}_x
   \;\coloneq\;
   \sum_{v\in{\argmin\limits_{\tilde{v}\in\verts}\tandist_x(C_{\tilde{v}})}}
   \mathfrak{m}_x(v, r_v).
\]

\begin{definition}[weighted Voronoi decomposition]
   \label{def:laguerre_voronoi_decomposition}
   Let $\surf^{\omega}$ be a properly decorated hyperbolic surface. The
   \define{weighted Voronoi decomposition of $\surf^{\omega}$} is defined
   in the following way: define $\mathcal{V}_{-1}\coloneq\emptyset$ and
   $\mathcal{V}_n\coloneq\{x\in\surf\,:\,\mathfrak{M}_x\geq3-n\}$, $n=0,1,2$.
   The \define{(open) Voronoi $n$-cells}, $n=0,1,2$, are the
   connected components of $\mathcal{V}_n\setminus\mathcal{V}_{n-1}$.
   The \define{attachment maps} are given by inclusion.
\end{definition}

\begin{theorem}[properties of the weighted Voronoi decomposition]
   \label{theorem:properties_laguerre_voronoi}
   To the weighted Voronoi decomposition of $\surf^{\omega}$ corresponds a
   cell-complex of $\csurf$ such that each $2$-cell contains exactly one of
   the points of $\verts$ in its interior. In particular, the Voronoi
   $0$- and $1$-cells form a $1$-dimensional cell-complex which is geodesically
   embedded into $\surf$.
\end{theorem}
\begin{proof}
   It is clear from the definition that the (open) Voronoi cells
   partition $\surf$. Lemma \ref{lemma:laguerre_0_cell} shows that
   Voronoi $0$-cells are points. By a similar construction, we see
   that the interiors of Voronoi $1$-cells are isomorphic to open
   hyperbolic segments. Finally, Lemma \ref{lemma:laguerre_2_cell} grants that
   to each open Voronoi $2$-cell $P_v$ corresponds exactly one
   $v\in\verts$ such that $P_v\cup\{v\}$ is an open disc in $\csurf$.
\end{proof}

\begin{lemma}
   \label{lemma:laguerre_2_cell}
   For each open Voronoi $2$-cell there is a $v\in\verts$ such that it is
   given by
   \begin{equation}
      \label{eq:laguerre_2_cell}
      P_v
      \;\coloneq\;
      \Big\{
         x\in\surf
         \;:\;
         \tandist_x(C_v)
         < \min_{\tilde{v}\in\verts\setminus\{v\}}\tandist_x(C_{\tilde{v}})
         \;\text{ and }\;
         \mathfrak{m}_x(v, r_v) = 1
      \Big\}.
   \end{equation}
   Conversely, for each $v\in\verts$ there is a neighbourhood $U_v\subset\csurf$
   of $v$ such that $U_v\setminus\{v\}\subset P_v$. In particular,
   $P_v\setminus\verts$ is homeomorphic to a punctured disc.
\end{lemma}
\begin{proof}
   Equation \eqref{eq:laguerre_2_cell} is a direct reformulation of the definition
   of Voronoi $2$-cells. Now, let $v\in\verts$ and define
   $f_v\colon D_v\to \R$ by $f_v(x)\coloneq\dist_{\surf}(x, C_v)$. Since $\surf$
   is properly decorated, we conclude that there is some $R_v>0$
   such that $f_v^{-1}(\{t\geq R_v\})\subset P_v$ (see Corollary
   \ref{cor:radical_line_intersection}). This shows that $P_v$ is not
   empty and $U_v\coloneq f_v^{-1}(\{t>R_v\})\cup\{v\}$ is the required
   neighbourhood of $v$.

   Its left to show that $P_v\setminus\verts$ is homeomorphic to a punctured disc
   (see Figure \ref{fig:laguerre_voronoi_2_cell}). The previous considerations and
   $P_v\cap P_{\tilde{v}}=\emptyset$ for $v\neq\tilde{v}$ show that
   $P_v\setminus\verts = P_v\setminus\{v\}$. For large enough $R_v$ the set
   $S_v\coloneq f_v^{-1}(\{R_v\})$ is an embedding of the topological circle
   $\SS^1$ into $\surf$. Indeed, the $S_v$ are vertex cycles which can be chosen
   such that they do not intersect each other or themselves. Denote by
   $\gamma_p^v\colon[0,L_p)\to\surf\setminus U_v$ the arc-length parametrised
   geodesic arc orthogonal to $S_v$ with $\gamma_p^v(0) = p\in S_v$ emitting into
   $\surf\setminus U_v$. Here $L_p^v\in\R_{>0}\cup\{\infty\}$ is either the smallest
   number such that $\gamma_p^v(L_p^v)\notin P_v$ or $L_p^v=\infty$. Since, by
   construction, $S_v$ has constant distance to the centre of $C_v$ and
   $\gamma_p^v([0, L_p^v))\subset P_v$, all $\gamma_p$ are distance minimising. It
   follows that $\gamma_p^v$ can not cross each other or themselves. Finally,
   $L_p^v<\infty$ for all $p\in S_v$. This follows form
   $\surf'\coloneq\surf\setminus\bigcup_{v\in\verts}U_v$ being compact and
   $\gamma_p^v([0,L_p^v))\subset\surf'$. Thus,
   $L_p^v = \sup_{t\in[0,L_p^v)}\dist_{\surf}(\gamma_p^v(t), S_v) < \infty$.
\end{proof}

\begin{figure}[h]
   \includegraphics[width=0.4\textwidth]{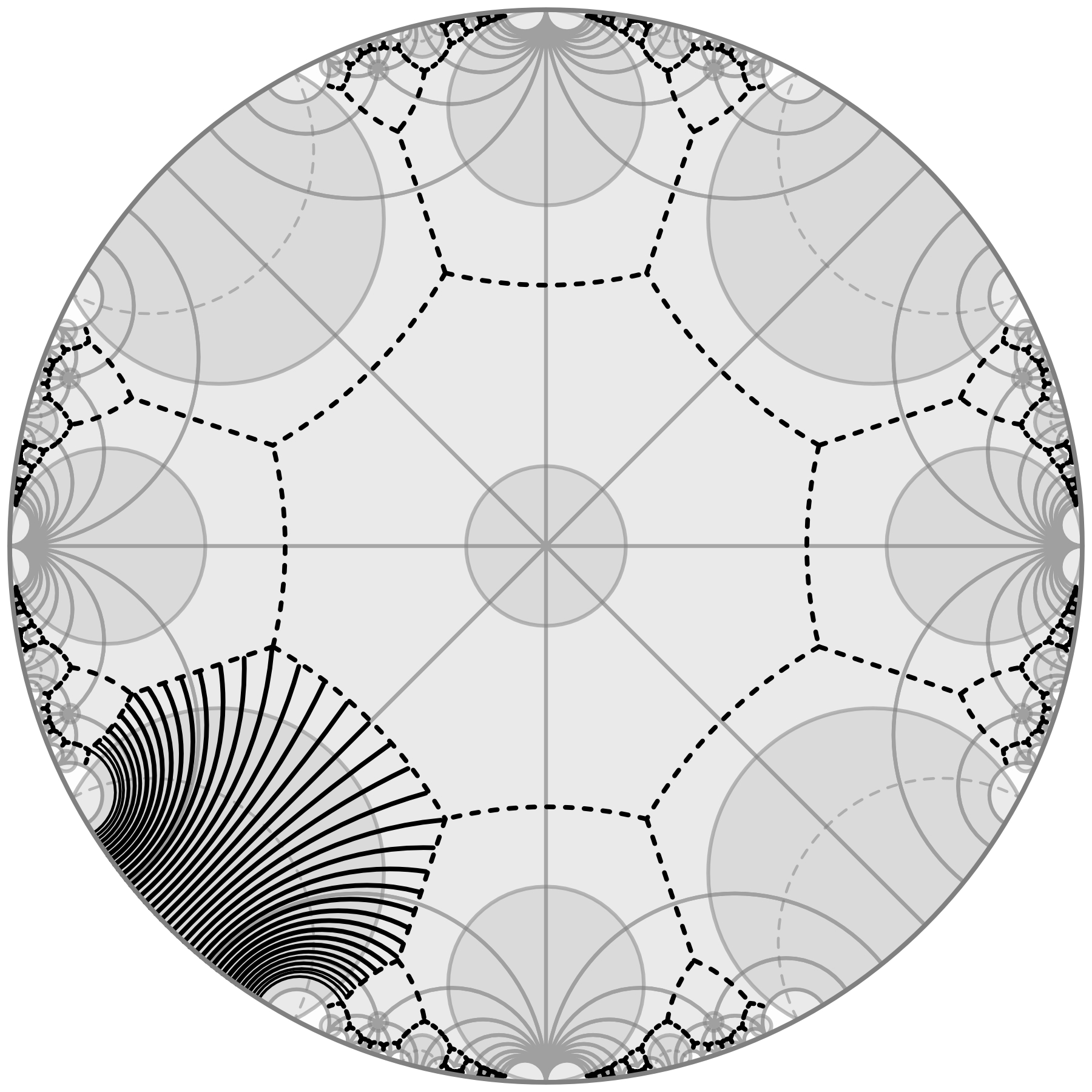}
   \caption{Depicted is the covering of the decorated hyperbolic surface discussed in
      Figure \ref{fig:surface_from_group} and the $1$-skeleton of its weighted Voronoi
      decomposition (dashed lines). Each Voronoi $2$-cell $P_v$
      without $\verts$ is homeomorphic to a once punctured disc. This can be seen
      by considering the \enquote{polar coordinates} $(p,t)\mapsto\gamma_p^v(t)$
      introduced in Lemma \ref{lemma:laguerre_2_cell}. The solid black segments are
      a sampling of their \enquote{radial coordinate lines}.
   }
   \label{fig:laguerre_voronoi_2_cell}
\end{figure}

\begin{lemma}
   \label{lemma:laguerre_0_cell}
   There is only a finite number of Voronoi $0$-cells each of which is
   a point.
\end{lemma}
\begin{proof}
   Let $p\in\surf$ be a point contained in a Voronoi $0$-cell. We can
   find a circular disc $D\subset\HH$ and an isometry $\varphi\colon D\to\surf$
   such that $\varphi^{-1}(p)\eqcolon\tilde{p}$ is the centre of $D$. By definition,
   there are $N\geq3$ geodesic arcs $\gamma_n\colon(0,1)\to\surf$ corresponding to
   the minimisers of $\tandist_p$ in $\{C_v\}_{v\in\verts}$. Suppose they are
   enumerated in counter-clockwise direction. These geodesic arcs are pulled back by
   $\varphi$ to the intersections of $D$ with hyperbolic rays starting at $\tilde{p}$
   (see Figure \ref{fig:delaunay_voronoi_0_cell}). Let $v_n\in\verts$ be the endpoint
   of $\gamma_n$. Then there is a hyperbolic cycle $C_n$ of type $\epsilon_{v_n}$ and
   radius $r_{v_n}$ on the ray corresponding to
   $\gamma_n$ such that $\tandist_{p}(C_{v_n}) = \tandist_{\tilde{p}}(C_n)$. Maybe
   after choosing a smaller disc $D$, it follows that
   $\tandist_{\varphi(x)}(C_{v_n}) = \tandist_{x}(C_n)$ for all $x\in D$ and
   $n=1,\dotsc,N$, because $\dist_{\surf}$ is continuous. Hence, $\tilde{p}$ is
   the common intersection point of the radical lines of consecutive cycles
   $C_n$ (see Lemma \ref{lemma:formula_radical_line}). Due to the requirement of
   properness of the decoration consecutive radical lines can not coincide. This
   shows that the set of Voronoi $0$-cells consists of isolated points.
   Observing that all $0$-cells have to be contained in the compact set $\surf'$
   constructed in Lemma \ref{lemma:laguerre_2_cell}, we see that there are only a
   finite number of Voronoi $0$-cells.
\end{proof}

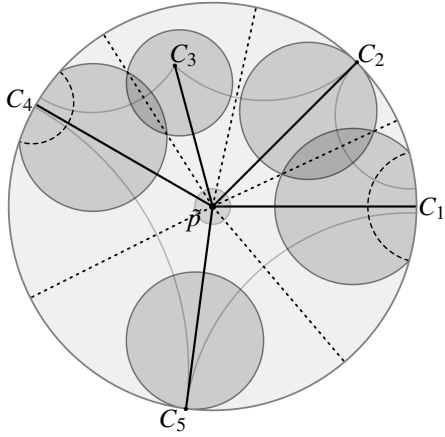
\begin{figure}[h]
   \begin{picture}(385, 155)
      \put(115, 0){
         \includegraphics[width=0.4\textwidth]{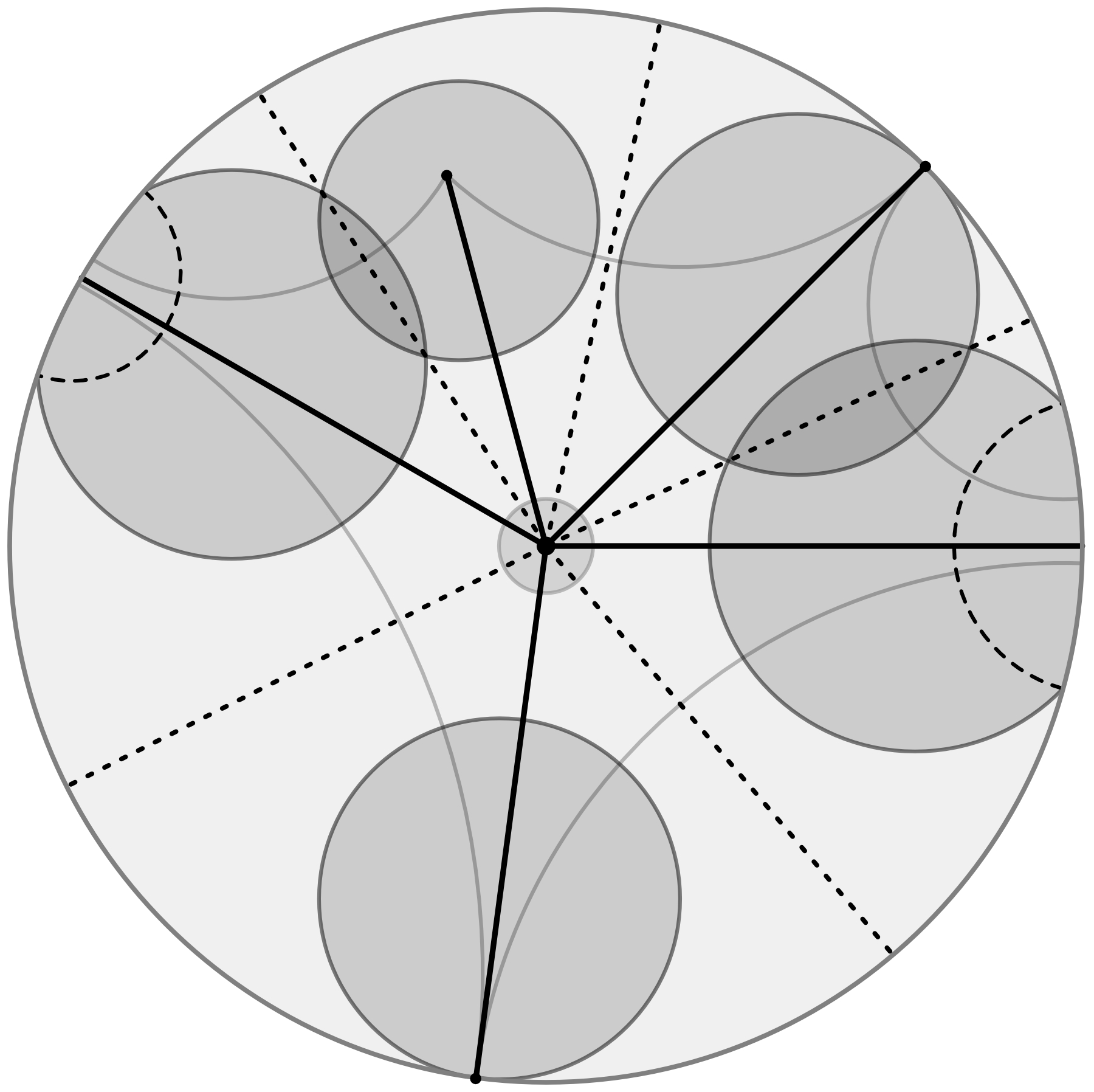}
      }
      \put(186, 71){\small $\tilde{p}$}
      \put(274, 75){\small $C_1$}
      \put(251, 133){\small $C_2$}
      \put(180, 133){\small $C_3$}
      \put(118, 116){\small $C_4$}
      \put(176, -5){\small $C_5$}
   \end{picture}
   \caption{Using a small isometrically embedded disc $(\varphi, D)$ about a
      point $p\in\surf$ contained in a Voronoi $0$-cell, we can pull
      back the $\tandist$-minimising cycles along hyperbolic rays (solid black)
      corresponding to the geodesic arcs in $\surf$. It follows that $p$ is
      an isolated point since $\tilde{p}\coloneq\varphi^{-1}(p)$ is the common
      intersection point of the radical lines (dashed black) of consecutive cycles.
   }
   \label{fig:delaunay_voronoi_0_cell}
\end{figure}

\begin{theorem}[dual tessellation of weighted Voronoi decomposition]
   \label{theorem:weighted_Delaunay_tessellations_revisited}
   Let $\surf^{\omega}$ be a properly decorated hyperbolic surface. The combinatorial
   dual of its weighted Voronoi decomposition can be realised as a geodesic
   tessellation of $\surf$. This realisation exhibits the following properties:
   \begin{enumerate}[label=(\roman*)]
      \item\label{item:laguerre_to_delaunay}
         If the vertex discs of the decoration do not intersect, the realisation is
         given by the weighted Delaunay tessellation of $\surf^{\omega}$ (Definition
         \ref{def:weighted_Delaunay_tessellation}). In particular, the Voronoi
         $0$-cells are the centres of the properly immersed discs defining the
         Delaunay $2$-cells.
      \item\label{item:laguerre_local_delaunay}
         All edges of a geodesic triangulation refining the realisation satisfy the
         local Delaunay condition (Definition \ref{def:local_delaunay_condition}). In
         particular, an edge satisfies the strict local Delaunay condition iff it
         already was an edge of the weighted Delaunay tessellation.
   \end{enumerate}
\end{theorem}

\begin{figure}[t]
   \begin{picture}(385, 155)
      \put(12, 0){
         \includegraphics[width=0.4\textwidth]
            {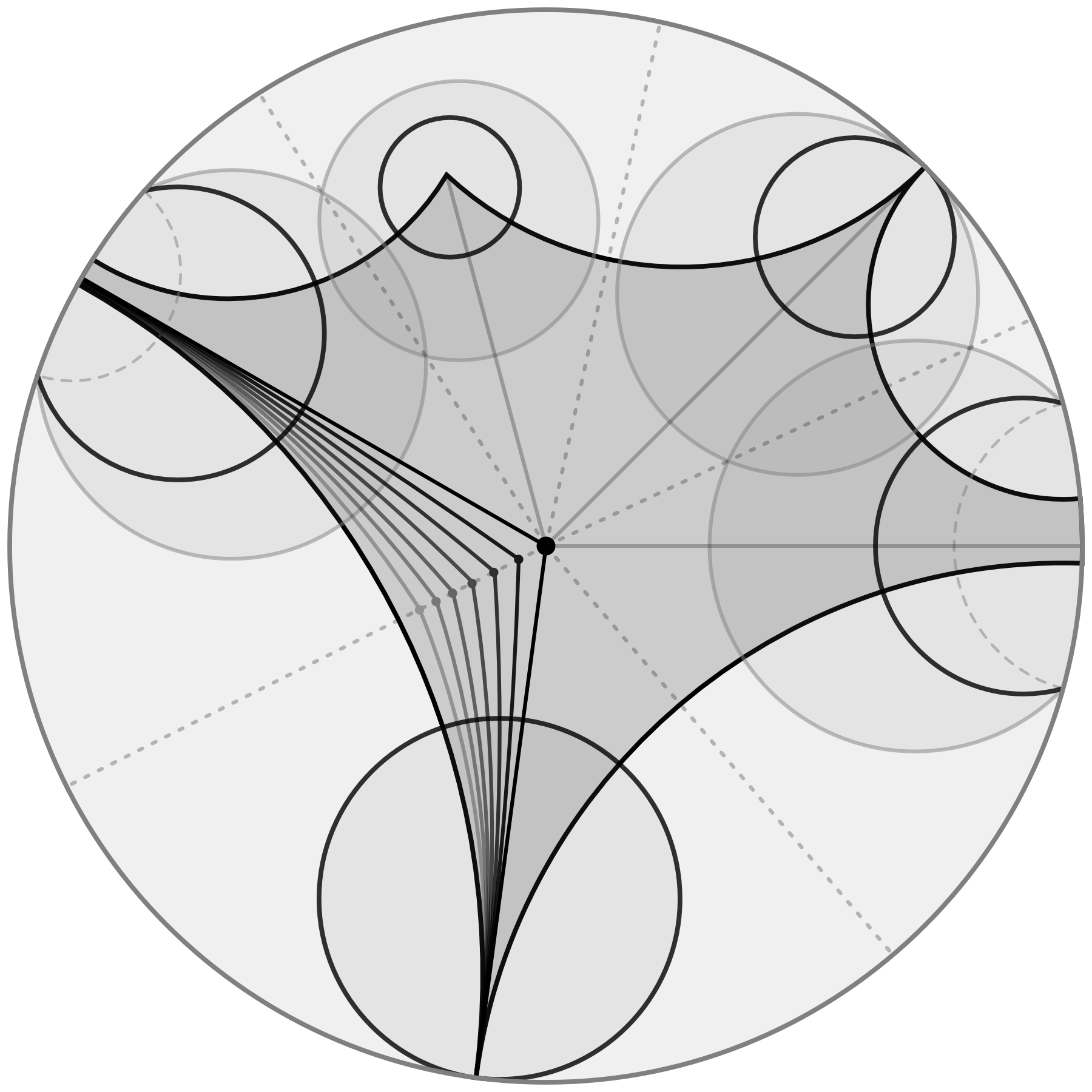}
      }
      \put(210, 0){
         \includegraphics[width=0.4\textwidth]
            {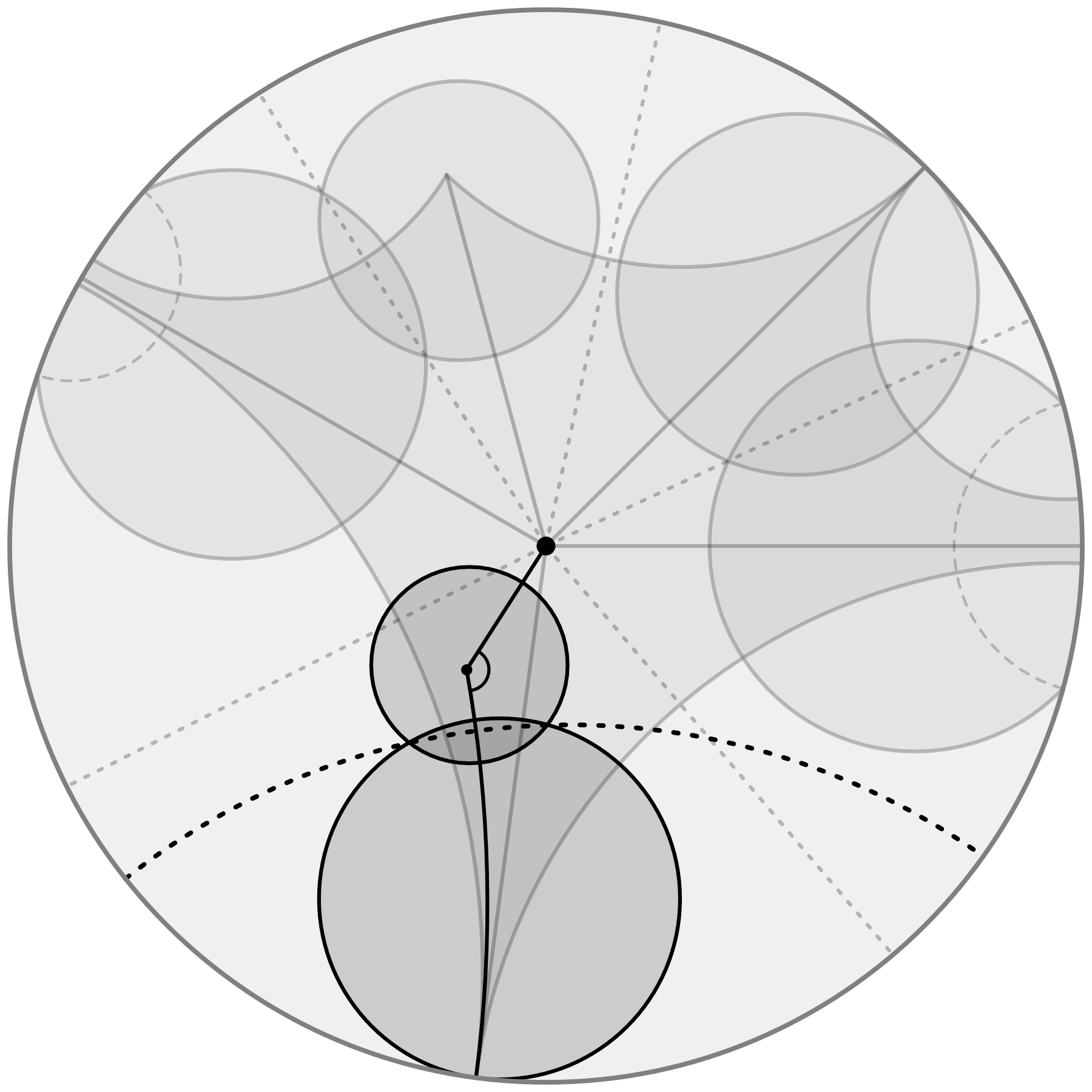}
      }

      \put(96, 76){\small $\tilde{p}$}
      \put(110, 88){\small $P$}
      \put(150, 48){\small $S_1$}
      \put(151, 115){\small $S_2$}
      \put(77, 142){\small $S_3$}
      \put(19, 86){\small $S_4$}
      \put(63, 0){\small $S_5$}

      \put(294, 76){\small $\tilde{p}$}
      \put(272, 59){\small $q$}
      \put(260, 0){\small $C_5$}
   \end{picture}
   \caption{Visualisation of the constructions used in Theorem
      \ref{theorem:weighted_Delaunay_tessellations_revisited}.
      \textsc{Left}: building upon the construction presented in Lemma
      \ref{lemma:laguerre_0_cell}, we can associate a hyperbolic polygon $P$
      (shaded) to a Voronoi $0$-cell. The embedding $\varphi$
      can be extended to an isometric immersion using the indicated homotopy.
      \textsc{Right}: should some cone-point $\Phi(q)$ obstruct
      extending $\varphi$ to all of $\interior(P)$ then
      $\tandist_p(C_{v_n})>\tandist_p(C_{\Phi(q)})$ as $q$ and $\tilde{p}$ lie on
      the same side of the radical line (dashed black) between $C_n$ and $q$.
   }
   \label{fig:delaunay_from_laguerre}
\end{figure}
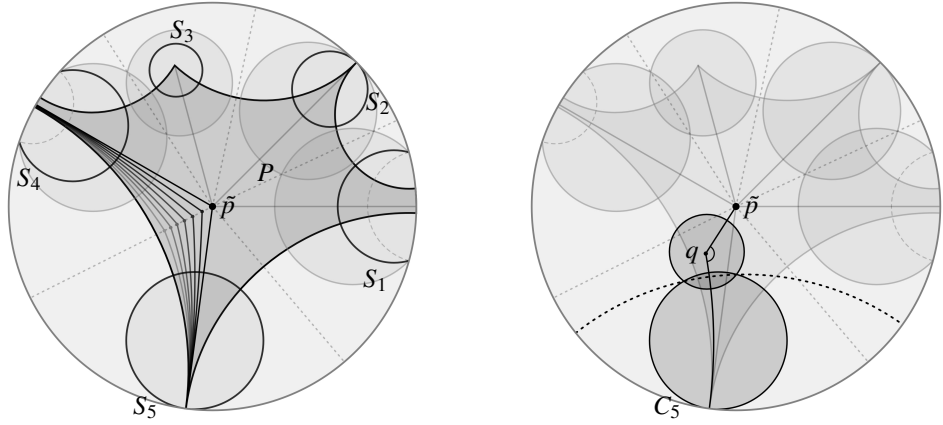

\begin{proof}
   Consider a Voronoi $0$-cell $p\in\surf$. Let $\varphi\colon D\to\surf$ be
   the isometry and $C_1, \dotsc, C_N$ be the hyperbolic cycles defined in Lemma
   \ref{lemma:laguerre_0_cell}. In addition to the $C_n$ we can find hyperbolic
   cycles $S_n$ corresponding to the $S_v$ defined in Lemma
   \ref{lemma:laguerre_2_cell}. By construction their associated discs do not
   intersect. Let $P\coloneq\poly(S_1,\dotsc,S_N)$. We show that the interior of
   $P$ can be isometrically mapped into $\surf$, i.e., it
   defines an (open) Delaunay $2$-cell. The rest of the assertions,
   including properties \ref{item:laguerre_to_delaunay} and
   \ref{item:laguerre_local_delaunay}, follow directly by tracing back the
   definitions made up to this point.

   All cycles $C_n$ have equal tangent distance to $\tilde{p}\coloneq\varphi^{-1}(p)$.
   Hence, we find an $F\in\Sym(2)$ such that $\ip{F}{C_n} = -1$ for all $n=1,\dotsc,N$.
   Indeed, $F = \tilde{p}/\tandist_{\tilde{p}}(C_n)$. Note that
   $F$ defines a properly immersed disc if $|F|^2 > -1$. Furthermore, the $C_n$
   give another decoration of $P$ as $C_n$ and $S_n$ share the same centre.
   So Proposition \ref{prop:associated_polygon} grants that the centres of the
   $C_n$ are exactly the vertices of $P$. Consider the cycles $C_1$ and $C_2$. We
   can homotope the path given by the concatenation of the two rays connecting
   $\tilde{p}$ with $C_1$ and $C_2$, respectively, to the edge between these
   cycles by moving $\tilde{p}$, and with it the rays, along the radical line of
   $C_1$ and $C_2$ (see Figure \ref{fig:delaunay_from_laguerre}, left). Proceeding like
   this for all consecutive cycles $C_n$ and $C_{n+1}$ we can extend $\varphi$ to an
   isometric immersion $\Phi\colon U\to\surf$. Here, $U\subseteq P$ shall be the largest
   set which can be obtained by the described homotopy such that $\Phi$ is still
   isometric. Suppose that $U\neq P$. Then there is a $q\in\partial U$ with
   $q\in\interior(P)$ and $\Phi(q)\in\verts_{-1}$. Let $C_n$ be a cycle whose centre
   has minimal distance to $q$. Then the angle at $q$ between the rays connecting it to
   $\tilde{p}$ and the centre of $C_n$ is $>\pi/2$ (see Figure
   \ref{fig:delaunay_from_laguerre}, right). Thus, by the properness of the decoration
   and Corollary \ref{cor:radical_line_intersection}, we see that
   \[
      \tandist_{\tilde{p}}(C_n)
      \;>\; \frac{\cosh(\dist_{\HH}(q,\tilde{p}))}{\cosh(r_{\Phi(q)})}
      \;>\; \tandist_{p}(C_{\Phi(q)}).
   \]
   This contradicts the minimality of
   $\tandist_{p}(C_{v_n}) = \tandist_{\tilde{p}}(C_n)$. So $U= P$.
   Finally, suppose $\Phi|_{\interior(P)}$ is not injective. Similarly to Lemma
   \ref{lemma:delaunay_n_cells}, we can find a non-trivial hyperbolic motion $M$
   such that $P\cap M(P)\neq\emptyset$ and
   $\Phi|_{P\cap M(P)}=(\Phi\circ M^{-1})|_{P\cap M(P)}$. Since $\Phi|_{\interior(P)}$
   is an isometry on each region belonging to a Voronoi $2$-cell (Lemma
   \ref{lemma:laguerre_2_cell})
   we see that there is no neighbourhood of $\tilde{p}$ over which $\Phi$ is injective.
   But this contradicts the initial assertion that $\Phi|_D$ is an isometry. It
   follows that $\Phi$ is injective over $\interior(P)$.
\end{proof}

This theorem justifies the following generalisation of the notion of weighted
Delaunay tessellation to properly decorated surfaces.

\begin{definition}[weighted Delaunay tessellation, proper decorations]
   For a properly decorated surface $\surf^{\omega}$ the geodesic tessellation
   dual to the weighted Voronoi decomposition of $\surf^{\omega}$ is called
   \define{weighted Delaunay tessellation of $\surf^{\omega}$}.
   A geodesic triangulation refining this tessellation is called a
   \define{weighted Delaunay triangulation of $\surf^{\omega}$}.
\end{definition}

\subsection{The flip algorithm}
\label{sec:flip_algorithm}
Consider a geodesic triangulation $\tri$ of a properly decorated hyperbolic
surface $\surf^{\omega}$. A triangle $\Delta$ of $\tri$ can be lifted to a decorated
triangle in $\HH$. This lift is not unique but any two lifts can be connected
by a hyperbolic motion. Denote the vertex cycles of some lift by $C_1$, $C_2$
and $C_3$. Furthermore, let $F_{\Delta}\in\Sym(2)$ denote their face-vector. The
\define{support function
$H_{\Delta}^{\omega}\colon\HH\setminus\{\ip{X}{F_{\Delta}}=0\}\to\R_{>0}$ of the
decorated triangle $\Delta$} is defined by
\[
   -1
   \;=\; \big\langle\sqrt{\smash[b]{H_{\Delta}^{\omega}(X)}}\,X,\,F_{\Delta}\big\rangle
\]
for all $X\in\HH\setminus\{\ip{X}{F_{\Delta}}=0\}$. We observe that
$\poly(C_1,C_2,C_3)\subset\HH\setminus\{\ip{X}{F_{\Delta}}=0\}$ because
$\ip{C_n}{F_{\Delta}}=-1$. Hence, $H_{\Delta}^{\omega}$ is well defined on
$\Delta\subset\surf$. For two adjacent triangles the
support functions agree on their common edge (see Figure \ref{fig:support_function}).
It follows, that each geodesic triangulation $\tri$ of a decorated hyperbolic surface
$\surf^{\omega}$ induces a continuous \define{support function}
$H_{\tri}^{\omega}\colon\surf\to\R_{>0}$ restricting to $H_{\Delta}^{\omega}$ on each
triangle $\Delta$, respectively.

\begin{figure}[h]
   \begin{picture}(385, 120)
      \put(77, 0){
         \includegraphics[width=0.6\textwidth]{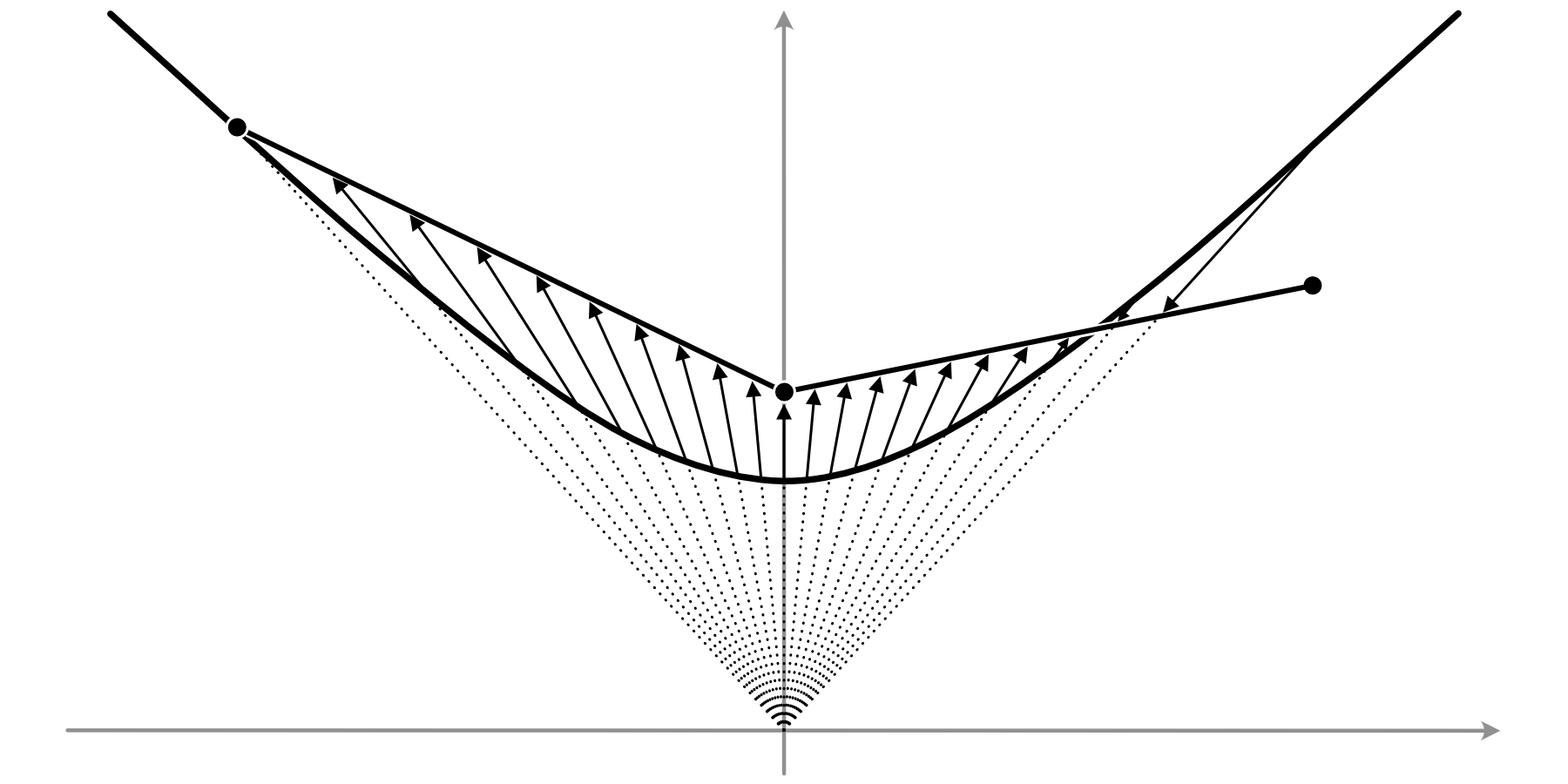}
      }

      \put(223, 67){\small $\Delta$}
      \put(151, 83){\small $\tilde{\Delta}$}
      \put(300, 113){\small $\HH$}
   \end{picture}
   \caption{Decorating a hyperbolic triangle $\Delta$ defines a lift to the
      affine plane $\{\ip{F_{\Delta}}{X}=-1\}\subset\Sym(2)$. The support
      function $H^{\omega}_{\Delta}$ is given by the scaling factors such that
      $\sqrt{\smash[b]{H_{\Delta}^{\omega}(X)}}X\in\{\ip{F_{\Delta}}{X}=-1\}$
      for all $X\in\HH\setminus\{\ip{F_{\Delta}}{X}=0\}$. If two decorated triangles
      $\Delta$ and $\tilde{\Delta}$ share an edge, then their support functions
      agree on the corresponding points.
   }
   \label{fig:support_function}
\end{figure}

\begin{theorem}[flip algorithm]
   \label{theorem:flip_algorithm}
   Let $\surf^{\omega}$ be a properly decorated hyperbolic surface. Start with any
   geodesic triangulation of $\surf$. Then consecutively flipping edges violating
   the strict local Delaunay condition terminates after a finite number of steps.
   The computed triangulation of $\surf$ is a weighted Delaunay triangulation with
   respect to the decoration $\omega$ in the sense of Theorem
   \ref{theorem:weighted_Delaunay_tessellations_revisited}.
\end{theorem}
\begin{proof}
   Let us suppose for a moment that consecutively flipping edges of a geodesic
   triangulation violating the strict local Delaunay condition terminates after a
   finite number of steps. Hence, all edges of the final geodesic triangulation
   satisfying the local Delaunay condition. Therefore, Proposition
   \ref{prop:local_to_global_Delaunay} grants that this triangulation refines
   the weighted Delaunay tessellation of the decorated hyperbolic surface
   $\surf^{\omega}$, i.e., it is a weighted Delaunay triangulation of $\surf^{\omega}$.

   It remains to show that the flip algorithm terminates. Proposition
   \ref{prop:flippable} together with Example \ref{ex:isosceles_triangle}
   guarantees that an edge violating the strict local Delaunay
   condition can always be flipped. Let $\tri$ and $\tilde{\tri}$ be geodesic
   triangulations. Suppose $\tilde{\tri}$ can be obtained from $\tri$ by flipping
   an edge $e\subset\surf$ of $\tri$ which violates the strict local Delaunay
   condition. Locally this is equivalent to changing to the lower convex hull
   of four points in $\Sym(2)$ (compare to the proof of Proposition
   \ref{prop:flippable}). Thus we deduce that
   $H_{\tilde{\tri}}^{\omega}\leq H_{\tri}^{\omega}$. In particular,
   $H_{\tilde{\tri}}^{\omega}(x)<H_{\tri}^{\omega}(x)$ for all $x\in e$. Finally, Lemma
   \ref{lemma:local_edge_length_bound} yields that there is an upper bound for the
   lengths of the edges of $\tri$ only determined by $H_{\tri}^{\omega}$ and the
   vertex radii. Using Lemma \ref{lemma:bounded_geodesics}, it follows that there
   is only a finite number of geodesic triangulations of $\surf$ whose edges satisfy
   this lengths constraint. This implies that the number of geodesic
   triangulations $\tilde{\tri}$ with $H_{\tilde{\tri}}^{\omega}\leq H_{\tri}^{\omega}$
   is finite.
\end{proof}

\begin{corollary}
   \label{corollary:min_support_function}
   Let $\tri$ be a weighted Delaunay triangulation of a properly decorated hyperbolic
   surface $\surf^{\omega}$. Then $H_{\tri}^{\omega}\leq H_{\tilde{\tri}}^{\omega}$
   for any other geodesic triangulation $\tilde{\tri}$ of $\surf$. In particular
   $H_{\tri}^{\omega}=H_{\tilde{\tri}}^{\omega}$ holds iff $\tilde{\tri}$ is
   another weighted Delaunay triangulation of $\surf^{\omega}$.
\end{corollary}

\begin{corollary}[generalised Epstein-Penner convex hull construction]
   \label{cor:convex_hull_construction}
   Let $\surf^{\omega}$ be a properly decorated hyperbolic surface coming from a
   finitely generated, non-elementary Fuchsian group $\Fgroup < \PSL(2;\R)$
   (see Example \ref{example:fuchsian_group}). Since the covering space of
   $\surf$ is $\HH$, we can find for each vertex $v\in\verts$ an orbit
   $\mathcal{C}_v\coloneq\{C_v^g\}_{g\in\Fgroup}\subset\Sym(2)$, each element
   representing the vertex cycle about $v$ (see Proposition
   \ref{prop:hyperbolic_cycles}). Then the boundary of
   $\conv\left(\bigcup_{v\in\verts}\mathcal{C}_v\right)$ exhibits the following
   properties:
   \begin{itemize}
      \item
         it consists of a countable number of codimension-one \enquote{faces}
         each of which is the convex hull of a finite number of points in
         $\bigcup_{v\in\verts}\mathcal{C}_v$,
      \item
         each face lies in an elliptic plane, i.e, its face-vector $F$ satisfies
         $|F|^2<0$,
      \item
         the set of faces is locally finite about each point in
         $\bigcup_{v\in\verts}\mathcal{C}_v$,
      \item
         the set of faces can be partitioned into a finite number of
         $\Fgroup$-invariant subsets,
      \item
         the faces project to the Delaunay $2$-cells of the decorated
         hyperbolic surface.
   \end{itemize}
\end{corollary}

The rest of the section is devoted to proving the technical details needed for
the proof of Theorem \ref{theorem:flip_algorithm}. We begin by analysing the
relation between the function $H_{\tri}^{\omega}$ and the edge-lengths.

\begin{lemma}
   \label{lemma:pencil_argmin}
   Suppose $C_0, C_1\in\Sym(2)$ are two hyperbolic cycles in $\HH$ whose
   corresponding discs do not intersecting. Then $|C_1-C_0|^2>0$ and
   \[
      0 \,<\, \argmin_{\lambda\in\R}|C_0+\lambda(C_1-C_0)|^2 \,<\, 1.
   \]
\end{lemma}
\begin{proof}
   Let $\hat{C}_0$ and $\hat{C}_1$ be the lifts of $C_0$ and $C_1$ to $\Herm(2)$
   as defined in Proposition \ref{prop:hyperbolic_cycles}, respectively.
   From Lemma \ref{lemma:formula_radical_line} follows that $C_0$ and $C_1$
   posses a radical line. By Lemma \ref{lemma:uniqueness_radical_line}, this
   radical line is given by $\hat{C}_1-\hat{C}_0 = C_1-C_0$ since it is contained
   in the pencil spanned by $\hat{C}_0$ and $\hat{C}_1$. Hence $|C_1-C_0|^2>0$ as
   $C_1-C_0$ is a M\"obius-circle. Furthermore, the cycles lie
   on different sides of the radical line, i.e., $\ip{C_0}{C_1-C_0}<0$ and
   $\ip{C_1}{C_1-C_0}>0$, because their discs do not intersect. It follows that
   \[
      |C_1-C_0|^2
      \eq \ip{C_1}{C_1-C_0} \minus \ip{C_0}{C_1-C_0}
      \,>\, -\ip{C_0}{C_1-C_0}.
   \]
   Finally, the expression $|C_0+\lambda(C_1-C_0)|^2$ is quadratic in $\lambda$.
   Thus its minimum point is given by the root of the first derivative, that is,
   \[
      \lambda_{\mathrm{min}}
      \eq -\frac{\ip{C_0}{C_1-C_0}}{|C_1-C_0|^2}.\qedhere
   \]
\end{proof}

\begin{lemma}
   \label{lemma:local_edge_length_bound}
   Let $\Delta\subset\surf$ be a triangle in a geodesic triangulation $\tri$ of a
   properly decorated hyperbolic surface $\surf^{\omega}$. Suppose it is incident to
   the vertices $v_0,v_1,v_2\in\verts$ and define
   \(
      H_{\Delta,\omega}^{\max}
      \coloneq\max\{1,\, \max_{x\in\Delta}H_{\tri}^{\omega}(x)\}.
   \)
   Then the edge-lengths of the edges of $\Delta$ are bounded from above by
   \[
      \max\big\{
         r_{v_m}+r_{v_n}+2\arcosh(H^{\max}_{\Delta,\omega})
         \;:\; m,n\in\{0,1,2\}^2,\; m<n
      \big\}.
   \]
\end{lemma}
\begin{proof}
   Consider two vertices, say $v_0$ and $v_1$. Lift the edge between $v_0$ and $v_1$
   to $\HH$. Denote by $C_0$ and $C_1$ the corresponding lifts of the vertex cycles.
   If the cycles $C_0$ and $C_1$ intersect, then the length of the edge between them
   is less or equal to $r_{v_0}+r_{v_1}$. Now assume otherwise. The previous Lemma
   \ref{lemma:pencil_argmin} shows that
   \(
      h_{01}
      \coloneq-\min_{\lambda\in\R}|C_0+\lambda(C_1-C_0)|^2
      \leq H^{\max}_{\Delta,\omega}.
   \)
   In addition, the pencil spanned by the cycles contains two points. The distance
   of these points bounds the distance of the cycles since each of the discs bounded
   by the cycles contains one of them. One computes that the distance of the points
   is given by $2\arcosh(h_{01})$. The monotonicity of the $\arcosh$-function yields
   the result.
\end{proof}

Finally, we are going to consider the relationship between the support function
$H_{\tri}^{\omega}$, the local Delaunay condition and global Delaunay triangulations.
The following proposition proves the ellipticity part of Corollary
\ref{cor:convex_hull_construction}. This provides us with the means to derive some
technical details about the geometry of support functions (Lemma
\ref{lemma:h_function_monotonicity}) building the core of the proof of Proposition
\ref{prop:local_to_global_Delaunay}.

\begin{proposition}
   \label{prop:elliptic_face_vectors}
   Consider a geodesic triangulation $\tri$ of a properly decorated hyperbolic
   surface $\surf^{\omega}$. Suppose that all edges satisfy the local Delaunay
   condition. Then the face-vector $F_{\Delta}\in\Sym(2)$ of any lift of a
   (decorated) triangle $\Delta$ to $\HH$ satisfies $|F_{\Delta}|^2 < 0$.
\end{proposition}
\begin{proof}
   Suppose otherwise. We are going to construct a geodesic ray
   $\gamma\colon [0, \infty)\to\surf\setminus\verts$ such that
   $H_{\tri}^{\omega}\circ\gamma$ is unbounded. This is a contradiction to the
   existence of $\max_{x\in\surf} H_{\tri}^{\omega}(x)$. There are two cases: either
   $|F_{\Delta}|^2 = 0$ or $|F_{\Delta}|^2 > 0$. We are only elaborating the first
   case. The proof for the other one works similarly.

   Let $\Delta_0$ be a triangle whose lift $\tilde{\Delta}_0$ has a face-vector
   satisfying $|F_{\Delta_0}|^2=0$. Choose a geodesic ray
   $\tilde{\gamma}\colon[0,\infty)\to\HH$
   starting in the interior of $\tilde{\Delta}_0$ which
   limits to the centre of the horocycle corresponding to $F_{\Delta_0}$ and
   intersects the interior of an edge of $\tilde{\Delta}_0$. Maybe after
   perturbing the ray a little bit, it projects to a geodesic ray
   $\gamma\colon[0, \infty)\to\surf\setminus\verts$. Successively
   lift triangles of $\tri$ along $\gamma$ to $\HH$ such that they cover
   $\tilde{\gamma}$ (see Figure \ref{fig:ellipticity_lemma}). Denote the triangles
   by $\Delta_1, \Delta_2, \dotsc$ and the times when $\gamma$ intersects edges
   of $\tri$ by $0\eqcolon s_0 < s_1 < s_2 < \cdots$. Hence,
   $(H_{\tri}^{\omega}\circ\gamma)(s) = (H_{\Delta_n}^{\omega}\circ\tilde{\gamma})(s)$
   for $s_n\leq s\leq s_{n+1}$. By the local Delaunay condition
   (compare to Lemma \ref{lemma:h_function_monotonicity}),
   \(
      (H_{\Delta_m}^{\omega}\circ\tilde{\gamma})(s)
      \geq(H_{\Delta_n}^{\omega}\circ\tilde{\gamma})(s)
   \)
   for $s>s_{n+1}$ and $m>n$. Finally, by construction, there is a $\lambda_s>0$
   for every $s\geq0$ such that
   \[
      (H_{\Delta_0}^{\omega}\circ\tilde{\gamma})(s)
      \;=\; -\big|\sqrt{h_0}\,\tilde{\gamma}(0) + \lambda_s F_{\Delta_0}\big|^2
      \;=\; h_0 + 2\lambda_s\ee^{\delta}.
   \]
   Here $h_0\coloneq (H_{\Delta_0}^{\omega}\circ\tilde{\gamma})(0)$ and $\delta$ is
   the (oriented) distance of $\tilde{\gamma}(0)$ to the horocycle given by
   $F_{\Delta_0}$. Note that $\lambda_s\to\infty$ as $s\to\infty$. Thus,
   $(H_{\Delta_0}^{\omega}\circ\tilde{\gamma})(s) \to \infty$ as $s\to\infty$, too.
\end{proof}

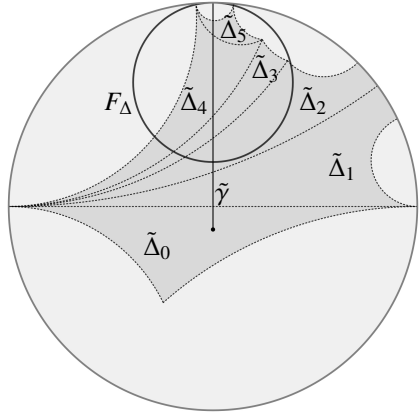
\begin{figure}[h]
   \begin{picture}(385, 155)
      \put(115, 0){
         \includegraphics[width=0.4\textwidth]{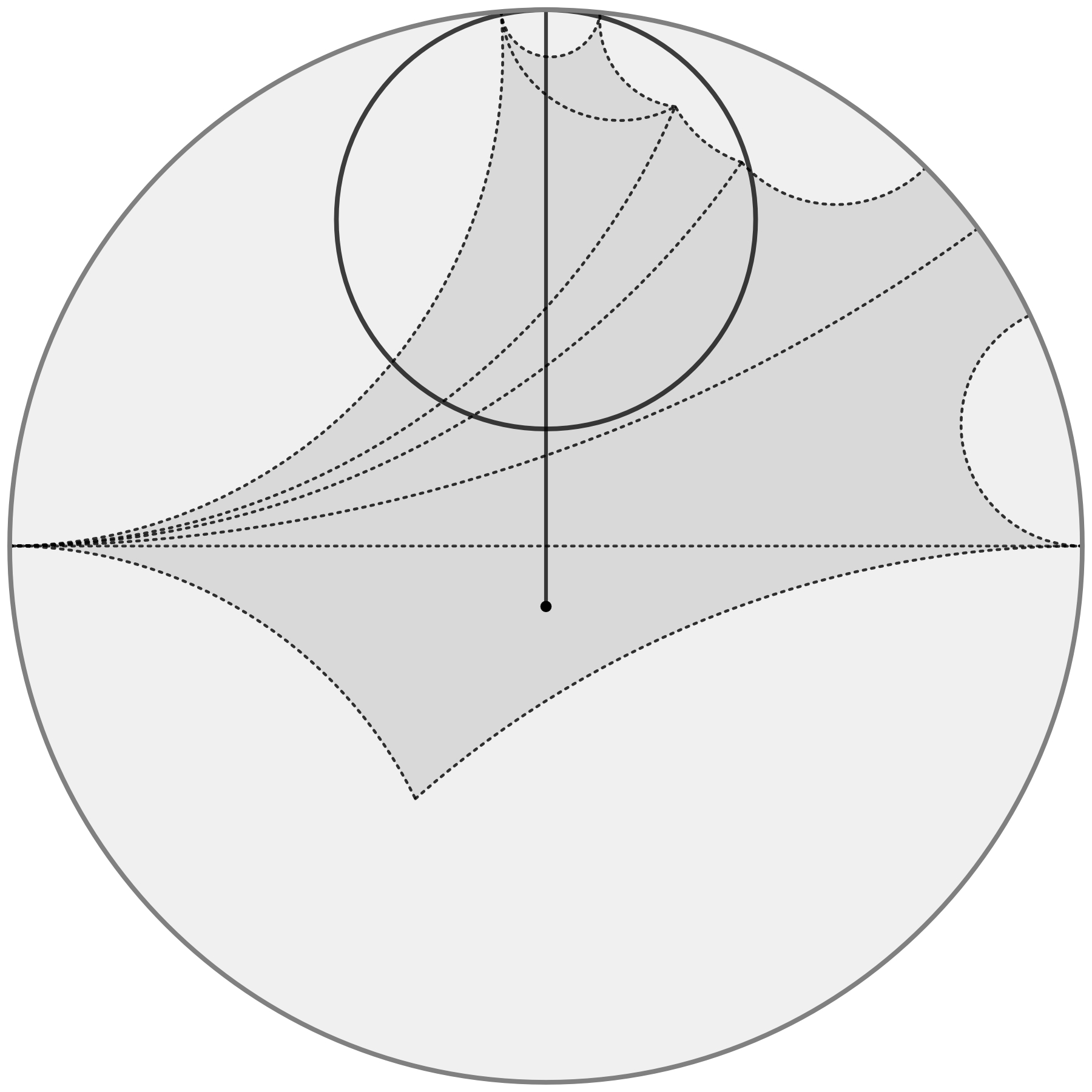}
      }
      \put(155, 115){\small $F_{\Delta}$}
      \put(197, 81){\small $\tilde{\gamma}$}

      \put(170, 60){\small $\tilde{\Delta}_0$}
      \put(240, 90){\small $\tilde{\Delta}_1$}
      \put(229, 115){\small $\tilde{\Delta}_2$}
      \put(211, 127){\small $\tilde{\Delta}_3$}
      \put(184, 115){\small $\tilde{\Delta}_4$}
      \put(199, 142){\small $\tilde{\Delta}_5$}
   \end{picture}
   \caption{
      Sketch of the construction used in Proposition
      \ref{prop:elliptic_face_vectors}.
   }
   \label{fig:ellipticity_lemma}
\end{figure}

\begin{lemma}
   \label{lemma:h_function_monotonicity}
   Let $C$ be a hyperbolic cycle in $\HH$ and $L$ a hyperbolic line orthogonal
   to it. Let $\gamma\colon\R\to L$ be the parametrisation of $L$ by the
   (oriented) distance to the (auxiliary) centre of $C$. If $C$ is a horo- or
   hypercycle choose $\gamma$ such that $\R_{<0}$ is mapped into the disc associated
   to $C$. Consider $F_1,F_2\in\Sym(2)\setminus\lin\{C\}$ with $\ip{C}{F_n} = -1$,
   $n=1,2$, such that $\nu F_1$ and $\nu F_2$ are hyperbolic circles for some $\nu>0$.
   Denote by $H_n\colon\HH\to\R$ the support function induced by $F_n$.
   Furthermore, let $\delta_n$ the distance of the (auxiliary) centre of $C$ to the
   orthogonal projection of the centre of $F_n$ to the line $L$ (see Figure
   \ref{fig:geometry_of_pencils}).

   Then the sign of $H_1\circ\gamma-H_2\circ\gamma$ is
   constant over $\R_{>0}$ if $C$ is a circle or $\R$ otherwise. Furthermore,
   if $(H_1\circ\gamma)(s)>(H_2\circ\gamma)(s)$ for some $s>0$, then
   $\delta_1 > \delta_2$.
\end{lemma}
\begin{proof}
   The first claim about the sign follows from observing that the functions
   $H_n\circ\gamma$ correspond to two intersecting affine lines in $\Sym(2)$.
   Indeed, for each $s\in\R$ there is a $\lambda_s > 0$ such that
   \[
      \sqrt{(H_n\circ\gamma)(s)}\,\gamma(s)
      \eq
      C \plus \lambda_sX_n
   \]
   where $X_n\coloneq\sqrt{(H_n\circ\gamma)(1)}\,\gamma(1)-C$, $n=1,2$,
   (compare to Figure \ref{fig:support_function}).

   Next, note that we can assume $F_1$, $F_2$ and $C$ to represent hyperbolic
   cycles, maybe after rescaling, i.e., considering $\nu F_1$, $\nu F_2$
   and $(1/\nu) C$ for some $\nu>0$. Denote by $\epsilon\in\{-1, 0, 1\}$ the type
   of $C$ and by $R$ its radius. Furthermore, let $r_n>0$ denote the radius of the
   circle corresponding to $F_n$ and by $d_n$ the distance between its centre and
   $L$. Finally, consider an $s>0$ such that $\gamma(s)$ is not contained in the
   discs associated to $F_1$ and $F_2$. Then there are circles centred at $\gamma(s)$
   orthogonally intersecting $F_1$ or $F_2$, respectively. Denote their radii by
   $r^s_1$ and $r^s_2$. Using Lemma \ref{lemma:cos_sin_laws} we see that
   $\cosh(r_n)\tau'_{\epsilon}(R) = \tau'_{\epsilon}(\delta_n)\cosh(d_n)$
   and $\cosh(r_n)\cosh(r_n^s) = \cosh(d_n)\cosh(s-\delta_n)$. Hence
   \[
      \cosh(r_n^s)
      \eq
      \frac{\tau'_{\epsilon}(R)\cosh(s-\delta_n)}{\tau'_{\epsilon}(\delta_n)}.
   \]
   By Lemma \ref{lemma:radius_from_norm}, $(H_n\circ\gamma)(s) = \cosh^{-2}(r_n^s)$.
   Therefore, assuming that $(H_1\circ\gamma)(s) > (H_2\circ\gamma)(s)$ is equivalent to
   \begin{equation}
      \label{eq:height_like_function_explicite}
      \frac{\ee^{\delta_2}+\epsilon\ee^{-\delta_2}}
         {\ee^{\delta_1}+\epsilon\ee^{-\delta_1}}
      \;<\;
      \ee^{\delta_1-\delta_2}
      \frac{1+\ee^{2(\delta_2-s)}}{1+\ee^{2(\delta_1-s)}}.
   \end{equation}
   Here we used that $\tau'_{\epsilon}(x) = (\ee^x+\epsilon\ee^{-x})/2$.
   After taking the limit $s\to\infty$ and rearranging we arrive at
   $\delta_1\geq\delta_2$. But equation \eqref{eq:height_like_function_explicite}
   prohibits equality.
\end{proof}

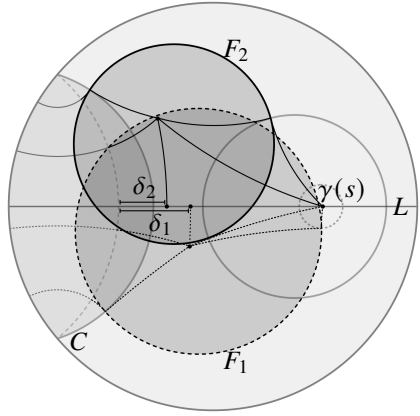
\begin{figure}[h]
   \begin{picture}(385, 155)
      \put(115, 0){
         \includegraphics[width=0.4\textwidth]{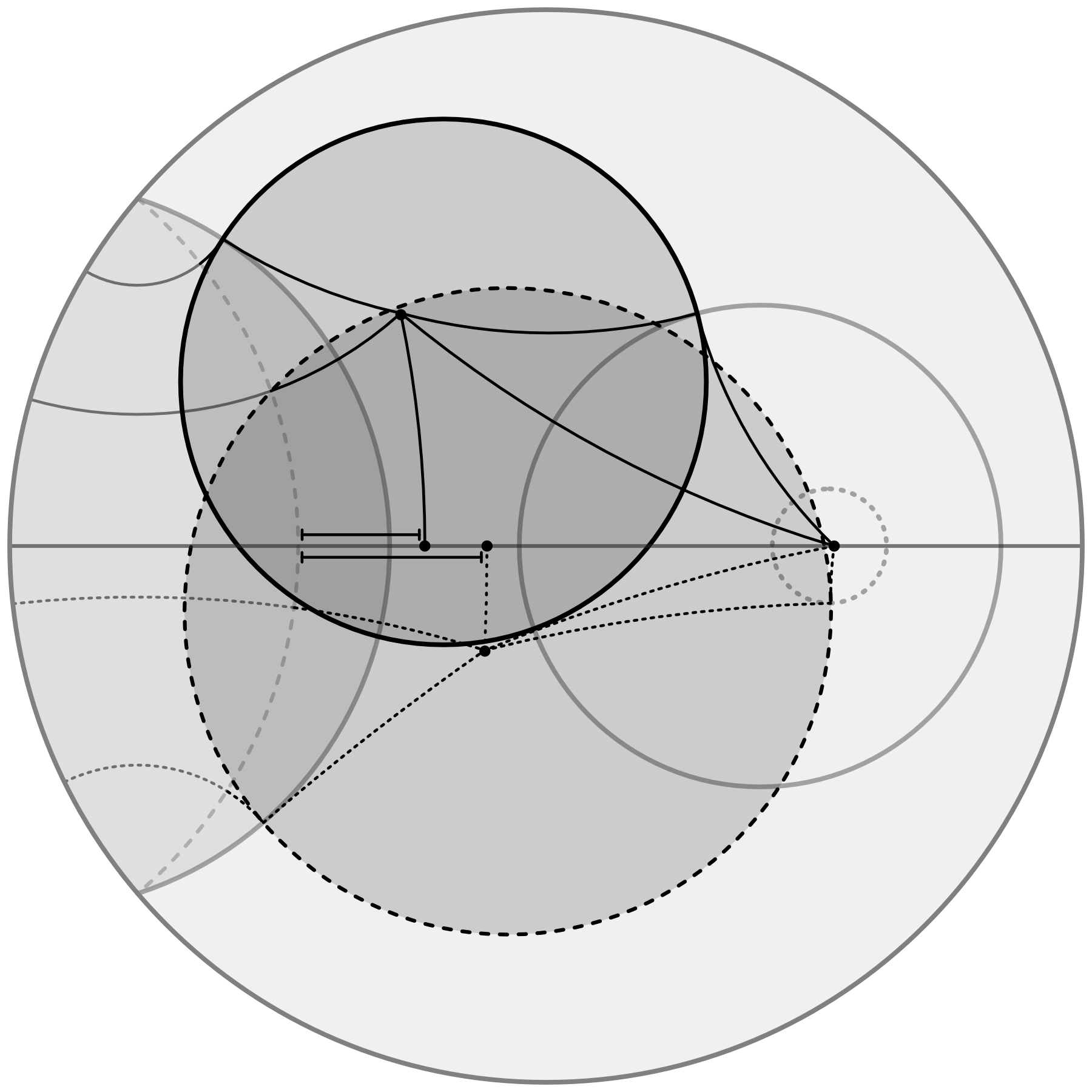}
      }
      \put(142, 24){\small $C$}
      \put(264, 75){\small $L$}
      \put(236, 82){\small $\gamma(s)$}
      \put(200, 135){\small $F_2$}
      \put(200, 17){\small $F_1$}
      \put(166, 83){\small $\delta_2$}
      \put(172, 69){\small $\delta_1$}
   \end{picture}
   \caption{Sketch of the geometric objects considered in Lemma
      \ref{lemma:h_function_monotonicity}. It shows elements of two distinct
      pencils of cycles which both contain a hyperbolic cycle $C$ and their dual
      pencils both contain a common hyperbolic line $L$. The radical line of the first
      pencil intersects $L$ further to the right then the radical line of the second
      pencil, i.e., $\delta_1>\delta_2$, if at some point $\gamma(s)\in L$ the
      radii of the corresponding members of the pencils satisfy $r_2^s > r_1^s$, that
      is, $(H_1\circ\gamma)(s) > (H_2\circ\gamma)(s)$.
   }
   \label{fig:geometry_of_pencils}
\end{figure}

\begin{proposition}
   \label{prop:local_to_global_Delaunay}
   Let $\tri$ be a geodesic triangulation of a properly decorated hyperbolic
   surface $\surf^{\omega}$ whose edges all satisfy the local Delaunay condition.
   Then it refines the weighted Delaunay tessellation of $\surf^{\omega}$.
\end{proposition}
\begin{proof}
   Let $\tilde{\tri}$ be another geodesic triangulation whose edges all
   satisfy the local Delaunay condition. We claim that an edge of $\tri$ is
   either also an edge of $\tilde{\tri}$ or it only intersects edges of
   $\tilde{\tri}$ which satisfy the local Delaunay condition with equality. The
   proposition now follows from the properties of the weighted Delaunay tessellation
   (Theorem \ref{theorem:weighted_Delaunay_tessellations_revisited} item
   \ref{item:laguerre_local_delaunay}).

   We proceed by proving the claim. Let $e$ be an edge of $\tri$ incident to the
   vertices $u,v\in\verts$. Denote by $\gamma\colon I\to e$ the parametrisation
   of $e$ by the (oriented) distance to the centre of $C_u$. If $C_u$ is a horo-
   or hypercycle choose $\gamma$ such that $\R_{<0}$ is mapped into the disc
   associated to $C_u$. Suppose that $e$ intersects $N>0$ edges of $\tilde{\tri}$
   at $(s_n)_{1\leq n\leq N}\subset I$ with
   \[
      \inf I \eqcolon s_0
      \;<\; s_1
      \;<\; \cdots
      \;<\; s_N
      \;<\; s_{N+1}\coloneq\sup I.
   \]
   By definition, $h\coloneq H_{\tri}^{\omega}\circ\gamma$ is a function satisfying
   the assumptions of Lemma \ref{lemma:h_function_monotonicity} whilst
   $\tilde{h}\coloneq H_{\tilde{\tri}}^{\omega}\circ\gamma$ is a continuous
   function agreeing on each $(s_n,s_{n+1})$ with such a function
   $\tilde{h}_n\colon I\to\R$. Towards a contradiction, suppose that there is an
   $x\in I$ such that $\tilde{h}(x)=\tilde{h}_n(x) > h(x)$.
   After possibly changing the role of $u$ and $v$, from Lemma
   \ref{lemma:h_function_monotonicity} follows that $\tilde{h}_n(s) > h(s)$
   for all $s<x$. Using the local Delaunay condition, we see by induction that
   $\tilde{h}_m(s)\geq\tilde{h}_n(s)$ for all $s<s_{m+1}$ and $m<n$.
   Hence, $\tilde{h}_0(s) > h(s)$ for all $s<s_1$. Utilising Lemma
   \ref{lemma:h_function_monotonicity} and the local Delaunay condition
   one more time we get
   \begin{equation}
      \label{eq:deltas}
      \delta \;<\; \tilde{\delta}_0 \;<\; \cdots \;<\; \tilde{\delta}_N
   \end{equation}
   and $\tilde{h}_N(s) > h(s)$ for all $s>s_N$. Considering the
   parametrisation of $e$ with respect of $C_v$ instead of $C_u$, the second
   inequality and Lemma \ref{lemma:h_function_monotonicity} imply
   $\len_e-\delta < \len_e-\tilde{\delta}_N$. Here $\len_e$ is the length of
   the edge $e$. But this is equivalent to $\tilde{\delta}_N < \delta$ contradicting
   inequality \eqref{eq:deltas}. Thus, $h(s) \geq \tilde{h}(s)$ for all $s\in I$.
   Applying the same argument to each edge of $\tilde{\tri}$ which intersects
   $e$ we also get $h(s_n)\leq\tilde{h}(s_n)$, $n=1,\dotsc,N$. It follows
   $h(s) = \tilde{h}(s)$ for all $s\in I$. This is equivalent to the claim.
\end{proof}

\section{The configuration space of decorations}
\label{sec:config_space}
Recall that the weight-vector of a decoration is defined as
$\omega\coloneq(\tau_{\epsilon_v}(r_v))_{v\in\verts}\in\R_{>0}^{\verts}$. We
call $\R_{>0}^{\verts}$ the \define{space of abstract (positive) weights}.
Its subspace $\config_{\surf}\subseteq\R_{>0}^{\verts}$ consisting of all
weight-vectors $\omega\in\R_{>0}^{\verts}$ satisfying the
homogeneous linear constraints
\begin{equation}
   \label{eq:proper_constraints}
   0 \;>\;
   \omega_v \,-\, \tau_{\epsilon_v}(\dist_{\surf}(v,\tilde{v}))\,\omega_{\tilde{v}}
\end{equation}
for all $(v,\tilde{v})\in\verts\times\verts_{-1}$ is the
\define{configuration space of proper decorations of $\surf$}. If
$\verts_{-1}=\emptyset$ all abstract weights can be realised as the (modified) radii
of decoration of $\surf$. Furthermore, it is clear that
$\config_{\surf}=\R_{>0}^{\verts}$. For $\verts_{-1}=\emptyset$ weights in
$\config_{\surf}$ can, in general, only be realised as a decoration of $\surf$
if $\omega_v<\tau_{\epsilon_v}(\dist_{\surf}(v,\tilde{v}))$ and
$1<\omega_{\tilde{v}}$ for all $(v,\tilde{v})\in\verts\times\verts_{-1}$. Still, the
derivations of the previous sections \ref{sec:laguerre_voronoi_decompositions} and
\ref{sec:flip_algorithm} are true for all weights in $\config_{\surf}$, since the key
observation, i.e., Corollary \ref{cor:radical_line_intersection}, only depends on the
constraints \eqref{eq:proper_constraints}. In particular, for any
$\omega\in\config_{\surf}$ Theorem
\ref{theorem:weighted_Delaunay_tessellations_revisited} grants the existence
of a unique weighted Delaunay tessellation with respect to the decoration induced
by $\omega$. Denote it by $\tri_{\surf}^{\omega}$. Let $\tri$ be some geodesic
tessellation of $\surf$. We define
\[
   \config_{\surf}(\tri)
   \;\coloneq\;
   \big\{\omega\in\config_{\surf}
   \,:\, \tri \text{ refines } \tri_{\surf}^{\omega}\big\}.
\]
Note that $\config_{\surf}(\tri)$ is allowed to be empty.

\begin{lemma}
   \label{lemma:config_delaunay_cell}
   Consider the weighted Delaunay tessellation $\tri_{\surf}^{\omega}$ corresponding
   to $\omega\in\config_{\surf}$. Then $\config_{\surf}(\tri_{\surf}^{\omega})$ is
   the intersection of $\config_{\surf}$ with a closed polyhedral cone
   $\mathcal{C}_{\surf}(\tri_{\surf}^{\omega})$. Furthermore, the
   faces of $\mathcal{C}_{\surf}(\tri_{\surf}^{\omega})$ are exactly given by
   those cones $\mathcal{C}_{\surf}(\tri_{\surf}^{\tilde{\omega}})$ defined by
   weighted Delaunay tessellations $\tri_{\surf}^{\tilde{\omega}}$ which
   $\tri_{\surf}^{\omega}$ refines.
\end{lemma}
\begin{proof}
   Let $\tri$ be a geodesic triangulation refining $\tri_{\surf}^{\omega}$. Then all
   edges of $\tri$ satisfy the local Delaunay condition (Theorem
   \ref{theorem:weighted_Delaunay_tessellations_revisited}). We observe
   that the tilts (Definition \ref{def:tilds}), and thus the local Delaunay
   condition (Proposition \ref{prop:local_delaunay_condition}), are linear in the
   $\tau_{\epsilon_v}(r_v)=\omega_v$. It follows that
   $\mathcal{C}_{\surf}(\tri_{\surf}^{\omega})$ is the solution space to a
   finite number of homogeneous linear inequalities and equalities, thus a
   closed polyhedral cone. The second claim about the faces of
   $\mathcal{C}_{\surf}(\tri_{\surf}^{\omega})$ is a reformulation of
   Theorem \ref{theorem:weighted_Delaunay_tessellations_revisited} item
   \ref{item:laguerre_local_delaunay}.
\end{proof}

\begin{corollary}
   \label{lemma:scalability}
   Let $\omega\in\config_{\surf}$. Then
   \(
      \{s\omega:s>0\}\subseteq\config_{\surf}(\tri_{\surf}^{\omega}).
   \)
   In particular, for any geodesic triangulation $\tri$ refining
   $\tri_{\surf}^{\omega}$ we have $H_{\tri}^{s\omega} = (1/s^2)H_{\tri}^{\omega}$.
   Here $H_{\tri}^{\omega}$ and $H_{\tri}^{s\omega}$ are the support functions
   induced by $\omega$ and $s\omega$, respectively.
\end{corollary}

\begin{theorem}[configuration space of proper decorations]
   \label{theorem:configuration_space}
   The configuration space $\config_{\surf}$ of proper decorations of $\surf$
   is a convex connected subset of $\R_{>0}^{\verts}$. There is only a finite number
   of geodesic tessellations $\tri_1,\dotsc,\tri_N$ such that
   $\config_{\surf}(\tri_n)$ are non-empty. In particular,
   $\config_{\surf} = \bigcup_n\config_{\surf}(\tri_n)$.
   In addition, for all $1\leq m<n\leq N$ either
   $\config_{\surf}(\tri_m)\cap\config_{\surf}(\tri_n)=\emptyset$ or there
   is a $k\neq m,n$ such that
   $\config_{\surf}(\tri_m)\cap\config_{\surf}(\tri_n)=\config_{\surf}(\tri_k)$.
\end{theorem}
\begin{proof}
   Everything except the (global) finiteness of the decomposition was covered in
   Lemma \ref{lemma:config_delaunay_cell}. Aiming for a contradiction, suppose that
   there are infinitely many geodesic tessellations $(\tri_n)_{n=1}^{\infty}$ with
   $\config_{\surf}(\tri_n)\neq\emptyset$. In particular, we assume that
   $\tri_m\neq\tri_n$ if $m\neq n$. Denote by
   $\SS^{\verts}\coloneq\{\omega\in\R^{\verts}:\sum_{v\in\verts}\omega_v^2 = 1\}$
   the unit sphere in $\R^{\verts}$. Choose for each $n\geq1$ a
   $\omega^n\in\SS^{\verts}\cap\config_{\surf}(\tri_n)$. Then
   there is a convergent subsequence of $(\omega^n)_{n=1}^{\infty}$ since
   $\SS^{\verts}$ is compact. To simplify notation, we assume that
   $(\omega^n)_{n=1}^{\infty}$ already converges. Let
   $\omega\in\R_{\geq0}^{\verts}\cap\SS^{\verts}$ be its limit point.
   Denote by $\verts_0^0$ and $\verts_1^0$ those vertices $v$ in
   $\verts_0$ or $\verts_1$ with $\omega_v = 0$, respectively.
   Note that by construction $\verts_0^0\cup\verts_1^0\neq\verts$.

   First, assume that $\verts_0^0\cup\verts_1^0=\emptyset$. If
   $\omega\in\config_{\surf}$, it induces a weighted Delaunay tessellation
   by Corollary \ref{lemma:scalability}. This contradicts the local finiteness of the
   decomposition implied by Lemma \ref{lemma:config_delaunay_cell}. Should
   $\omega\in\partial\config_{\surf}$ instead, then
   \(
      \tandist_x(\omega_v^n)
      \coloneq\tau_{\epsilon_v}(\dist_{\surf}(v,x))/\omega_v^n
   \)
   still converges for all $x\in\surf$ and $v\in\verts$ as $n\to\infty$. So the
   global finiteness follows again from the local finiteness. Now assume
   $\verts_0^0\cup\verts_1^0\neq\emptyset$. The idea is to
   show that the (combinatorial) star of each vertex in $\verts_0^0\cup\verts_1^0$
   is constant for large enough $n$. Then we can use the same argument as in the
   first case. To this end, consider the auxiliary sequence
   $(\tilde{\omega}^n)_{n=1}^{\infty}$ with
   \[
      \tilde{\omega}_v^n
      \;\coloneq\;\begin{cases}
         \omega_v^n &,\text{if } v\in\verts_0^0\cup\verts_1^0,\\
         \omega_v &,\text{otherwise}.
      \end{cases}
   \]
   In the weighted Voronoi decomposition dual to the weighted Delaunay tessellation
   $\tri_{\surf}^{\tilde{\omega}^n}$ to each vertex $v$
   corresponds an open Voronoi $2$-cell $P_v^n$ containing it. Denote its
   closure by $\bar{P}_v^n$. The boundary of
   $P_v^n$ is comprised of Voronoi $1$- and $0$-cells. By definition,
   for each (open) $1$-cell $e\subset\partial P_v^n$ there is a unique $u\in\verts$ such
   that $\tandist_x(\tilde{\omega}_u^n)=\tandist_x(\tilde{\omega}_v^n)$ for all
   $x\in e$. Remember the embedded cycle $S_v$ and the map
   $\gamma^v\colon p\mapsto\gamma_p^v(L_p)$ introduced in Lemma
   \ref{lemma:laguerre_2_cell}. The latter maps $S_v$ continuously
   onto $\partial P_v^n$. Using this map, the Voronoi $1$-cells induce
   a decomposition of $S_v$ into segments. Hence, traversing $S_v$ in
   counter-clockwise direction, we can associate to each vertex $v$ and step $n$
   a finite sequence $U_v^n\coloneq(u_1^{v,n},\dotsc,u_{i_{v,n}}^{v,n})$ of vertices
   corresponding to Voronoi $1$-cells in the boundary of $P_v^n$.
   Note that $U_v^n$ is defined up to cyclic permutations and that it determines the
   star of $v$ in $\tri_{\surf}^{\tilde{\omega}^n}$ together with $\gamma^v$. We
   observe that $\tandist_x(\tilde{\omega}_v^n)\to\infty$ as $n\to\infty$ for all
   pairs $(x,v)\in\trunc(\surf)\times(\verts_0^0\cup\verts_1^0)$. So for large enough
   $n$ all faces of $\tri_{\surf}^{\tilde{\omega}^n}$ contain at most one
   vertex from $\verts_0^0\cup\verts_1^0$, counted with multiplicity.
   Moreover the discs $D_v(\tilde{\omega}^n)$ about $v\in\verts_0^0\cup\verts_1^0$
   corresponding to $\tilde{\omega}^n$ exist for large $n$. From these observations
   and the definition of the modified tangent distance
   (Definition \ref{def:modified_tangent_distance}) follows that for each $n$
   and $v\in\verts_0^0\cup\verts_1^0$ there is an $N_v^n>0$ such that
   $\bar{P}_v^m \subset D_v(\tilde{\omega}^n)$ for all $m>N_v^n$. Conversely, there
   is an $M_v^n>0$ such that $D_v(\tilde{\omega}^m)\subset\bar{P}_v^n$ for all
   $m>M_v^n$. It follows that we can find for each $n$ an $N^n>0$ such that
   $\bar{P}_v^m\subset\bar{P}_v^n$ for all $v\in\verts_0^0\cup\verts_1^0$
   and $m>N^n$. Thus $U_v^m$ is a subsequence of $U_v^n$. We conclude that
   there is some $N>0$ such that $U_v^n=U_v^m$ for all
   $v\in\verts_0^0\cup\verts_1^0$ and $n,m>N$.
\end{proof}

\begin{remark}
   This theorem shows that the notion of a partial decoration can be
   extended from hyperbolic cusp surfaces to hyperbolic surfaces of finite type.
   For more information on partial decorations of cusp surfaces see
   \cite{Springborn20}*{\S 5}.
\end{remark}

\begin{corollary}
   \label{cor:simplicial_representation}
   The configuration space of proper decorations $\config_\surf$ can be identified
   up to scaling with the interior of a convex $(|V|-1)$-polytope $\mathcal{P}_{\surf}$
   contained in the standard simplex
   \(
      \Delta^{|V|-1}
      = \conv\{(\delta_{v\tilde{v}})_{\tilde{v}\in\verts}\}_{v\in\verts}.
   \)
   Here $\delta_{v\tilde{v}}$ is the Kronecker-delta. The weight-vectors can be
   recovered using barycentric coordinates with respect to $\Delta^{|V|-1}$.
   In particular, $\mathcal{P}_{\surf} = \Delta^{|V|-1}$ if $\verts_{-1}=\emptyset$.
   Moreover, there is a (finite) simplicial decomposition of $\mathcal{P}_{\surf}$
   such that each facet contains all points which induce the same weighted Delaunay
   decomposition of $\surf$.
\end{corollary}

\begin{figure}[h]
   \includegraphics[width=0.48\textwidth]{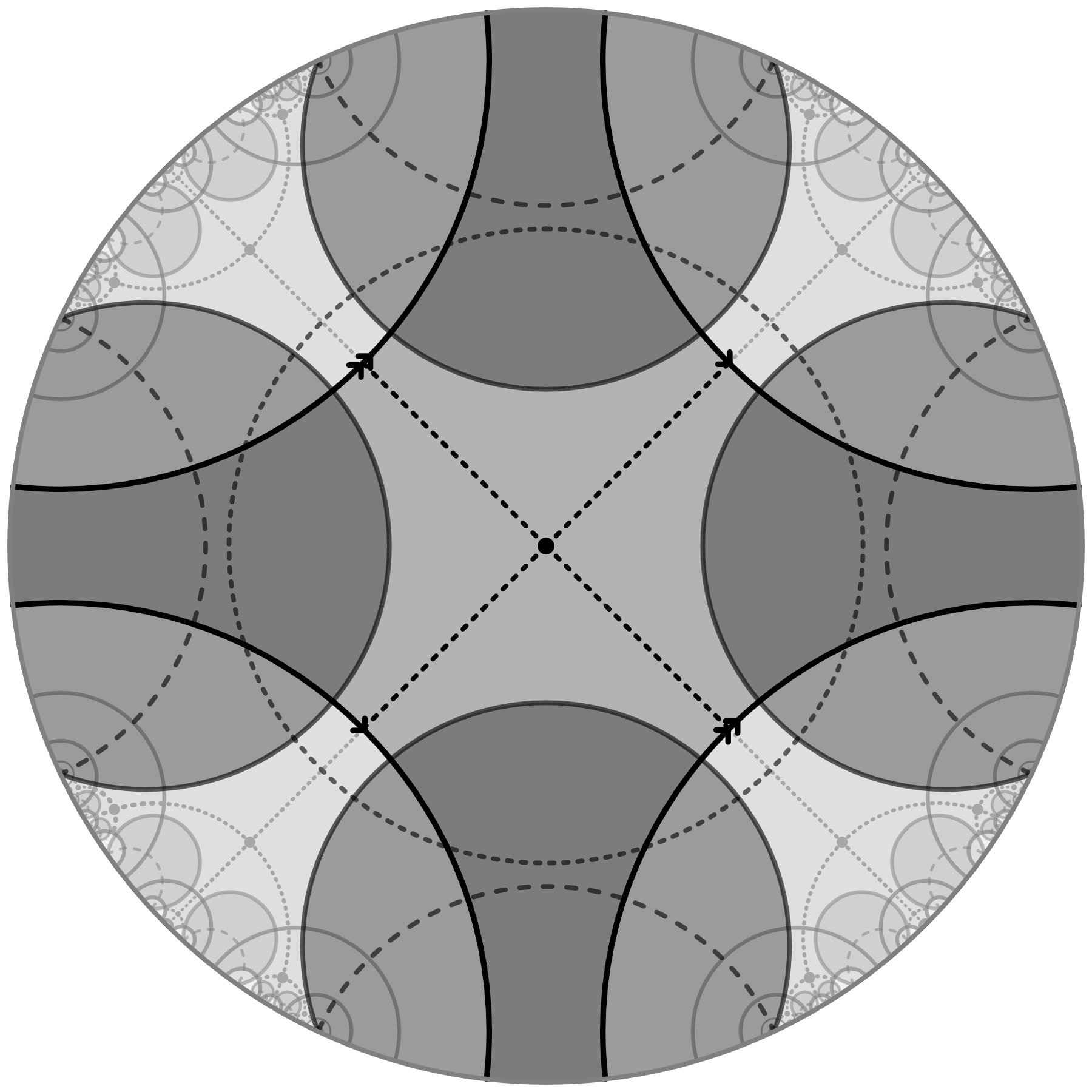}
   \caption{A maximally symmetric hyperbolic quadrilateral (shaded). Its opposite
      edges can be identified (indicated by arrows) to obtain a hyperbolic torus with
      a single flare. By symmetry all vertex cycles of the quadrilateral are orthogonal
      to a common circle (dashed circle). Consequently the corresponding weighted
      Delaunay tessellation (solid edges) of the hyperbolic surface possesses only a
      single $2$-cell and the weighted Voronoi decomposition (dashed edges) only a
      single $0$-cell, respectively.
   }
   \label{fig:torus_with_flare}
\end{figure}

\begin{example}
   It is worth noting that a geodesic tessellation $\tri$ defining a
   $|\verts|$-dimensional set $\config_{\surf}(\tri)\subset\R_{>0}^{\verts}$ is not
   always a triangulation. To see this consider a maximally symmetric hyperbolic
   quadrilateral (see Figure \ref{fig:torus_with_flare}). Necessarily its vertices all
   have the same type and all its edges have the same length. We can glue opposite
   edges of the quadrilateral to obtain a genus-$1$ hyperbolic surface $\surf$
   with a single vertex. By symmetry, for any non self-intersecting decoration of
   the quadrilateral, all vertex cycles posses a single common orthogonal circle. It
   follows that the corresponding weighted Delaunay tessellation $\tri$ possesses only
   a single Delaunay $2$-cell, the interior of the initial quadrilateral. Hence,
   $\config_{\surf}(\tri)=\config_{\surf}$.
\end{example}

\begin{example}
   Let $\Gamma<\PSL(2;\R)$ be a non-elementary free Fuchsian group with
   finite-sided fundamental domain. Denote by $\verts$ the vertex set of the
   hyperbolic surface $\surf\coloneq\HH/\Gamma$. Note that $\verts_{-1}=\emptyset$.
   Extend the vertex set by some $p\in\trunc(\surf)$. Then the weighted Voronoi
   decomposition for the \enquote{undecorated} surface, i.e., $\omega_v=0$ for
   $v\in\verts$ and $\omega_p=1$, exists by Theorem
   \ref{theorem:configuration_space}. Indeed, the open Voronoi $2$-cell
   containing $p$ is given by
   \[
      \big\{x\in\interior(\trunc(\surf)) \;:\; \mathfrak{m}_x(p, 0)=1\big\}.
   \]
   In other words, it is the intersection of the interior of $\trunc(\surf)$
   with the Dirichlet domain of $\surf$ defined by $p$
   (see \cite{Beardon83}*{\S 9.4}).
\end{example}

\begin{example}
   Let $n\geq4$. The $(n, n, n)$-Triangle group is the subgroup $\Fgroup<\PSL(2;\R)$
   of all M\"obius transformations contained in the group generated by reflections
   in the hyperbolic triangle with three angles of $\pi/n$. It is a co-compact
   Fuchsian group. In particular, $\surf\coloneq\HH/\Fgroup$ is homeomorphic to a
   sphere and has three cone-points, i.e., $\verts=\verts_{-1}$ and $|\verts_{-1}|=3$.
   About each cone-point there is a cone-angle of $(2\pi)/n$
   (\cite{Beardon83}*{\S 10.6}). The sphere with three marked points admits four
   combinatorial triangulations. Each of them is a Delaunay triangulation of $\surf$
   for some $\omega\in\config_{\surf}$ (see
   Figure \ref{fig:symmetric_triangle_groups}).
\end{example}

\begin{figure}[h]
   \begin{picture}(385, 180)
      \put(17, 0){
         \includegraphics[width=0.9\textwidth]{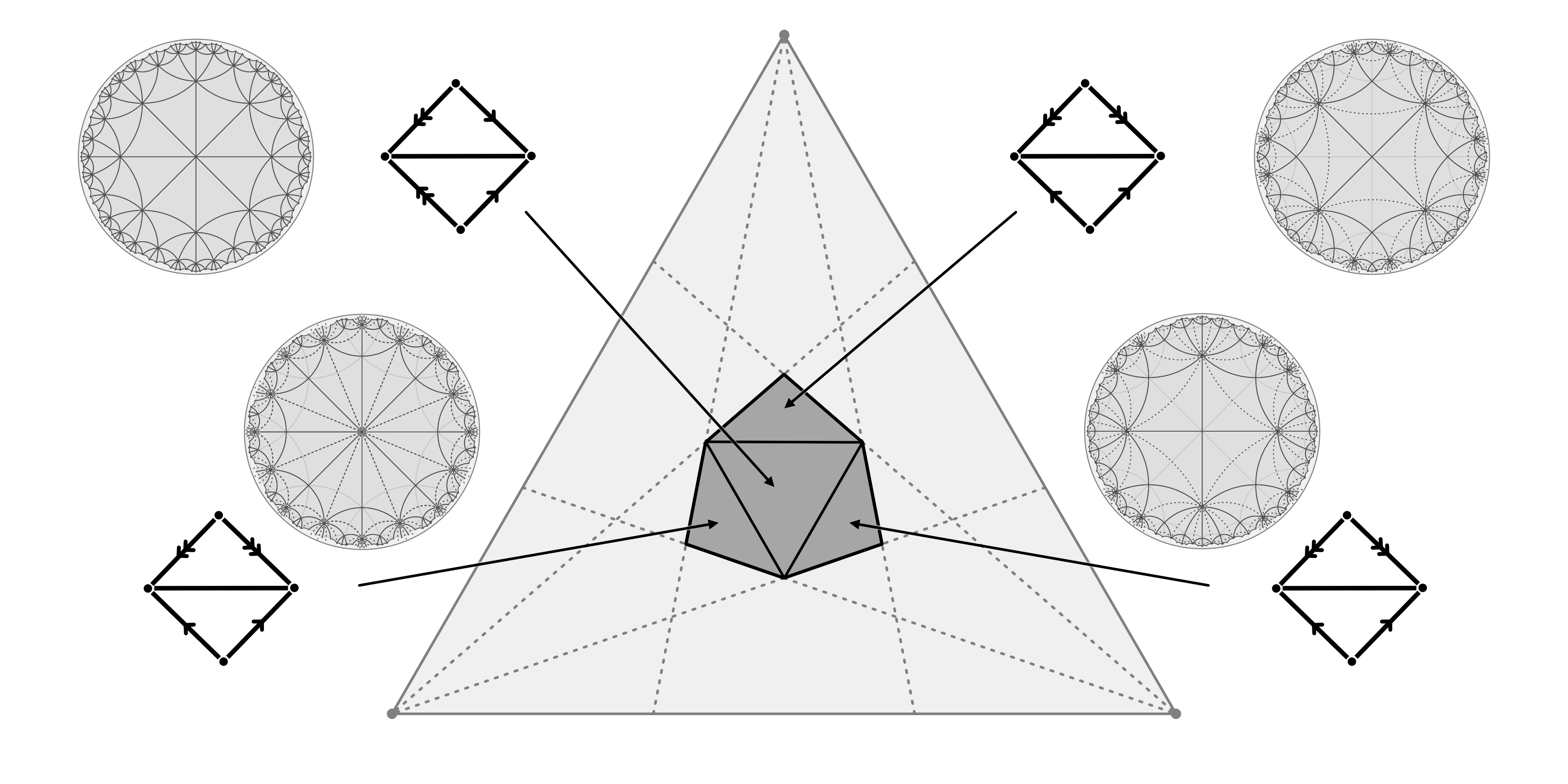}
      }
      \put(119, 161){\footnotesize $v_1$}
      \put(119, 118){\footnotesize $v_1$}
      \put(96, 139){\footnotesize $v_2$}
      \put(141, 139){\footnotesize $v_3$}

      \put(261, 161){\footnotesize $v_2$}
      \put(261, 118){\footnotesize $v_3$}
      \put(238, 139){\footnotesize $v_1$}
      \put(284, 139){\footnotesize $v_1$}

      \put(43, 42){\footnotesize $v_2$}
      \put(88, 42){\footnotesize $v_2$}
      \put(66, 63){\footnotesize $v_1$}
      \put(66, 21){\footnotesize $v_3$}

      \put(297, 42){\footnotesize $v_3$}
      \put(343, 42){\footnotesize $v_3$}
      \put(319, 64){\footnotesize $v_1$}
      \put(319, 21){\footnotesize $v_2$}

      \put(176, 172){\small $[1:0:0]$}
      \put(66, 7){\small $[0:1:0]$}
      \put(285, 7){\small $[0:0:1]$}
   \end{picture}
   \caption{Depicted is the space of abstract weights for the hyperbolic surface
      associated to the $(4, 4, 4)$-Triangle group using the simplicial representation
      (Corollary \ref{cor:simplicial_representation}). Weight-vectors can be recovered
      using barycentric coordinates $[\omega_{v_1}:\omega_{v_2}:\omega_{v_3}]$.
      The configuration space of proper decorations is highlighted. Its simplicial
      decomposition corresponding to weighted Delaunay tessellations of the surface
      is indicated.
   }
   \label{fig:symmetric_triangle_groups}
\end{figure}
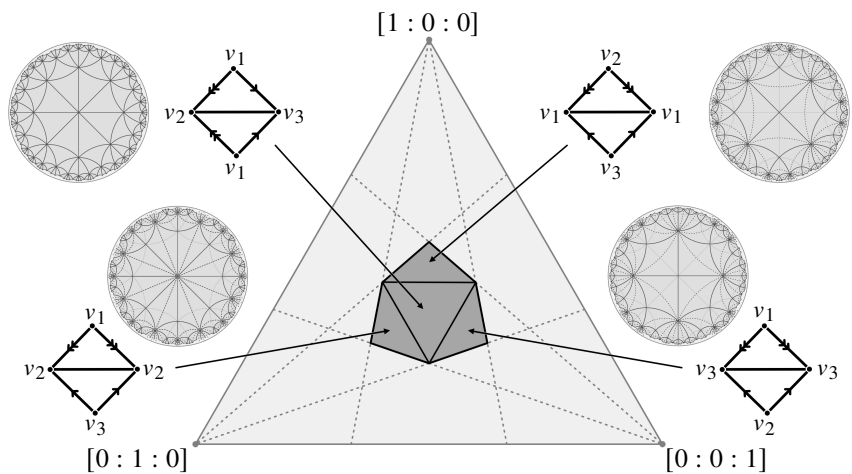


\bib{Akiyoshi00}{article}{
  title={Finiteness of Polyhedral Decompositions of Cusped Hyperbolic Manifolds Obtained by the Epstein-Penner's Method},
  author={Akiyoshi, Hirotaka},
  journal={Proc. Amer. Math. Soc.},
  volume={129},
  number={8},
  pages={2431--2439},
  date={2000},
}

\bib{Beardon83}{book}{
  title={The Geometry of Discrete Groups},
  author={Beardon, Alan F.},
  publisher={Springer},
  place={New York, USA},
  date={1983},
  series={GTM},
  volume={91},
}

\bib{Benz73}{book}{
  title={Vorlesungen \"uber die Geometrie der Algebren},
  author={Benz, Walter},
  publisher={Springer},
  place={Berlin, Germany},
  date={1973},
  series={Grundlehren math. Wiss.},
  volume={197},
}

\bib{Blaschke29}{book}{
  title={Vorlesungen über Differentialgeometrie und geometrische Grundlagen von Einsteins Relativitätstheorie. III. Differentialgeometrie der Kreise und Kugeln},
  author={Blaschke, Wilhelm},
  publisher={Springer},
  place={Berlin, Germany},
  date={1929},
  series={Grundlehren math. Wiss.},
  volume={29},
}

\bib{BobenkoEtAl21}{book}{
  title={Non-Euclidean Laguerre Geometry and Incircular Nets},
  author={Bobenko, Alexander I.},
  author={Lutz, Carl O. R.},
  author={Pottmann, Helmut},
  author={Techter, Jan},
  publisher={Springer},
  place={Cham, Switzerland},
  date={2021},
  series={SpringerBriefs in Mathematics},
}

\bib{BobenkoEtAl15}{article}{
  title={Discrete conformal maps and ideal hyperbolic polyhedra},
  author={Bobenko, Alexander I.},
  author={Pinkall, Ulrich},
  author={Springborn, Boris},
  journal={Geom. Topol.},
  volume={19},
  number={4},
  date={2015},
  pages={2155--2215},
}

\bib{BobenkoSpringborn07}{article}{
  title={A discrete Laplace-Beltrami operator for simplicial surfaces},
  author={Bobenko, Alexander I.},
  author={Springborn, Boris},
  journal={Discrete Comput. Geom.},
  volume={38},
  number={4},
  date={2007},
  pages={740--756},
}

\bib{BogdanovEtAl14}{article}{
  title={Hyperbolic Delaunay Complexes and Voronoi Diagrams Made Practical},
  author={Bogdanov, Mikhail},
  author={Devillers, Olivier},
  author={Teillaud, Monique},
  journal={J. Comput. Geom.},
  volume={5},
  number={1},
  date={2014},
  pages={56--85},
}

\bib{BowditchEpstein88}{article}{
  title={Natural triangulations associated to a surface},
  author={Bowditch, Brian H.},
  author={Epstein, David B. A.},
  journal={Topology},
  volume={27},
  date={1988},
  pages={91--117},
}

\bib{BridsonHaefliger99}{book}{
  title={Metric Spaces of Non-Positive Curvature},
  author={Bridson, Martin R.},
  author={Haefliger, Andr{\'e}},
  series={Grundlehren Math. Wiss.},
  volume={319},
  publisher={Springer},
  place={Berlin, Germany},
  date={1999},
}

\bib{CannonEtAl97}{article}{
  title={Hyperbolic Geometry},
  author={Cannon, James W.},
  author={Floyd, William J.},
  author={Kenyon, Richard},
  author={Parry, Walter},
  booktitle={Flavors of Geometry},
  series={MSRI Publications},
  volume={31},
  publisher={Cambridge University Press},
  date={1997},
}

\bib{Cecil92}{book}{
  title={Lie sphere geometry: with applications to submanifolds},
  author={Cecil, Thomas E.},
  series={Universitext},
  publisher={Springer},
  place={New York, USA},
  date={1992},
}

\bib{CooperEtAl00}{book}{
  title={Three-dimensional orbifolds and cone-manifolds},
  author={Cooper, Daryl},
  author={Hodgson, Craig D.},
  author={Kerckhoff, Steve},
  series={MSJ memoirs},
  volume={5},
  date={2000},
  publisher={Math. Soc. Japan},
  place={Tokyo, Japan},
}

\bib{CooperLong13}{article}{
  title={A generalization of the Epstein-Penner construction to projective manifolds},
  author={Cooper, Daryl},
  author={Long, Darren},
  journal={Proc. Amer. Math. Soc.},
  volume={143},
  number={10},
  date={2013},
  pages={4561--4569},
}

\bib{DeLoeraEtAl10}{book}{
  title={Triangulations},
  author={{De L}oera, Jes{\'u}s A.},
  author={Rambau, J{\"o}rg},
  author={Santos, Francisco},
  series={Algo. Comput. Math.},
  volume={25},
  publisher={Springer},
  place={Berlin},
  date={2010},
}

\bib{Delaunay34}{article}{
  title={Sur la sph{\`e}re vide},
  author={Delaunay, Boris},
  journal={Bull. Acad. Sci URSS, VII. Ser.},
  number={6},
  date={1934},
  pages={793--800},
}

\bib{DespreEtAl20}{article}{
  title={Flipping geometric triangulations on hyperbolic surfaces},
  author={Despr{\'e}, Vincent},
  author={Schlenker, Jean-Marc},
  author={Teillaud, Monique},
  editor={Cabello, Sergio},
  editor={Chen, Danny Z.},
  booktitle={36th International Symposium on Computational Geometry (SoCG 2020)},
  series={Leibniz International Proceedings in Informatics (LIPIcs)},
  publisher={Schloss Dagstuhl -- Leibniz-Zentrum f{\"u}r Informatik},
  address={Dagstuhl, Germany},
  volume={164},
  date={2020},
  pages={35:1--35:16},
}

\bib{Dirichlet50}{article}{
  title={\"Uber die Reduction der positiven quadratischen Formen mit drei unbestimmten ganzen Zahlen},
  author={Dirichlet, G. Lejeune},
  journal={J. Reine Angew. Math.},
  volume={40},
  date={1850},
  pages={209--227},
}

\bib{Edelsbrunner00}{article}{
  title={Triangulations and meshes in computational geometry},
  author={Edelsbrunner, Herbert},
  journal={Acta Numerica},
  volume={9},
  date={2000},
  pages={133--213},
}

\bib{EpsteinPenner88}{article}{
  title={Euclidean decompositions of noncompact hyperbolic manifolds},
  author={Epstein, David B. A.},
  author={Penner, Robert C.},
  journal={J. Differential Geom.},
  volume={27},
  number={1},
  date={1988},
  pages={67--80},
}

\bib{EricksonLin21}{article}{
  title={A toroidal Maxwell–Cremona–Delaunay Correspondence},
  author={Erickson, Jeff},
  author={Lin, Patrick},
  journal={J. Comput. Geom.},
  volume={12},
  number={2},
  date={2021},
  pages={55--85},
}

\bib{Fillastre08}{article}{
  title={Polyhedral hyperbolic metrics on surfaces},
  author={Fillastre, Fran{\c {c}}ois},
  journal={Geom. Dedicata},
  volume={134},
  number={1},
  date={2008},
  pages={177--196},
}

\bib{Fillastre13}{article}{
  title={Fuchsian convex bodies: basics of Brunn-Minkowski theory},
  author={Fillastre, Fran{\c {c}}ois},
  journal={Geom. Funct. Anal.},
  volume={23},
  date={2013},
  pages={295--333},
}

\bib{GelfandEtAl94}{book}{
  title={Discriminants, Resultants, and Multidimensional Determinants},
  author={Gelfand, Israel M.},
  author={Kapranov, Mikhail M.},
  author={Zelevinsky, Andrei V.},
  publisher={Birkh{\"a}user},
  place={Boston, MA},
  date={1994},
}

\bib{GillespieEtAl21}{article}{
  title={Discrete conformal equivalence of polyhedral surfaces},
  author={Gillespie, Mark},
  author={Springborn, Boris},
  author={Crane, Keenan},
  journal={ACM Trans. Graph.},
  volume={40},
  number={4},
  date={2021},
  pages={1--20},
}

\bib{GuEtAl18a}{article}{
  title={A discrete uniformization theorem for polyhedral surfaces},
  author={Gu, Xianfeng},
  author={Luo, Feng},
  author={Sun, Jian},
  author={Wu, Tianqi},
  journal={J. Differential Geom.},
  volume={109},
  number={2},
  date={2018},
  pages={223--256},
}

\bib{GuEtAl18b}{article}{
  title={A discrete uniformization theorem for polyhedral surfaces II},
  author={Gu, Xianfeng},
  author={Guo, Ren},
  author={Luo, Feng},
  author={Sun, Jian},
  author={Wu, Tianqi},
  journal={J. Differential Geom.},
  volume={109},
  number={3},
  date={2018},
  pages={431--466},
}

\bib{Harer88}{article}{
  title={The Cohomology of the Moduli Space of Curves},
  author={Harer, John L.},
  editor={Sernesi, Edoardo},
  booktitle={Theory of moduli},
  series={Lecture notes in mathematics},
  publisher={Springer},
  address={Berlin, Germany},
  volume={1337},
  date={1988},
  pages={138--221},
}

\bib{ImaiEtAl85}{article}{
  title={Voronoi Diagram in the Laguerre Geometry and its Applications},
  author={Imai, Hiroshi},
  author={Iri, Masao},
  author={Murota, Kazuo},
  journal={SIAM J. Comput.},
  volume={14},
  number={1},
  date={1985},
  pages={93--105},
}

\bib{IndermitteEtAl01}{article}{
  title={Voronoi diagrams on piecewise flat surfaces and an application to biological growth},
  author={Indermitte, Claude},
  author={Liebling, Thomas M.},
  author={Troyanov, Marc},
  author={Clemencon, Heinz},
  journal={Theor. Comput. Sci.},
  volume={263},
  number={1-2},
  date={2001},
  pages={263--274},
}

\bib{Izmestiev19}{article}{
  title={Statics and kinematics of frameworks in Euclidean and non-Euclidean geometry},
  booktitle={Eighteen Essays in Non-Euclidean Geometry},
  author={Izmestiev, Ivan},
  editor={Alberge, Vincent},
  editor={Papadopoulos, Athanase},
  series={IRMA lectures in mathematics and theoretical physics},
  number={29},
  publisher={EMS Publishing},
  date={2019},
  pages={191--233},
}

\bib{JoswigEtAl20}{article}{
  title={Secondary fans and secondary polyhedra of punctured Riemann surfaces},
  author={Joswig, Michael},
  author={L{\"o}we, Robert},
  author={Springborn, Boris},
  journal={Exp. Math.},
  volume={29},
  number={4},
  date={2020},
  pages={426--442},
}

\bib{Leibon02}{article}{
  title={Characterizing the Delaunay decompositions of compact hyperbolic surfaces},
  author={Leibon, Gregory},
  journal={Geom. Topol.},
  volume={6},
  number={1},
  date={2002},
  pages={361--391},
}

\bib{NaatanenPenner91}{article}{
  title={The Convex Hull Construction for Compact Surfaces and the Dirichlet Polygon},
  author={N{\"a}{\"a}t{\"a}nen, Marjatta},
  author={Penner, Robert C.},
  journal={Bull. Lond. Math. Soc.},
  volume={23},
  number={6},
  date={1991},
  pages={568--574},
}

\bib{Penner87}{article}{
  title={The Decorated Teichm{\"u}ller Space of Punctured Surfaces},
  author={Penner, Robert C.},
  journal={Commun. Math. Phys.},
  number={113},
  date={1987},
  pages={299--339},
}

\bib{Penner12}{book}{
  title={Decorated Teichm{\"u}ller Theory},
  author={Penner, Robert C.},
  publisher={EMS},
  address={Z{\"u}rich, Switzerland},
  date={2012},
}

\bib{Poincare82}{article}{
  title={Th{\'e}orie des groupes fuchsiens},
  author={Poincar{\'e}, Henri},
  journal={Acta Math.},
  volume={1},
  date={1882},
  pages={1--62},
}

\bib{Prosanov20}{article}{
  title={Ideal polyhedral surfaces in Fuchsian manifolds},
  author={Prosanov, Roman},
  journal={Geom. Dedicata},
  volume={206},
  date={2020},
  pages={151--179},
}

\bib{Ratcliffe94}{book}{
  title={Foundations of Hyperbolic Manifolds},
  author={Ratcliffe, John G.},
  publisher={Springer},
  address={New York},
  date={1994},
  series={GTM},
  volume={149},
}

\bib{Rivin94}{article}{
  title={Euclidean structures on simplicial surfaces and hyperbolic volume},
  author={Rivin, Igor},
  journal={Ann. Math.},
  volume={139},
  number={3},
  date={1994},
  pages={553--580},
}

\bib{SakumaWeeks95}{article}{
  title={The generalized tilt formula},
  author={Sakuma, Makoto},
  author={Weeks, Jeffrey R.},
  journal={Geom. Dedicata},
  volume={55},
  number={2},
  date={1995},
  pages={115--123},
}

\bib{Springborn08}{article}{
  title={A variational principle for weighted Delaunay triangulations and hyperideal polyhedra},
  author={Springborn, Boris},
  journal={J. Diff. Geom.},
  volume={78},
  number={2},
  date={2008},
  pages={333-367},
}

\bib{Springborn20}{article}{
  title={Ideal Hyperbolic Polyhedra and Discrete Uniformization},
  author={Springborn, Boris},
  journal={Discrete Comput. Geom.},
  volume={64},
  date={2020},
  pages={63--108},
}

\bib{Thurston97}{book}{
  title={Three-Dimensional geometry and topology},
  author={Thurston, William P.},
  editor={Levy, Silvio},
  publisher={Princeton University Press},
  place={Princeton, NJ},
  date={1997},
  series={Princeton mathematical series},
  volume={35},
}

\bib{TillmannWong16}{article}{
  title={An algorithm for the Euclidean cell decomposition of a cusped strictly convex projective surface},
  author={Tillmann, Stephan},
  author={Wong, Sampson},
  journal={J. Comput. Geom.},
  volume={7},
  number={1},
  date={2016},
  pages={237--255},
}

\bib{Ushijima99}{article}{
  title={A Canonical Cellular Decomposition of the Teichm{\"u}ller Space of Compact Surfaces with Boundary},
  author={Ushijima, Akira},
  journal={Commun. Math. Phy.},
  volume={201},
  number={2},
  date={1999},
  pages={305--326},
}

\bib{Ushijima02}{article}{
  title={The Tilt Formula for Generalized Simplices in Hyperbolic Space},
  author={Ushijima, Akira},
  journal={Discrete Comput. Geom.},
  volume={28},
  number={1},
  date={2002},
  pages={19--27},
}

\bib{Voronoi08}{article}{
  title={Nouvelles applications des param{\`e}tres continues {\`a} la th{\'e}orie des formes quadratiques, II},
  subtitle={Recherches sur les parall{\'e}lo{\'e}dres primitifs},
  author={Vorono{\"{\i }}, Georges F.},
  journal={J. Reine Angew. Math.},
  volume={134},
  date={1908},
  pages={198--287},
}

\bib{Weeks93}{article}{
  title={Convex hulls and isometries of cusped hyperbolic 3-manifolds},
  author={Weeks, Jeffrey R.},
  journal={Topology Appl.},
  volume={52},
  number={2},
  date={1993},
  pages={127--149},
}

\bib{Yaglom68}{book}{
  title={Complex numbers in geometry},
  author={Yaglom, Isaak M.},
  publisher={Academic Press},
  address={New York, USA},
  date={1968},
}



\begin{bibdiv}
   \begin{biblist}
      \bibselect{references}
   \end{biblist}
\end{bibdiv}
\end{document}